\documentclass[leqno,10pt]{amsart}
\usepackage{amsmath,amscd,amsthm,amsxtra}
\usepackage{epsfig,graphics,color,colortbl}
\usepackage{amssymb,latexsym}
\usepackage{mathabx}
\usepackage{mathrsfs,eucal,upgreek}
\usepackage[poly,all]{xy}
\usepackage{hyperref}
\usepackage{tikz-cd}
\usepackage{arydshln}
\usepackage{ytableau}
\usepackage{adjustbox}
\usepackage{listings}
\usepackage[draft]{todonotes}
\usepackage{datetime} 
\usepackage{tabularx}
\usepackage{lipsum}
\usepackage{dynkin-diagrams}
\usepackage{marginnote}
\usepackage[normalem]{ulem}  % * for strikeout \sout
\usepackage[nocompress]{cite}  % * for sorting the numbered order of multiple citations 
\usepackage{float}
\usepackage{enumerate}
\usepackage{comment}
\usepackage{longtable}

\newtheorem{thm}{\bf Theorem}[section]
\newtheorem{df}[thm]{\bf Definition}
\newtheorem{prop}[thm]{\bf Proposition}
\newtheorem{cor}[thm]{\bf Corollary}
\newtheorem{lem}[thm]{\bf Lemma}
\newtheorem{rem}[thm]{\bf Remark}
\newtheorem{ex}[thm]{\bf Example}

\numberwithin{equation}{section}

\newcommand{\mr}{\mathring}
\newcommand{\ot}{\otimes}
\newcommand{\wtd}{\widetilde}
\newcommand{\ov}{\overline}

\newcommand{\mc}{\mathcal}
\newcommand{\mf}{\mathfrak}
\newcommand{\ms}{\mathscr}

\newcommand{\blue}[1]{{\color{blue}#1}}

\newcommand{\N}{\mathbb{N}}
\newcommand{\Z}{\mathbb{Z}}
\newcommand{\Q}{\mathbb{Q}}
\newcommand{\C}{\mathbb{C}}

\newcommand{\te}{\widetilde{e}}

\newcommand{\bk}{\mathbf{k}}
\newcommand{\g}{\mathfrak{g}}
\newcommand{\bo}{\mathfrak{b}}
\newcommand{\h}{\mathfrak{h}}
\newcommand{\hd}{\mathfrak{h}^*}
\newcommand{\cg}{\mathring{\mathfrak{g}}}
\newcommand{\ch}{\mathring{\mathfrak{h}}}
\newcommand{\chd}{\mathring{\mathfrak{h}}^*}
\newcommand{\qi}{{q_i}}
\newcommand{\cW}{\mathring{W}}
\newcommand{\eW}{\widehat{W}}
\newcommand{\xz}{\mathsf{x}_0}

\newcommand{\al}{{\alpha}}
\newcommand{\E}{\mathsf{E}} % real root vector E
\newcommand{\F}{\mathsf{F}} % real root vector F
\newcommand{\qbn}[2]{{\left[\!\!\begin{array}{c} #1 \\ #2 \end{array}\!\!\right]}}

\newcommand{\sfA}{\mathsf{A}} % Cartan matrix
\newcommand{\sfC}{\mathsf{C}}
\newcommand{\sfP}{\mathsf{P}} % Weight lattice
\newcommand{\sfQ}{\mathsf{Q}} % Root lattice
\newcommand{\sfDel}{\mathsf{\Delta}} % Set of roots
\newcommand{\sfD}{\mathsf{D}} % Symmetrizer
\newcommand{\sfM}{\mathsf{M}} % Z-lattice for Translations
\newcommand{\eM}{\widehat{\sfM}} %Z-lattice for Translations

\newcommand{\I}{I_0}

\newcommand{\La}{\Lambda}
\newcommand{\fin}{\mathsf{fin}}

\newcommand{\spin}[2]{
\begin{tikzpicture}[x=1.2em,y=1.2em,baseline=(current bounding box.center)]
    \coordinate (base) at (0,0);
    \draw (0,2) -- (0.5,2);
    \draw (0.5,0.5) -- (0.5,2);
    \draw (0,0.5) -- (0,2);
    \node at (0.25,1.5) {\scriptsize #1};

    \draw (0,1) -- (0.5,1);
    \draw (0.5,0) -- (0.5,1);
    \draw (0,0) -- (0,1);
    \draw (0,0) -- (0.5,0);
    \node at (0.25,0.5) {\scriptsize #2};
  \end{tikzpicture}%
}

\newcommand{\spinii}[3]{
\begin{tikzpicture}[x=1.2em,y=1.2em,baseline=(current bounding box.center)]
    \draw (0,3) -- (0.5,3);
    \draw (0.5,2) -- (0.5,3);
    \draw (0,2) -- (0,3);
    \node at (0.25,2.5) {\scriptsize #1};

    \draw (0,2) -- (0.5,2);
    \draw (0.5,1) -- (0.5,2);
    \draw (0,1) -- (0,2);
    \node at (0.25,1.5) {\scriptsize #2};

    \draw (0,1) -- (0.5,1);
    \draw (0.5,0) -- (0.5,1);
    \draw (0,0) -- (0,1);
    \draw (0,0) -- (0.5,0);
    \node at (0.25,0.5) {\scriptsize #3};
\end{tikzpicture}%
}

\newcommand{\rootesix}[6]{\substack{\tiny #6 \\ #1 #2 #3 #4 #5}}
\newcommand{\rootf}[4]{\substack{\tiny #1 #2 #3 #4}}

\makeatletter
\newcommand*{\shifttext}[2]{%
  \settowidth{\@tempdima}{#2}%
  \makebox[\@tempdima]{\hspace*{#1}#2}%
}
\makeatother

\usepackage{todonotes}

\title[Minuscule prefundamental representations]
{Unipotent quantum coordinate ring and minuscule prefundamental representations: twisted case}

\author{IL-SEUNG JANG}
\address{Department of Mathematics, Incheon National University, Incheon 22012, Republic of Korea}
\email{ilseungjang@inu.ac.kr}


\begin{document}
\begin{abstract}
We study the prefundamental modules $L_{s,a}^{\pm}$ over the Borel subalgebras of the twisted quantum loop algebras, which are introduced by Wang. A character formula for $L_{s,a}^{\pm}$ is obtained from that for the prefundamental modules over the untwisted quantum loop algebras by applying a character folding map. This allows us to realize minuscule prefundamental modules $L_{s,a}^{\pm}$ for types $A_{2n-1}^{(2)}$ and $D_{n+1}^{(2)}$ in terms of the unipotent quantum coordinate ring associated with the $s$-th level $0$ fundamental weight, where $s = 1$ for type $A_{2n-1}^{(2)}$ and $s = n$ for type $D_{n+1}^{(2)}$. This result is a continuation of the realization of (co)minuscule prefundamental modules established by earlier works \cite{JKP23, JKP25}.
\end{abstract}

\maketitle
\setcounter{tocdepth}{1}
\tableofcontents

\section{Introduction}

Let $\g$ be a finite-dimensional simple Lie algebra over $\C$ with an index set $I_\fin$ of simple roots. Let $U_q(\ms{L}\g)$ (resp.~$U_q(\ms{L}\g^\sigma)$) be the (resp.~twisted) quantum loop algebra associated to $\g$, where $\sigma$ is an automorphism of the Dynkin diagram of $\g$. 
Then the Borel subalgebra $U_q(\bo)$ (resp.~$U_q(\bo^\sigma)$) is the $\bk$-subalgebra of $U_q(\ms{L}\g)$ (resp.~$U_q(\ms{L}\g^\sigma)$) generated by the Chevalley generators $e_i$ and $k_i^{\pm 1}$ for $i \in I$, where $I = I_\fin \sqcup \{ 0 \}$ (resp.~$I = I_\fin^\sigma \sqcup \{ 0 \}$). Here $I_\fin^\sigma$ is the set of orbits of $\sigma$ and we denote by $\ov{i}$ the orbit of $i \in I_\fin$ (cf.~Remark \ref{rem: convention-1; notation for Borel}).

In \cite{HJ}, Hernandez and Jimbo introduced the prefundamental modules over $U_q(\bo)$, denoted by $L_{s,a}^\pm$ ($s \in I_\fin$ and $a \in \C(q)^\times$), in a category $\mc{O}$ of $U_q(\bo)$-modules containing the category of finite-dimensional $U_q(\ms{L}\g)$-modules (of type $1$), motivated by the existence of a limit of normalized $q$-characters of Kirillov--Reshetikhin modules \cite{Na03a, Her06}. Note that $L_{s,a}^\pm$ are infinite-dimensional irreducible $U_q(\bo)$-modules with the highest $\ell$-weight $(\Psi_i(z))_{i \in I_\fin}$ given by
\begin{equation*} %\label{eq: highest l-weight}
	\Psi_i(z) = 
	\begin{cases}
		(1-az)^{\pm 1} & \text{if $i = s$,} \\
		1 & \text{otherwise.}
	\end{cases}
\end{equation*}
It turns out that the category $\mc{O}$ and the prefundamental modules $L_{s,a}^\pm$ play a crucial role in generalizing Baxter’s relations \cite{FH15}, and in studying the Bethe Ansatz equations \cite{FJMM} and the $Q\widetilde{Q}$-systems \cite{FH18}.

Since it follows from \cite{H10a} that there exists a limit of normalized (twisted) $q$-characters of Kirillov--Reshetikhin modules over $U_q(\ms{L}\g^\sigma)$, one may expect analogues of $L_{s,a}^\pm$ over $U_q(\bo^\sigma)$ following \cite{HJ}. In \cite{Wan23}, Wang introduced such $U_q(\bo^\sigma)$-modules, namely the prefundamental modules $L_{s,a}^\pm$ over $U_q(\bo^\sigma)$ in a category $\mc{O}^\sigma$ of $U_q(\bo^\sigma)$-modules (see Definition \ref{df:prefundamental}), to prove the $Q\widetilde{Q}$-system for twisted quantum affine algebras conjectured by Frenkel and Hernandez \cite{FH18}.

In this paper, we realize the prefundamental modules $L_{s,a}^\pm$ over $U_q(\bo^\sigma)$ for types $A_{2n-1}^{(2)}$ and $D_{n+1}^{(2)}$ in terms of the unipotent quantum coordinate ring associated to the $s$-th level $0$ fundamental weight, where $s = 1$ for type $A_{2n-1}^{(2)}$ and $s = n$ for type $D_{n+1}^{(2)}$, as a continuation of the realization of (co)minuscule prefundamental modules established by earlier works \cite{JKP23, JKP25}. 

Let us explain the motivation of this work as follows.
Let $\lambda_s$ be a weight in the dual Cartan subalgebra $\h^\ast$ corresponding to the $s$-th fundamental (co)weight of the underlying simple Lie algebra $\mr{\g}$ corresponding to $I_\fin$ or $I_\fin^\sigma$ (see \eqref{eq: M-hat and la-s}). 
Take $w_s \in W$ such that $t_{-\lambda_s} = w_s \tau_s \in \widehat{W}$ for an affine Dynkin diagram automorphism $\tau_s$, where $W$ (resp.~$\widehat{W}$) is the (resp.~extended) affine Weyl group and $t_{-\lambda_s}$ is the translation by $-\lambda_s$.
Let $U_q^-(w_s)$ denote the unipotent quantum coordinate ring generated by the root vectors associated to a reduced expression of $w_s$.

For untwisted types, the ordinary character of $L_{s,a}^\pm$ is given by 
\begin{equation} \label{eq: char of Lsa - intro}
    {\rm ch}(L_{s,a}^\pm) = \prod_{\beta \in \mr{\sfDel}^+} \left( \frac{1}{1-e^{-{\beta}}} \right)^{[\beta]_s},
\end{equation}
where $\mr{\sfDel}^+$ is the set of positive roots of $\g_\fin$ and $[\beta]_s$ is the coefficient of the simple root $\alpha_s$ in $\beta$. 
This character formula was conjectured in \cite{MY14} and proved for types $A_n^{(1)}$, $B_n^{(1)}$, $C_n^{(1)}$ in \cite{Nao13} (cf.~\cite{Nao14} for type $D_n^{(1)}$) and for type $G_2^{(1)}$ in \cite{LiNa16}, for all untwisted types except for some cases in type $E_8^{(1)}$ in \cite{LeeCH}, and for all untwisted types in \cite{Neg26} (cf.~Remark \ref{rem:character formula}).
In \cite{JKP25}, we verify
\begin{equation*}
    {\rm ch}(U_q^-(w_s)) = \prod_{\beta \in \mr{\sfDel}^+} \left( \frac{1}{1-e^{-{\beta}}} \right)^{[\beta]_s},
\end{equation*}
which was one of the motivations for realizing $L_{s,a}^\pm$ in terms of $U_q^-(w_s)$.

In \cite{H10a}, Hernandez established an isomorphism of Grothendieck rings of finite-dimensional $U_q(\ms{L}\g)$- and $U_q(\ms{L}\g^\sigma)$-modules (of type $1$) in terms of (twisted) $q$-characters, which is induced from a folding map.
Since the (twisted) $q$-character is compatible with the ordinary character, the isomorphism also gives a homomorphism, denoted by $\pi$, of Grothendieck rings of finite-dimensional $U_q(\ms{L}\g)$- and $U_q(\ms{L}\g^\sigma)$-modules (of type $1$), where the map $\pi$ is obtained from a folding map such that $e^{-\alpha_i} \mapsto e^{-\alpha_{\ov{i}}}$. 
In particular, the $q$-characters (resp.~characters) of Kirillov-Reshetikhin modules over $U_q(\ms{L}\g)$ correspond to those over $U_q(\ms{L}\g^\sigma)$ under the isomorphism (resp.~$\pi$).

From the viewpoint of \cite{HJ}, it may be natural to ask whether this framework can be extended to the categories $\mc{O}$ and $\mc{O}^\sigma$ sending the ($q$-)characters of prefundamental modules over $U_q(\bo)$ to those over $U_q(\bo^\sigma)$. This is done by Wang \cite{Wan23}, where we also denote by $\pi$ the homomorphism of Grothendieck rings of the categories $\mc{O}$ and $\mc{O}^\sigma$.

By applying the map $\pi$ to \eqref{eq: char of Lsa - intro}, we prove 
\begin{equation} \label{eq: char of Lsa ; twisted - intro}
    {\rm ch}(L_{\ov{s},a}^\pm) = \prod_{\beta \in \ov{\sfDel_+(t_{-\lambda_{\ov{s}}})}} \left( \frac{1}{1-e^{-\beta}} \right)^{\xi_{\ov{s}}(\beta)\,[\beta]_{\ov{s}}},
\end{equation}
where $\xi_{\ov{s}}(\beta)$ is given in \eqref{eq: def of xis}, $\sfDel_+(t_{-\lambda_{\ov{s}}})$ is the set of positive roots associated to $t_{-\lambda_{\ov{s}}}$, and $\ov{(\,\cdot\,)}$ denotes a projection from $\h^\ast$ onto $\ch^\ast$. Interestingly, it follows from the well-known properties of $t_{-\lambda_{\ov{s}}}$ that we also have
\begin{equation} \label{eq: char of Uqws ; twisted - intro}
    {\rm ch} ( U_q^-(w_{\ov{s}}) ) = \prod_{\beta \in \ov{\sfDel_+(t_{-\lambda_{\ov{s}}})}} \left( \frac{1}{1-e^{-\beta}} \right)^{\xi_{\ov{s}}(\beta)\,[\beta]_{\ov{s}}}.
\end{equation}
By \eqref{eq: char of Lsa ; twisted - intro} and \eqref{eq: char of Uqws ; twisted - intro}, this allows us to realize $L_{\ov{s},a}^\pm$ in terms of $U_q^-(w_{\ov{s}})$ following \cite{JKP23, JKP25}. 
More precisely, we assume that $s = 1$ for type $A_{2n-1}^{(2)}$, and $s = n$ for type $D_{n+1}^{(2)}$.
For a homogeneous element $u \in U_q^-(w_s)$, 
we define 
\begin{equation*}
    k_i(u) = 
    \begin{cases}
        q^{(\alpha_i, {\rm wt}(u))} u & \text{if $i \in I \setminus \{ 0 \}$,} \\
        q^{-(\theta, {\rm wt}(u))}u & \text{if $i = 0$,}
    \end{cases}
    \qquad
    e_i^-(u) = 
    \begin{cases}
        e_i'(u) & \text{if $i \in I \setminus \{ 0 \}$,} \\
        a F^{\rm up}(\theta) u & \text{if $i = 0$,}
    \end{cases}
\end{equation*}
where $e_i'$ is the $q$-derivation on $U_q^-(w_s)$ \eqref{eq:derivation} and $F^{\rm up}(\theta)$ is the dual root vector whose weight is $\theta = \delta - a_0 \alpha_0$.
Similarly, we also define operators $e_i^+ = e_i^-$ for $i \in \I$, while $e_0^+$ is defined by $e_0^+(u) = aq^{-(\theta, {\rm wt}(u))} u F^{\rm up}(\theta)$.
We remark that the operators $e_i^\pm$ are motivated by crystal base theory. See Section \ref{subsec:non-(co)minuscule cases} for further discussion.
\medskip

Now the main result of this paper is as follows.

\begin{thm}[Theorem \ref{thm: main1}, Theorem \ref{thm: main2}] \label{thm:main}
For types $A_{2n-1}^{(2)}$ and $D_{n+1}^{(2)}$, 
the operators $k_i$ and $e_i^-$ (resp.~$e_i^+$) define a representation of $U_q(\bo^\sigma)$ on $U_q^-(w_s)$, denoted by $\rho_-$ (resp.~$\rho_+$), respectively. Furthermore, the $U_q(\bo^\sigma)$-module $U_q^-(w_s)$ with respect to $\rho_-$ (resp.~$\rho_+$) has the same $\ell$-highest weight as that of $L_{s,a\eta_s}^-$ (resp.~$L_{s,-a\eta_s}^+$) for some $\eta_s \in \C(q)^\times$, so they are isomorphic as $U_q(\bo^\sigma)$-modules by \eqref{eq: char of Lsa ; twisted - intro} and \eqref{eq: char of Uqws ; twisted - intro}.
\end{thm}

The above construction reveals a somewhat unexpected phenomenon, namely, the same parameter $\eta_s$ appears for $L_{s,a}^-$ and for $L_{s,a}^+$. 
Note that this phenomenon holds for all (co)minuscule prefundamental modules of untwisted affine type \cite{JKP23, JKP25}. 
On the other hand, there exists a functor $\mc{F}_q$ from a category of $U_{q^{-1}}(\bo)$-modules to that of $U_q(\bo)$-modules such that $\mc{F}(L_{s,a}^{\pm,q^{-1}})$ are also prefundamental modules \cite{Pinet24} (cf.~\cite{HL16b}), where $L_{s,a}^{\pm,q^{-1}}$ are prefundamental modules over $U_{q^{-1}}(\bo)$. In fact, there exists $\gamma \in \mathbb{C}(q)^\times$ such that $\mc{F}_q(L_{s,a}^{\pm,q^{-1}}) \cong L_{s,a\gamma}^\mp$ for all $s \in I_\fin$ and $a \in \mathbb{C}(q)^\times$. 
This similar phenomenon may suggest finding a connection between Theorem \ref{thm:main}, \cite{JKP23, JKP25} and \cite{Pinet24}.
\smallskip

By combining Theorem \ref{thm:main} with \cite{JKP23, JKP25}, we summarize

\begin{enumerate}[(a)]
    \item for all affine types, $U_q^-(w_s)$ admits a module structure over the Borel subalgebra by using the $q$-derivations $e_i'$\, $(i \neq 0)$ \eqref{eq:derivation} together with the $0$-action given by left multiplication by the root vector \eqref{eq: xz} of weight $\alpha_0\!\pmod{\delta}$, where $\delta$ is the null root \eqref{eq: left action}, 
    
    \item for all affine types, we have ${\rm ch}(L_{s,a}^\pm) = {\rm ch}(U_q^-(w_s))$ (Theorem \ref{thm: characters}), 
    
    \item for all (co)minuscule nodes $s$, $U_q^-(w_s)$ is isomorphic to $L_{s,\mp a\eta_s}^\pm$ as a module over the Borel subalgebra, where the $0$-action of $U_q^-(w_s)$ in the case $U_q^-(w_s) \cong L_{s,-a\eta_s}^+$ is replaced by right multiplication by the root vector \eqref{eq: xz}, as in \eqref{eq: right action}.
\end{enumerate}

An advantage of this realization of $L_{s,a}^\pm$ is that one may use nice properties of $U_q^-(w_s)$ to analyze the structure of $L_{s,a}^\pm$. 
Let $U_q^-(w_s)^{\rm up}_{\mc{A}}$ be the $\mc{A}$-form of $U_q^-(w_s)$ generated by $F^{\rm up}({\bf c},\wtd{w_s})$ for ${\bf c}\in \Z_+^\ell$ \eqref{eq:dual PBW monomial}, and let $G^{\rm up}(\infty)$ be the dual canonical basis of $U_q^-(\g)$ (see Section \ref{subsec:DCB} for more details).
Then it is known in \cite{Ki12} that the dual canonical basis is compatible with $U_q^-(w_s)^{\rm up}_{\mc{A}}$, that is, $G^{\rm up}(\infty) \cap U_q^-(w_s)^{\rm up}_{\mc{A}}$ is an $\mc{A}$-basis of $U_q^-(w_s)^{\rm up}_{\mc{A}}$. 
Therefore the prefundamental modules $L_{s,a}^\pm$ also admit such a basis for every (co)minuscule node $s \in I_\fin$ by \cite{JKP23, JKP25} and Theorem \ref{thm:main}.
Hence it would be interesting to clarify the behavior of $G^{\rm up}(\infty)$ in $L_{s,a}^\pm$ with respect to the $U_q(\bo)$-actions, which may be helpful to understand the internal structure of $L_{s,a}^\pm$.
We plan to address this approach in future work.

For example, it is known in \cite{FH15} that there exists a certain grading on positive prefundamental modules with nice properties with respect to the action of Drinfeld generators, which plays a crucial role in \cite{FH15} to prove the main result of that paper. 
One may analyze this grading in the above approach.

Finally, it is worth noting that in \cite{Neg26}, Neguţ constructs $L_{s,a}^\pm$ for all untwisted affine types by using quantum shuffle algebra, and in \cite{Qun26}, Qunell realizes (the dual of) $L_{s,a}^\pm$ in type $A_n^{(1)}$ by a $2$-categorical approach. It may be worthwhile to compare the results in the present paper with the recent works.

The paper is organized as follows.
In Section \ref{sec:affine root systems and Weyl groups}, 
we recall some properties of $t_{-\lambda_s}$,
which play crucial roles throughout the paper.
In Section \ref{sec:quantum groups and dual canonical basis}, 
we introduce and study the (dual) canonical basis of the negative half of the quantum group associated to (affine) Kac--Moody algebras \cite{Lus90b, Kas91a}.
In Section \ref{sec:representation theory of borel algebras}, we review representation theory of the Borel subalgebras following \cite{HJ, Wan23}.
In Section \ref{sec:character formula of Lsa}, we prove that the character of $U_q^-(w_s)$ coincides with that of $L_{r,a}^\pm$ (Theorem \ref{thm: characters}).
In Section \ref{sec:minuscule prefundamental modules}, we prove that for types $A_{2n-1}^{(2)}$ and $D_{n+1}^{(2)}$, there exist $U_q(\bo^\sigma)$-module structures on $U_q^-(w_s)$ isomorphic to $L_{s,a}^\pm$ up to shift of spectral parameter $a$, where $s = 1$ for type $A_{2n-1}^{(2)}$, and $s = n$ for type $D_{n+1}^{(2)}$ (Theorem \ref{thm: main1}, Theorem \ref{thm: main2}).

\section{Affine root systems and Weyl groups} \label{sec:affine root systems and Weyl groups}

\subsection{Affine root systems} \label{subsec: affine root system}
Let $\g = \g(\sfA)$ be the affine Kac--Moody algebra associated with the generalized Cartan matrix $\sfA = (a_{ij})_{i,j \in I}$ of type $X_N^{(r)}$, where $I = \{ 0, 1, \dots, n\}$. We use the numbering of the Dynkin diagrams as in \cite{Kac}, except in the case of type $A_{2n}^{(2)}$ $(n \ge 1)$, where we reverse the numbering of the simple roots (see~Table \ref{tab:dynkins}). Then the diagonal matrix $\sfD = {\rm diag}(d_i)_{i \in I}$ is chosen such that $d_i \in \Z_{\ge 1}$ and $\sfD\sfA$ is symmetric with $\min \{ d_i \, | \, i \in I \} = 1$.
Then we assume that the following data are given.
\begin{itemize}
	\item $\sfP^\vee$ is the dual weight lattice and $\{\, \Lambda_i^\vee \in \sfP^\vee \,\mid\, i \in I \,\}$ is the set of fundamental coweights,
	\item $\Pi^\vee = \left\{\,h_i \in \sfP^\vee \,|\, i\in I\,\right\}$ is the set of simple coroots, 
	\item $\sfP = \left\{\,\lambda\in\h^*\,|\,\lambda(P^\vee) \subset \Z\,\right\}$ is the weight lattice and $\{\, \Lambda_i \in \sfP \,\mid\, i \in I \,\}$ is the set of fundamental weights,
	\item $\Pi = \left\{\,\alpha_i\in\h^*\,|\,i\in I\,\right\}$ is the set of simple roots.
\end{itemize}
Set $\h = \mathbb{C} \otimes_{\,\mathbb{Z}} \sfP^\vee$,
and $\h^*$ is the dual space of $\h$.
Let $\sfP_+$ be the set of dominant integral weights, and let $\sfDel = \sfDel_+ \cup \sfDel_-$ be the set of real roots, where $\sfDel_+$ (resp.~$\sfDel_-$) is the set of positive (resp.~negative) real roots. We denote by $\sfDel_l$ (resp.~$\sfDel_s$) the set of long (resp.~short) roots. Let $\sfQ = \oplus_{i \in I} \Z \alpha_i$ be the root lattice. Then we set $\sfQ_+ = \sum_{i \in I} \Z_+ \alpha_i$, and $\sfQ_- = -\sfQ_+$.
For $\beta = \sum_{i \in I} k_i \alpha_i \in \sfQ$, the number ${\rm ht}(\beta) = \sum_{i \in I} k_i$ is called the height of $\beta$.

Take a nondegenerate symmetric $\C$-valued bilinear form $\left(\,\cdot\,,\,\cdot\,\right)$ on $\h^*$ such that $(\alpha_i, \alpha_i) = 2d_i$ for $i \in I$. Note that
\begin{equation*} %\label{eq:aij}
    a_{ij} = \langle \alpha_j, h_i \rangle = \frac{2(\alpha_i, \alpha_j)}{(\alpha_i, \alpha_i)} \quad (i, j \in I),
\end{equation*}
where $\langle \,\cdot\,,\,\cdot\, \rangle$ is the pairing between $\h$ and $\h^*$, and $\langle \Lambda_i, h_j \rangle  = \delta_{ij}$ and $\langle \alpha_i, \Lambda_j^\vee \rangle = \delta_{ij}$ for $i, j \in I$.
We denote by $\nu : \h \rightarrow \hd$ the $\C$-linear isomorphism from $\h$ to $\hd$ induced from $(\,\cdot\,,\,\cdot\,)$.

Set $\I = I \setminus \{ 0  \}$, $\mr{\Pi}^\vee = \left\{\,h_i \in \sfP^\vee \,|\, i\in \I\,\right\}$, and  $\mr{\Pi} = \left\{\,\alpha_i \in \sfP \,|\, i\in \I\,\right\}$.
Let $\ch$ be the subspace of $\h$ spanned by $\mr{\Pi}^\vee$ over $\C$, which is regarded as the Cartan subalgebra of the underlying finite-dimensional simple Lie subalgebra $\cg$ corresponding to $\mr{\sfA} = (a_{ij})_{i,j \in \I }$.
For $\lambda \in \h^{\ast}$, let $\ov{\lambda}$ denote the orthogonal projection of $\lambda$ onto $\ch^\ast \subset \h^*$ with respect to $(\,\cdot\,,\,\cdot\,)$. For $\lambda^\vee \in \h$, the dual notion $\ov{\lambda^\vee}$ is also defined similarly. 
Let $\mr{\sfQ}$ be the sublattice of $\sfQ$ generated by $\mr{\Pi}$, and let $\mr{\sfP}$ be the projection of $\sfP$ on $\ch^\ast$, where $\ov{\Lambda_i}$ $(i = 1, \dots, n)$ is identified with the $i$-th fundamental weight of $\cg$.
Note that $\ov{\La_i}$\, $(i \neq 0)$ is given by $\La_i = \ov{\La_i} + a_i^\vee \La_0$ \cite[(12.4.3)]{Kac}.
The dual notions $\ch$, $\mr{\sfQ}^\vee$, $\mr{\sfP}^\vee$ and $\ov{\La_i^\vee}$ are defined similarly.
We denote by $\mr{\sfDel}$ the set of roots for $\cg$, and $\mr{\sfDel}_+$ (resp.~$\mr{\sfDel}_-$) the set of positive (resp.~negative) roots for $\cg$. 
Also we denote by $\mr{\sfDel}_s$ (resp.~$\mr{\sfDel}_l$) the set of short (resp.~long) roots of $\cg$.

Let $a_0, a_1, \cdots, a_n$ be the Kac labels, where $a_0 = 1$ for all cases. We denote by $a_i^\vee$ $(i \in I)$ the dual Kac label with respect to $\sfA^t$, where $a_0^\vee = 1$ unless $\sfA$ is of type $A_{2n}^{(2)}$, in which case $a_0^\vee = 2$ (see Table \ref{tab:dynkins}, cf.~\cite[TABLE Aff, pages 54--55]{Kac}).
Let $\delta \in \h^*$ be an imaginary null root and $K \in \h$ a central element given by
\begin{equation*} %\label{eq:del and K}
	\delta = \sum_{i \in I} \,a_i \alpha_i \text{\,\, and \,\,}
	K = \mathsf{t}\,\sum_{i \in I} \,a_i^\vee h_i,
\end{equation*}
where $\mathsf{t} \in \Q$ is chosen to be $\nu(K) = \delta$. 
Note that the factor $\mathsf{t}$ is due to our choice of $\sfD$ (see \cite[Remark 3.2]{JKP25} for more detail). 

Put $\theta = \delta - a_0 \alpha_0$. 
It is known in \cite[Proposition 6.4]{Kac} that $\theta$ is the unique long root in $\mr{\sfDel}$ of maximal height for untwisted affine types, whereas it is the unique short root in $\mr{\sfDel}$ of maximal height for other twisted affine types except for  $A_{2n}^{(2)}$. Note that $\theta = 2(\alpha_1 + \dots + \alpha_n)$ for type $A_{2n}^{(2)}$ under our choice of the numbering of the simple roots, where $\alpha_1 + \dots + \alpha_n$ is a short root and $\theta + \delta$ is a long root.

\subsection{Affine Weyl groups}
Let $W$ be the Weyl group of $\g$, which is defined by the subgroup of ${\rm GL}(\h^*)$ generated by the simple reflections $s_i$ given by $s_i(\lambda) = \lambda - \langle \lambda,\, h_i \rangle\, \alpha_i$ for $i \in I$ and $\lambda \in \h^*$.
Note that the bilinear form $\left(\,\cdot\,,\,\cdot\,\right)$ is $W$-invariant.
Let $R(w)=\{\,(i_1,\dots,i_\ell) \,\,|\,\, i_j \in I \,\,\, \text{and} \,\, w=s_{i_1}\dots s_{i_\ell}\,\}$ be the set of reduced expressions of $w$, where $\ell$ is the length of $w$, denoted by $\ell(w)$.
Denote by $\cW$ the subgroup of $W$ generated by $s_1, \dots, s_n$.

For $\beta \in \hd$, the endomorphism $t_\beta$ of $\hd$ is defined by
\begin{equation} \label{eq:translation}
	t_\beta (\lambda)
	=
	\lambda + \langle \lambda,\, K \rangle \,\beta - \left( (\lambda, \beta) + \frac{1}{2}(\beta,\beta) \langle \lambda, K \rangle \right) \delta,
\end{equation}
which is called the translation by $\beta$ \cite[(6.5.2)]{Kac}.
Note that the translations satisfy the additive property, that is, $t_\alpha t_\beta = t_{\alpha + \beta}$ for $\alpha, \beta \in \h^\ast$, and $t_{w(\beta)} = w t_\beta w^{-1}$ for $w \in \cW$.

If $\g$ is of type $X_N^{(r)}$ with $r > 1$ except for $A_{2n}^{(2)}$, then for $\beta \in \sfDel^{\rm re}$, we set $\gamma_\beta = r$ if $\beta \in \sfDel_l$, and $\gamma_\beta = 1$ otherwise.
If $\g$ is of type $X_n^{(1)}$ or $A_{2n}^{(2)}$, we set $\gamma_\beta = 1$ for any $\beta \in \sfDel^{\rm re}$.
We define
\begin{equation*} %\label{eq:Z-lattices}
	\sfM = \nu\left( \Z(\cW \cdot \theta^\vee) \right)
	\, \text{ and } \,\,\,
	\eM = \left\{ \left. \, \lambda \in \chd \, \right| \, (\lambda, \beta) \in \gamma_\beta \mathbb{Z} \,\, \text{\,for } \beta \in \sfDel^{\rm re}  \, \right\}
\end{equation*}
to be the $\Z$-lattices of $\mr{\h}^\ast$.
Let $T_\sfM$ (resp.~$T_{\eM}$) be the group of translations of $\sfM$ (resp.~$\eM$). Then it is well-known that $W = \cW \ltimes T_\sfM$, and then we define $\eW = \cW \ltimes T_{\eM}$, called the extended affine Weyl group of $\g$, so that $W \subset \eW$.
We remark that $\eM$ is defined such that $\eW$ acts on $\sfDel$.
The description of $\eM$ and its natural basis, denoted by $\{ \lambda_i \, \mid \, i \in \I \}$, are given as follows:
\begin{equation} \label{eq: M-hat and la-s}
\begin{split}
    \eM &= \begin{cases}
        \nu\left( \mr{\sfP}^\vee \right) & \text{if $r = 1$ or $\g$ is of type $A_{2n}^{(2)}$,} \\
        \mr{\sfP} & \text{otherwise,}
    \end{cases}
    \\
    \lambda_i &= \begin{cases}
        \nu\left(\,\ov{\La_i^\vee}\,\right) & \text{if $r = 1$ or $\g$ is of type $A_{2n}^{(2)}$,} \\
        \ov{\La_i} & \text{otherwise,}
    \end{cases}
\end{split}
\end{equation}
where $\mr{\sfP}$ (resp.~$\mr{\sfP}^\vee$) is the (resp.~dual) weight lattice of $\cg$.

\begin{rem}
{\em 
For $\lambda \in \sfP$, we call $\lambda$ a level $0$ weight if $\langle \lambda, K \rangle = 0$. One can check that $\ov{\La_i}$ is a level $0$ weight when $\sfA$ is not of type $A_{2n}^{(2)}$, whereas $\La_i - \La_0 = \ov{\La_i} + \La_0$\, $(i \neq n)$ and $2\La_n - \La_0 = 2\ov{\La_n} + \La_0$ are level $0$ weights when $\sfA$ is of type $A_{2n}^{(2)}$.
}
\end{rem}

Let $\mr{\h}_{\mathbb{R}}^*$ be the $\mathbb{R}$-linear span of $\alpha_1, \dots, \alpha_n$, and let
\begin{equation*}
	C_{\rm af} = \left\{ \left. \, \lambda \in \mr{\h}_{\mathbb{R}}^* \,\, \right| \, (\lambda, \alpha_i) \ge 0 \text{ for $i \in \I$, and } (\lambda, \theta) \le 1 \,\right\},
\end{equation*}
which is called the affine Weyl chamber (also called the fundamental alcove). Let $\mc{T}$ be the subgroup of $\eW$, which stabilizes $C_{\rm af}$. Then it is well-known (e.g.~see \cite{Bou02}) that the subgroup $\mc{T}$ is isomorphic to $\eW /W \simeq T_{\eM} / T_{\sfM}$, and $\eW$ is isomorphic to $W \rtimes \mc{T}$. 
Here each element of $\mc{T}$ is understood as an affine Dynkin diagram automorphism. More explicitly, since the map $\mu \mapsto t_\mu$ identifies
$T_{\eM}$ and $T_{\sfM}$ with the additive groups $\eM$ and $\sfM$,
respectively, we have $T_{\eM}/T_{\sfM} \simeq \eM/\sfM$.
This finite abelian group is given by
\[
T_{\eM}/T_{\sfM}\simeq
\begin{cases}
\Z/(n+1)\Z
    & \text{if $\sfA$ is of type $A_n^{(1)}$,} \\[1mm]
\Z/2\Z
    & \text{if $\sfA$ is of type $B_n^{(1)},\, C_n^{(1)},\, E_7^{(1)},\,
       A_{2n-1}^{(2)},\, D_{n+1}^{(2)}$,} \\[1mm]
\Z/4\Z
    & \text{if $\sfA$ is of type $D_n^{(1)}$ with $n$ odd,} \\[1mm]
\Z/2\Z \times \Z/2\Z
    & \text{if $\sfA$ is of type $D_n^{(1)}$ with $n$ even,} \\[1mm]
\Z/3\Z
    & \text{if $\sfA$ is of type $E_6^{(1)}$,} \\[1mm]
\{0\}
    & \text{if $\sfA$ is of type $E_8^{(1)},\, F_4^{(1)},\, G_2^{(1)},\, A_{2n}^{(2)},\, E_6^{(2)},\, D_4^{(3)}$.}
\end{cases}
\]
Here $\{0\}$ denotes the trivial group.

\begin{ex}
{\em 
For types $A_3^{(1)}$, $C_2^{(1)}$, $B_3^{(1)}$, $E_6^{(1)}$, $D_4^{(1)}$, $D_5^{(1)}$, $A_5^{(2)}$, and $D_3^{(2)}$, the group $T_{\eM}/T_{\sfM}$ is given as follows.
  \[
    \setlength{\arraycolsep}{0pt}
    \renewcommand{\arraystretch}{1.4}
    \!\!\!\!\!
    \begin{array}{cccc}
    % ====================  ROW 1 : box (W,10)(x,-1)  ====================

    % A_3^{(1)} : 4-cycle, Z/4Z cyclic rotation a_0->a_1->a_2->a_3->a_0
    \begin{array}{c}
    {\setlength{\unitlength}{0.09in}\begin{picture}(16,10)(0,-1)
    % nodes
    \put(3,2.1){\makebox(0,0)[c]{$\circ$}}
    \put(8,2.1){\makebox(0,0)[c]{$\circ$}}
    \put(13,2.1){\makebox(0,0)[c]{$\circ$}}
    \put(8,4.5){\makebox(0,0)[c]{$\circ$}}   % a_0 (affine, top)
    % edges
    \put(3.3,2.1){\line(1,0){4.4}}
    \put(8.3,2.1){\line(1,0){4.4}}
    \put(7.7,4.4){\line(-2,-1){4.4}}
    \put(8.3,4.4){\line(2,-1){4.4}}
    % labels
    \put(8,3.5){\makebox(0,0)[c]{\tiny $\alpha_0$}}
    \put(3,1.3){\makebox(0,0)[c]{\tiny $\alpha_1$}}
    \put(8,1.3){\makebox(0,0)[c]{\tiny $\alpha_2$}}
    \put(13,1.3){\makebox(0,0)[c]{\tiny $\alpha_3$}}
    % Z/4Z : 4-cycle
    {\color{red}
    \put(3.50,2.34){\makebox(0,0)[c]{\rotatebox{206}{$\succ$}}}  % a_0->a_1
    \put(7.45,2.10){\makebox(0,0)[c]{$\succ$}}                   % a_1->a_2
    \put(12.45,2.10){\makebox(0,0)[c]{$\succ$}}                  % a_2->a_3
    \put(8.50,4.26){\makebox(0,0)[c]{\rotatebox{154}{$\succ$}}}  % a_3->a_0
    }
    \end{picture}} \\
    A_3^{(1)},\,\,\Z/4\Z
    \end{array}
    &
    % C_2^{(1)} : path a_0 == a_1 == a_2 (arrows inward), Z/2Z swap a_0<->a_2
    \begin{array}{c}
    {\setlength{\unitlength}{0.09in}\begin{picture}(13,10)(0,-1)
    \put(3,2.1){\makebox(0,0)[c]{$\circ$}}
    \put(6.5,2.1){\makebox(0,0)[c]{$\circ$}}
    \put(10,2.1){\makebox(0,0)[c]{$\circ$}}
    % double edges
    \put(3.3,2){\line(1,0){2.9}}\put(3.3,2.2){\line(1,0){2.9}}
    \put(6.8,2){\line(1,0){2.9}}\put(6.8,2.2){\line(1,0){2.9}}
    \put(4,1.7){\Large $\succ$}
    \put(7.7,1.7){\Large $\prec$}
    \put(3,1.3){\makebox(0,0)[c]{\tiny $\alpha_0$}}
    \put(6.5,1.3){\makebox(0,0)[c]{\tiny $\alpha_1$}}
    \put(10,1.3){\makebox(0,0)[c]{\tiny $\alpha_2$}}
    % Z/2Z : swap a_0 <-> a_2 (top arc)
    {\color{red}
    \qbezier(3.42,2.45)(6.5,5)(9.58,2.45)
    \put(3.42,2.45){\makebox(0,0)[c]{\rotatebox{220}{$\succ$}}}
    \put(9.58,2.45){\makebox(0,0)[c]{\rotatebox{-40}{$\succ$}}}
    }
    \end{picture}} \\
    C_2^{(1)},\,\,\Z/2\Z
    \end{array}
    &
    % B_3^{(1)} : fork (a_0,a_1)-a_2 == a_3, Z/2Z swap a_0<->a_1
    \begin{array}{c}
    {\setlength{\unitlength}{0.09in}\begin{picture}(14,10)(-1,-1)
    \put(3,0.7){\makebox(0,0)[c]{$\circ$}}    % a_1
    \put(3,3.3){\makebox(0,0)[c]{$\circ$}}    % a_0
    \put(6.5,2){\makebox(0,0)[c]{$\circ$}}    % a_2 (fork point)
    \put(10,2){\makebox(0,0)[c]{$\circ$}}     % a_3
    % fork + double edge
    \put(6.2,1.7){\line(-3,-1){2.8}}
    \put(6.2,2.3){\line(-3,1){2.8}}
    \put(6.8,1.9){\line(1,0){2.9}}\put(6.8,2.1){\line(1,0){2.9}}
    \put(7.5,1.6){\Large $\succ$}
    \put(3,2.5){\makebox(0,0)[c]{\tiny $\alpha_0$}}
    \put(3,0){\makebox(0,0)[c]{\tiny $\alpha_1$}}
    \put(6.5,1){\makebox(0,0)[c]{\tiny $\alpha_2$}}
    \put(10,1){\makebox(0,0)[c]{\tiny $\alpha_3$}}
    % Z/2Z : swap a_0 <-> a_1 (left arc)
    {\color{red}
    \qbezier(2.51,3.05)(0.5,2)(2.51,0.95)
    \put(2.51,3.05){\makebox(0,0)[c]{\rotatebox{28}{$\succ$}}}
    \put(2.51,0.95){\makebox(0,0)[c]{\rotatebox{-28}{$\succ$}}}
    }
    \end{picture}} \\
    B_3^{(1)},\,\, \Z/2\Z
    \end{array}
    &
    % E_6^{(1)} : T-shape + extended affine node, Z/3Z rotates 3 branches
    %   outer cycle a_0->a_1->a_6->a_0, inner cycle a_2->a_3->a_5->a_2
    \begin{array}{c}
    {\setlength{\unitlength}{0.09in}\begin{picture}(16,10)(0,-1)
    \put(2,2){\makebox(0,0)[c]{$\circ$}}      % a_1
    \put(5,2){\makebox(0,0)[c]{$\circ$}}      % a_3
    \put(8,2){\makebox(0,0)[c]{$\circ$}}      % a_4 (fixed point)
    \put(11,2){\makebox(0,0)[c]{$\circ$}}     % a_5
    \put(14,2){\makebox(0,0)[c]{$\circ$}}     % a_6
    \put(8,5){\makebox(0,0)[c]{$\circ$}}      % a_2
    \put(8,8){\makebox(0,0)[c]{$\circ$}}      % a_0 (affine)
    \put(2.3,2){\line(1,0){2.4}}
    \put(5.3,2){\line(1,0){2.4}}
    \put(8.3,2){\line(1,0){2.4}}
    \put(11.3,2){\line(1,0){2.4}}
    \put(8,2.3){\line(0,1){2.4}}
    \put(8,5.3){\line(0,1){2.4}}
    \put(2,1.2){\makebox(0,0)[c]{\tiny $\alpha_1$}}
    \put(5,1.2){\makebox(0,0)[c]{\tiny $\alpha_3$}}
    \put(8,1.2){\makebox(0,0)[c]{\tiny $\alpha_4$}}
    \put(11,1.2){\makebox(0,0)[c]{\tiny $\alpha_5$}}
    \put(14,1.2){\makebox(0,0)[c]{\tiny $\alpha_6$}}
    \put(8.8,5){\makebox(0,0)[c]{\tiny $\alpha_2$}}
    \put(8.8,8){\makebox(0,0)[c]{\tiny $\alpha_0$}}
    {\color{red}
    % outer cycle a_0 -> a_1 -> a_6 -> a_0
    \qbezier(7.48,7.84)(1.77,6.07)(1.97,2.55)
    \put(1.97,2.55){\makebox(0,0)[c]{\rotatebox{273}{$\succ$}}}
    \qbezier(2.43,1.66)(8,-2.80)(13.57,1.66)
    \put(13.57,1.66){\makebox(0,0)[c]{\rotatebox{39}{$\succ$}}}
    \qbezier(14.03,2.55)(14.23,6.07)(8.53,7.84)
    \put(8.53,7.84){\makebox(0,0)[c]{\rotatebox{163}{$\succ$}}}
    % inner cycle a_2 -> a_3 -> a_5 -> a_2
    \qbezier(7.48,4.84)(4.89,4.04)(4.97,2.55)
    \put(4.97,2.55){\makebox(0,0)[c]{\rotatebox{273}{$\succ$}}}
    \qbezier(5.39,1.61)(8,-1)(10.61,1.61)
    \put(10.61,1.61){\makebox(0,0)[c]{\rotatebox{45}{$\succ$}}}
    \qbezier(11.03,2.55)(11.11,4.04)(8.53,4.84)
    \put(8.53,4.84){\makebox(0,0)[c]{\rotatebox{163}{$\succ$}}}
    }
    \end{picture}} \\
    E_6^{(1)},\,\,\Z/3\Z
    \end{array}
    \\[3em]

    % ====================  ROW 2 : box (W,8)(x,-1.5)  ====================

    % D_4^{(1)} : star (center a_2 + 4 outer), (Z/2)^2 = two commuting swaps
    %   red  : vertical swaps  (a_0<->a_1, a_3<->a_4)
    %   blue : horizontal swaps (a_0<->a_3, a_1<->a_4)
    \begin{array}{c}
    {\setlength{\unitlength}{0.09in}\begin{picture}(15,8)(-1,-1.5)
    \put(3,0.7){\makebox(0,0)[c]{$\circ$}}    % a_1
    \put(3,3.3){\makebox(0,0)[c]{$\circ$}}    % a_0
    \put(7,2){\makebox(0,0)[c]{$\circ$}}      % a_2 (fixed)
    \put(11,3.3){\makebox(0,0)[c]{$\circ$}}   % a_3
    \put(11,0.7){\makebox(0,0)[c]{$\circ$}}   % a_4
    \put(6.7,1.7){\line(-3,-1){3.4}}
    \put(6.7,2.3){\line(-3,1){3.4}}
    \put(7.3,2.3){\line(3,1){3.4}}
    \put(7.3,1.7){\line(3,-1){3.4}}
    \put(3,2.5){\makebox(0,0)[c]{\tiny $\alpha_0$}}
    \put(3,0){\makebox(0,0)[c]{\tiny $\alpha_1$}}
    \put(7,1){\makebox(0,0)[c]{\tiny $\alpha_2$}}
    \put(11,2.5){\makebox(0,0)[c]{\tiny $\alpha_3$}}
    \put(11,0){\makebox(0,0)[c]{\tiny $\alpha_4$}}
    % red : vertical swaps
    {\color{red}
    \qbezier(2.557,3.070)(0.5,2)(2.557,0.930)
    \put(2.557,3.070){\makebox(0,0)[c]{\rotatebox[origin=c]{27}{$\succ$}}}
    \put(2.557,0.930){\makebox(0,0)[c]{\rotatebox[origin=c]{-27}{$\succ$}}}
    \qbezier(11.443,3.070)(13.5,2)(11.443,0.930)
    \put(11.443,3.070){\makebox(0,0)[c]{\rotatebox[origin=c]{153}{$\succ$}}}
    \put(11.443,0.930){\makebox(0,0)[c]{\rotatebox[origin=c]{207}{$\succ$}}}
    }
    % blue : horizontal swaps
    {\color{blue}
    \qbezier(3.047,3.524)(7,5.3)(10.153,3.424)
    \put(3.047,3.524){\makebox(0,0)[c]{\rotatebox[origin=c]{207}{$\succ$}}}
    \put(10.153,3.424){\makebox(0,0)[c]{\rotatebox[origin=c]{-27}{$\succ$}}}
    \qbezier(3.047,0.476)(7,-1.3)(10.153,0.476)
    \put(3.047,0.476){\makebox(0,0)[c]{\rotatebox[origin=c]{153}{$\succ$}}}
    \put(10.153,0.476){\makebox(0,0)[c]{\rotatebox[origin=c]{27}{$\succ$}}}
    }
    \end{picture}} \\
    D_4^{(1)},\,\,\Z/2\Z \times \Z/2\Z
    \end{array}
    &
    % D_5^{(1)} : two forks joined by single edge
    %   Z/4Z generator : outer 4-cycle a_0->a_4->a_1->a_5->a_0  +  swap a_2<->a_3
    \begin{array}{c}
    {\setlength{\unitlength}{0.09in}\begin{picture}(15,8)(0,-1.5)
    \put(3,0.7){\makebox(0,0)[c]{$\circ$}}    % a_1
    \put(3,3.3){\makebox(0,0)[c]{$\circ$}}    % a_0
    \put(6,2){\makebox(0,0)[c]{$\circ$}}      % a_2
    \put(9,2){\makebox(0,0)[c]{$\circ$}}      % a_3
    \put(12,3.3){\makebox(0,0)[c]{$\circ$}}   % a_4
    \put(12,0.7){\makebox(0,0)[c]{$\circ$}}   % a_5
    \put(5.7,1.7){\line(-2,-1){2.3}}
    \put(5.7,2.3){\line(-2,1){2.3}}
    \put(6.3,2){\line(1,0){2.4}}
    \put(9.3,2.3){\line(2,1){2.3}}
    \put(9.3,1.7){\line(2,-1){2.3}}
    \put(3,2.5){\makebox(0,0)[c]{\tiny $\alpha_0$}}
    \put(3,0){\makebox(0,0)[c]{\tiny $\alpha_1$}}
    \put(6,1){\makebox(0,0)[c]{\tiny $\alpha_2$}}
    \put(9,1){\makebox(0,0)[c]{\tiny $\alpha_3$}}
    \put(12,2.5){\makebox(0,0)[c]{\tiny $\alpha_4$}}
    \put(12,0){\makebox(0,0)[c]{\tiny $\alpha_5$}}
    {\color{red}
    % outer 4-cycle
    \qbezier(3.50,3.52)(7.5,5.3)(11.50,3.52)                  % a_0 -> a_4
    \put(11.50,3.52){\makebox(0,0)[c]{\rotatebox{-24}{$\succ$}}}
    \qbezier(11.72,3.15)(7.5,4.88)(3.40,1.07)                 % a_4 -> a_1
    \put(3.40,1.07){\makebox(0,0)[c]{\rotatebox{223}{$\succ$}}}
    \qbezier(3.50,0.48)(7.5,-1.3)(11.50,0.48)                 % a_1 -> a_5
    \put(11.50,0.48){\makebox(0,0)[c]{\rotatebox{24}{$\succ$}}}
    \qbezier(11.72,0.98)(7.5,-0.88)(3.40,2.93)                % a_5 -> a_0
    \put(3.40,2.93){\makebox(0,0)[c]{\rotatebox{137}{$\succ$}}}
    % center swap a_2 <-> a_3
    \qbezier(6.42,2.36)(7.5,3.3)(8.58,2.36)
    \put(6.42,2.36){\makebox(0,0)[c]{\rotatebox{221}{$\succ$}}}
    \put(8.58,2.36){\makebox(0,0)[c]{\rotatebox{-41}{$\succ$}}}
    }
    \end{picture}} \\
    D_5^{(1)},\,\, \Z/4\Z
    \end{array}
    &
    % A_5^{(2)} : same shape as B_3^{(1)} but double edge arrow reversed
    %   Z/2Z swap a_0 <-> a_1
    \begin{array}{c}
    {\setlength{\unitlength}{0.09in}\begin{picture}(11,8)(-1,-1.5)
    \put(3,0.7){\makebox(0,0)[c]{$\circ$}}
    \put(3,3.3){\makebox(0,0)[c]{$\circ$}}
    \put(6.5,2){\makebox(0,0)[c]{$\circ$}}
    \put(10,2){\makebox(0,0)[c]{$\circ$}}
    \put(6.2,1.7){\line(-3,-1){2.8}}
    \put(6.2,2.3){\line(-3,1){2.8}}
    \put(6.8,1.9){\line(1,0){2.9}}\put(6.8,2.1){\line(1,0){2.9}}
    \put(7.5,1.55){\Large $\prec$}            % reversed vs B_3^{(1)}
    \put(3,2.5){\makebox(0,0)[c]{\tiny $\alpha_0$}}
    \put(3,0){\makebox(0,0)[c]{\tiny $\alpha_1$}}
    \put(6.5,1){\makebox(0,0)[c]{\tiny $\alpha_2$}}
    \put(10,1){\makebox(0,0)[c]{\tiny $\alpha_3$}}
    {\color{red}
    \qbezier(2.51,3.05)(0.5,2)(2.51,0.95)
    \put(2.51,3.05){\makebox(0,0)[c]{\rotatebox{28}{$\succ$}}}
    \put(2.51,0.95){\makebox(0,0)[c]{\rotatebox{-28}{$\succ$}}}
    }
    \end{picture}} \\
    A_5^{(2)},\,\, \Z/2\Z
    \end{array}
    &
    % D_3^{(2)} : same shape as C_2^{(1)} but double edge arrows reversed (outward)
    %   Z/2Z swap a_0 <-> a_2
    \begin{array}{c}
    {\setlength{\unitlength}{0.09in}\begin{picture}(13,8)(0,-1.5)
    \put(3,2.1){\makebox(0,0)[c]{$\circ$}}
    \put(6.5,2.1){\makebox(0,0)[c]{$\circ$}}
    \put(10,2.1){\makebox(0,0)[c]{$\circ$}}
    \put(3.3,2){\line(1,0){2.9}}\put(3.3,2.2){\line(1,0){2.9}}
    \put(6.8,2){\line(1,0){2.9}}\put(6.8,2.2){\line(1,0){2.9}}
    \put(4,1.7){\Large $\prec$}               % reversed vs C_2^{(1)}
    \put(7.7,1.7){\Large $\succ$}
    \put(3,1.3){\makebox(0,0)[c]{\tiny $\alpha_0$}}
    \put(6.5,1.3){\makebox(0,0)[c]{\tiny $\alpha_1$}}
    \put(10,1.3){\makebox(0,0)[c]{\tiny $\alpha_2$}}
    {\color{red}
    \qbezier(3.42,2.45)(6.5,5)(9.58,2.45)
    \put(3.42,2.45){\makebox(0,0)[c]{\rotatebox{220}{$\succ$}}}
    \put(9.58,2.45){\makebox(0,0)[c]{\rotatebox{-40}{$\succ$}}}
    }
    \end{picture}} \\
    D_3^{(2)},\,\, \Z/2\Z
    \end{array}
    \end{array}
  \]
where the blue and red arrows in the affine Dynkin diagram for type $D_4^{(1)}$ denote the generators of $\Z/2\Z \times \Z/2\Z$.
}
\end{ex}

For $\hat{w} \in \eW$, one can write
\begin{equation*}
	\hat{w} = w \tau,
\end{equation*}
where $w \in W$ and $\tau \in \mc{T}$. The length of $\hat{w} \in \eW$ is defined by $\ell(w)$, which is also denoted by $\ell(\hat{w})$.
Note that $\mc{T} = \{ \, \hat{w} \in \eW \, | \, \ell(\hat{w}) = 0 \, \}$.
We define
\begin{equation*}
    \sfDel_+(\hat{w})=\sfDel_+\cap \left(\hat{w}\sfDel_-\right) \subset \sfDel_+.
\end{equation*}
Note that $\ell(\hat{w})  = |\sfDel_+(\hat{w})|$ \cite{Bou02, Hum90}.
Take $w_i \in W$ such that $t_{-\lambda_i} = w_i \tau_i \in \widehat{W}$ for an affine Dynkin diagram automorphism $\tau_i$.
For $i \in I \setminus \{0\}$, we collect some well-known properties of $t_{-\lambda_i}$ as follows.

\begin{prop} \label{prop: properties of fundamental translations}
\begin{enumerate}[\em (1)]
    \item If $\sfA$ is not of type $A_{2n}^{(2)}$, then we have 
        \begin{equation*}
            \sfDel_+(t_{-\lambda_i}) = \left\{ \alpha + k \gamma_\alpha \delta  \, \left| \, \alpha \in \mr{\sfDel}_+,\, 0 \le k < \frac{1}{\gamma_\alpha}(\lambda_i, \alpha) \right.  \right\},
        \end{equation*}
        where $\mr{\sfDel}_+$ is the set of positive roots for $\mr{\g}$. 
        \smallskip

    \item If $\sfA$ is of type $A_{2n}^{(2)}$, then we have 
        \begin{align*}
            \sfDel_+(t_{-\lambda_i}) &=  \left\{ \alpha + k \delta  \, \left| \, \alpha \in \mr{\sfDel}_+,\, 0 \le k < (\lambda_i, \alpha) \right.  \right\}
            \\
            & \qquad \bigcup
            \left\{ 2\alpha + (2k+1) \delta   \, \left| \, \alpha \in (\mr{\sfDel}_+)_s,\, 0 \le k < (\lambda_i, \alpha) \right.  \right\},
        \end{align*}
        where $(\mr{\sfDel}_+)_s$ is the set of positive short roots for $\mr{\g}$.
        \smallskip

    \item For $k \in I$, we have
    \begin{equation*}
		\ell(s_k t_{-\lambda_i}) =
		\begin{cases}
			\ell(t_{-\lambda_i}) + 1 & \text{if $k \neq i$,} \\
			\ell(t_{-\lambda_i}) - 1 & \text{if $k = i$,}
		\end{cases}
		\qquad
		\ell(t_{-\lambda_i} s_k) =
		\begin{cases}
			\ell(t_{-\lambda_i}) + 1 & \text{if $k \neq 0$,} \\
			\ell(t_{-\lambda_i}) - 1 & \text{if $k = 0$.}
		\end{cases}
    \end{equation*}

    \item If $t_{-\lambda_i} = w_i \tau_i$ for $w_i \in W$ and $\tau_i \in\mc{T}$, then $w_i$ has a reduced expression $s_{i_1}\dots s_{i_\ell}$ such that $s_{i_1} = s_i$ and $s_{i_\ell} = s_{\tau_i(0)}$. 
    Furthermore, if $a_i = 1$, then the reduced expression $s_{i_1}\dots s_{i_\ell}$ of $w_i$ is unique up to $2$-braid relations, where a 2-braid relation means $s_i s_j = s_j s_i$ for $i, j \in I$ such that $a_{ij} = 0$.
\end{enumerate}
\end{prop}
\begin{proof}
The properties (1)--(3) follow from the definition of $t_{-\lambda_i}$.
For example, see \cite[Section 3]{JKP25} and the references therein for more details. 
Note that $\sfDel_+(t_{-\lambda_i}) = \sfDel_+(w_i)$ and $W$ is a Coxeter group \cite[Proposition 3.13]{Kac}.
Let us prove (4). The existence of the reduced expression of $w_i$ satisfying the condition follows from (3) and \cite[Lemma 3.11]{Kac}.
Suppose that $a_i = 1$. In this case, one can check that $w_i \in \mr{W}$ and $(\lambda_i, \alpha) \in \{ 0, 1\}$ for $\alpha \in \mr{\sfDel}^+$.
    By (1), we have $\alpha + \beta \notin \sfDel_+(w_i)$ for $\alpha, \beta \in \sfDel_+(w_i)$.
    Recall \cite[Proposition 3.13]{Kac} for the braid relations in $W$.
    Any reduced expression of $w_i$ cannot have braid relations different from $2$-braid relations, otherwise there exist $\alpha, \beta \in \sfDel_+(w_i)$ such that $\alpha + \beta \in \sfDel_+(w_i)$, which is not possible (cf.~\cite{Pap94}).
    Since any reduced expression of $w_i$ can be transformed into another using only the braid relations \cite{Iwa64}, the uniqueness follows. 

For type $A_{2n}^{(2)}$, our convention for the Dynkin diagram (see Table \ref{tab:dynkins}) reverses the numbering of the simple roots used in \cite{Kac}. Therefore, for completeness, we include the proof of (2) in this case.
It is known in \cite[Proposition 5.10]{Kac} that $\sfDel$ is characterized by
\begin{equation} \label{eq:char. of real roots}
    \sfDel = \left\{\left. \alpha = \sum_{i=0}^n k_i \alpha_i \in \sfQ \, \right| \, \left| \alpha \right|^2 > 0 \text{ and } \frac{k_i|\alpha_i|^2}{|\alpha|^2} \in \Z \right\},
\end{equation}
and $\sfDel_s$ is given by 
\begin{equation} \label{eq:char. of short real roots}
    \sfDel_s = \left\{ \alpha \in \sfQ \, \left| \, |\alpha|^2 = \min_i |\alpha_i|^2  \right.\right\}.
\end{equation} 

For type $A_{2n}^{(2)}$, recall that $\delta = \alpha_0 + 2\alpha_1 + \dots + 2\alpha_{n-1} + 2\alpha_n$, and the short (resp.~long) real roots have square length $2$ (resp.~$8$). Note that $\alpha_0$ is a long root.
Set $\sfDel_m = \sfDel \setminus (\sfDel_s \cup \sfDel_l)$. Then the real roots in $\sfDel_m$ have square length $4$.
For type $A_{2n}^{(2)}$ with $n \ge 2$, 
we claim that $\sfDel = \sfDel_s \cup \sfDel_m \cup \sfDel_l$, where 
\begin{align} %\label{eq:Del for A2n2}
    \sfDel_s &= \{ \alpha + k\delta \mid \alpha \in \mr{\sfDel}_s,\, k \in \Z \}, \label{eq:Dels for A2n2} \\
    \sfDel_m &= \{ \alpha + k\delta \mid \alpha \in \mr{\sfDel}_l,\, k \in \Z \}, \label{eq:Delm for A2n2} \\
    \sfDel_l &= \{ 2\alpha + (2k+1)\delta \mid \alpha \in \mr{\sfDel}_s,\, k \in \Z \}. \label{eq:Dell for A2n2}
\end{align}

If $\alpha = \sum_{i=0}^n k_i \alpha_i \in \sfDel_s$, then we have $|\alpha|^2 = |\alpha-k_0\delta|^2$ and hence $\alpha-k_0\delta \in \mr{\sfDel}_s$ by \eqref{eq:char. of short real roots}. This implies \eqref{eq:Dels for A2n2}.
Next, if $\alpha = \sum_{i=0}^n k_i \alpha_i \in \sfDel_m$, then it follows from \eqref{eq:char. of real roots} that $k_n$ is divisible by $2$. Thus the $n$-th coefficient of $\alpha-k_0\delta$ is also divisible by $2$, so that $\alpha-k_0\delta \in \mr{\sfDel}_l$ by \eqref{eq:char. of real roots} and $|\alpha-k_0\delta|^2 = |\alpha|^2 = 4$.
This proves \eqref{eq:Delm for A2n2}.
% \[
%     \sfDel_m = \{ \alpha + k\delta \mid \alpha \in \mr{\sfDel}_l,\, k \in \Z \}.
% \]
%
Finally, if $\alpha = \sum_{i=0}^n k_i \alpha_i \in \sfDel_l$, then \eqref{eq:char. of real roots} implies that $k_1, \dots, k_{n-1}$ are divisible by $2$, and $k_n$ is divisible by $4$, so we have
\[
    \frac{1}{2}(\alpha-k_0\delta) \in \sfQ, \quad
    \left| \frac{1}{2}(\alpha-k_0\delta) \right|^2 = 2.
\]
By \eqref{eq:char. of short real roots}, $\frac{1}{2}(\alpha-k_0\delta) \in \mr{\sfDel}_s$ so that
% Since
% \[
%     \frac{1}{2} (\alpha-k_0\delta) = \frac{k_1-2k_0}{2} \alpha_1 + \dots + \frac{k_n-2k_0}{2} \alpha_n \in \mr{\sfDel}_s,
% \]
$k_n - 2k_0 = 2$ (cf.~\eqref{eq: positive short roots of Bn}). 
Since $k_n$ is divisible by $4$, we conclude that $k_0$ is odd, thus we have \eqref{eq:Dell for A2n2}.
% \[
%     \sfDel_l = \{ 2\alpha + (2k+1)\delta \mid \alpha \in \mr{\sfDel}_s,\, k \in \Z \}.
% \]
%
Similarly, for type $A_2^{(2)}$, we have 
\begin{align} %\label{eq:Del for A22}
    \sfDel_s &= \{ \alpha + k\delta \mid \alpha \in \mr{\sfDel},\, k \in \Z \}, \label{eq:Dels for A22} \\
    \sfDel_l &= \{ 2\alpha + (2k+1)\delta \mid \alpha \in \mr{\sfDel},\, k \in \Z \}. \label{eq:Dell for A22}
\end{align}

Since $\tau\left(\sfDel_{\pm}\right) = \sfDel_{\pm}$ for $\tau \in \mc{T}$, we have $\sfDel^+\left(t_{-\lambda_i}\right) = \sfDel_+(w_i)$.
Let $\beta \in \sfDel_+\left(t_{-\lambda_i}\right)$ be given.
Suppose that $\beta$ is a long positive root.
Then there exists a negative long root $\gamma$ such that $\beta = t_{-\lambda_i}\left(\gamma\right)$.
By \eqref{eq:Dell for A2n2} and \eqref{eq:Dell for A22}, 
$\gamma$ is in $\mr{\sfDel}_-$ or given as $2\alpha + (2k+1)\delta$ for $\alpha \in \mr{\sfDel}_s$ and $k < 0$.
Since $\langle \gamma, K \rangle = 0$, it follows from \eqref{eq:translation} that 
%by \eqref{eq:translation}
\begin{equation*}
	\beta = t_{-\lambda_i}(\gamma) = \gamma + \left( \gamma, \lambda_i \right)\delta \in \sfDel_+.
\end{equation*}
If $\gamma \in \mr{\sfDel}_-$, then $\beta \in \sfDel_-$, which contradicts $\beta \in \sfDel_+$. So we assume that $\gamma = 2\alpha + (2k+1)\delta$ for $\alpha \in \mr{\sfDel}_s$ and $k < 0$. Then we have
\begin{equation*}
	\beta 
    %= 2\alpha+(2k+1)\delta + \left( 2\alpha+(2k+1)\delta, \lambda_i \right)\delta 
    %= 2\alpha + \left( 2\left(k + (\alpha, \lambda_i)\right) + 1 \right) \delta
    = 2\alpha + (2l+1)\delta,
\end{equation*}
where $l = k + (\alpha, \lambda_i)$.
Since $\beta \in \sfDel_+$ and $k < 0$, we have $\alpha \in (\mr{\sfDel}_+)_s$ and $0 \le l < (\alpha, \lambda_i)$.
Conversely, let $\beta = 2\alpha + (2k+1)\delta$ for $\alpha \in (\mr{\sfDel}_+)_s$ and $0 \le k < (\alpha, \lambda_i)$. Then we have
\begin{equation*}
	t_{\lambda_i}(\beta) 
    = 2t_{\lambda_i}(\alpha) + (2k+1)\delta 
    = 2\alpha + ( 2\left(k-(\alpha,\lambda_i)\right) +1) \delta \in \sfDel_-,
\end{equation*}
which implies $\beta \in \sfDel_+ \cap t_{-\lambda_i}(\sfDel_-) = \sfDel_+(t_{-\lambda_i})$.
The other cases can be proved similarly.
\end{proof}

\begin{rem}
{\em 
	The statement of \cite[Corollary 3.7]{JKP25} should be revised as in Proposition \ref{prop: properties of fundamental translations}(4). For example, for type $B_3^{(1)}$ with $i = 2$, a reduced expression of $w_2$ is given by $s_2s_3s_2s_1s_2s_3s_2s_0$, where $s_2s_1s_2 = s_1s_2s_1$, hence the uniqueness fails in this case.
}
\end{rem}

\section{Quantum groups and dual canonical basis} \label{sec:quantum groups and dual canonical basis}

\subsection{Quantum groups} \label{subsec: quantum group}
Fix an indeterminate $q$. We put $q_i = q^{d_i}$ and
{\allowdisplaybreaks
\begin{gather*}
[m]_{q_i}=\frac{q_i^m-q_i^{-m}}{q_i-q_i^{-1}}\quad (m\in \Z_+),\quad
[m]_{q_i}!=[m]_{q_i}[m-1]_{q_i}\cdots [1]_{q_i}\quad (m\geq 1),\quad [0]_{q_i}!=1, \\
\begin{bmatrix} m \\ k \end{bmatrix}_{q_i} = \frac{[m]_{q_i}[m-1]_{q_i}\cdots [m-k+1]_{q_i}}{[k]_{q_i}}\quad (0\leq k\leq m).
\end{gather*}}
\!\!If there is no confusion, then we often use the subscript $i$ instead of $q_i$ in the above notations for simplicity.

%Let $d \in \Z$ be a positive integer such that $(\alpha_i, \alpha_i)/2 \in \Z d^{-1}$ for all $i \in I$. 
%Set $q_s = q^{1/d}$. 
%Note that $d = 1$ for all affine types, except for type $A_{2n}^{(2)}$ $(n \ge 1)$, where $d = 2$.
%
%Set $\bk = \C(q_s)$ to be the base field, where $q_s = q^{1/d}$.
Set $\bk = \bigcup_{n \in \N} \C(\!(q^{1/n})\!)$ to be the base field. The quantum group corresponding to $\g$, denoted by $U_q(\g)$, is the $\bk$-algebra generated by the symbols $e_i$, $f_i$ $(i \in I)$, and $q^h$ $(h \in \sfP^\vee)$ subject to the following relations:
{\allowdisplaybreaks
\begin{gather}
	q^0 = 1, \quad q^h q^{h'} = q^{h+h'}, \label{eq: relations 1} \\
	q^h e_i q^{-h}= q^{\langle \alpha_i, h \rangle} e_i,\quad q^h f_i q^{-h}= q^{-\langle \alpha_i, h \rangle} f_i,\quad 
	e_if_j-f_je_i=\delta_{ij}\frac{k_i-k_i^{-1}}{\qi-\qi^{-1}}, \label{eq: relations 2} \\
	\sum_{m=0}^{1-a_{ij}}(-1)^{m} \qbn{1-a_{ij}}{m}_i e_i^{1-a_{ij}-m}e_je_i^{m}=0, \quad (i \neq j), \label{eq: qsr1} \\
	\,\sum_{m=0}^{1-a_{ij}}(-1)^{m} \qbn{1-a_{ij}}{m}_i f_i^{1-a_{ij}-m}f_jf_i^{m}=0, \quad (i \neq j), \label{eq: qsr2}
\end{gather}}
\!\!where $k_i = q^{d_ih_i}$.
We set $e_i^{(m)}=e_i^m/[m]_i!$ and $f_i^{(m)}=f_i^m/[m]_i!$ for $i \in I$ and $m \in \Z_+$.
An element $x \in U_q(\g)$ is said to be homogeneous if there exists $\xi \in \sfP$ such that $q^h x q^{-h} = q^{\langle \xi, h \rangle} x$ for all $h \in \sfP^\vee$, where we denote $\xi$ by ${\rm wt}(x)$, called the weight of $x$.

Let $U_q^0(\g)$ be the $\bk$-subalgebra of $U_q(\g)$ generated by $q^h$ for $h \in \sfP^\vee$.
Also, we denote by $U_q^+(\g)$ (resp.~$U_q^-(\g)$) the $\bk$-subalgebra of $U_q(\g)$ generated by $e_i$ (resp.~$f_i$) for $i \in I$.
Then $U_q(\g)$ has the triangular decomposition as a $\bk$-vector space,
\begin{equation} \label{eq:triangular decomp}
	U_q(\g) \cong U_q^+(\g) \otimes U_q^0(\g) \otimes U_q^-(\g).
\end{equation}
The negative half \,$U_q^-(\mf{g})$ has a root space decomposition, that is, $U_q^-(\mf{g})=\bigoplus_{\beta\in \sfQ_-}U_q(\mf{g})_{\beta}$, where $U_q(\mf{g})_{\beta}=\{\,x\,|\,k_ixk_i^{-1}=q^{(\beta, \alpha_i)}x\ \text{ for } 0\le i\le n \,\}$.
Similarly, the positive half \,$U_q^+(\g)$ also has a root space decomposition.

Let us adopt a Hopf algebra structure on $U_q(\g)$, where the comultiplication $\Delta$ and the antipode $S$ are given by 
{\allowdisplaybreaks
\begin{gather}
	\Delta(q^h)= q^h \ot q^h, \quad
	\Delta(e_i)= e_i \ot 1 + k_i \ot e_i, \quad
	\Delta(f_i)= f_i \ot k_i^{-1} + 1 \ot f_i, \label{eq:comultiplication} \\
	S(q^h)=q^{-h}, \ \ S(e_i)=-k_i^{-1} e_i, \ \  S(f_i)=-f_i k_i. \label{eq:antipode}
\end{gather}}
\!\!We remark that $\Delta$ is denoted by $\Delta_+$ in \cite{Kas91a}, which coincides with the choice of $\Delta$ in \cite{HJ}.

\subsection{Unipotent quantum coordinate rings}
For $i \in I$, let $T_i$ be the $\bk$-algebra automorphism of $U_q({\mf g})$ in \cite{Lu10}, which satisfies
{\allowdisplaybreaks
\begin{align*}
&
T_i(q^h) = q^{s_i(h)}, \\ %\label{eq:Ti(qh)} \\
& 
T_i(e_i) = -f_i k_i, \quad\,\,\,\,\,
T_i(e_j) = \sum_{r+s=-a_{ij}} (-1)^r q_i^{-r} e_i^{(s)} e_j e_i^{(r)}\quad (i\neq j), \\ %\label{eq:Ti(ej)} \\
&T_i(f_i) = -k_i^{-1}e_i, \quad 
T_i(f_j) = \sum_{r+s=-a_{ij}} (-1)^r q_i^{ r} f_i^{(r)} f_j f_i^{(s)}\quad\,\, (i\neq j), %\label{eq:Ti(fj)}
\end{align*}
}\!\!
for $j \in I$ and $h \in \sfP^\vee$, where $s_i(h) = h - \langle \alpha_i, h \rangle\, h_i$.
Note that $T_i=T''_{i,1}$ and $T_i^{-1} = T_{i,-1}'$ in \cite{Lu10}.
It is known in \cite{Lu10} (cf.~\cite{Sai94}) that $T_i^{-1} = \ast T_i \ast$, where $\ast$ is the ${\bf k}$-algebra anti-automorphism of $U_q(\g)$ given by $(e_i)^\ast = e_i$, $ (f_i)^\ast=f_i$ $(0\le i\le n)$, and $(q^h)^\ast = q^{-h}$ $(h \in \sfP^\vee)$. 

Let $w \in W$ be given.
For $\wtd{w} = (i_1,\dots,i_\ell)\in R(w)$, we have $\sfDel_+(w)=\{\,\beta_k\,|\,1\le k\le \ell\,\}$ (e.g.~see \cite{Pap94, Pap95}), where $\beta_k=s_{i_1}\dots s_{i_{k-1}}(\alpha_{i_k})$ for $1\le k\le \ell$.
For ${\bf c}=(c_1,\dots,c_\ell)\in \Z_+^\ell$, we define
\begin{equation} \label{eq:PBW monomial}
F\left({\bf c},\wtd{w}\right)=F(c_1\beta_1) \cdots F(c_\ell\beta_\ell),
\end{equation}
where $F(c_k\beta_k)=T_{i_1}\dots T_{i_{k-1}}( f_{i_k}^{(c_k)} )$ for $1 \le k \le \ell$.
Assume that $\sfDel_+(w)$ is linearly ordered by $\beta_1<\dots<\beta_\ell$.
The $q$-commutation relations of \eqref{eq:PBW monomial} are known as follows:
\begin{equation}\label{eq:LS formula}
F\left(c_j\beta_j\right)F\left(c_i\beta_i\right)-q^{-(c_i\beta_i,c_j\beta_j)}F\left(c_i\beta_i\right)F\left(c_j\beta_j\right)
=\sum_{{\bf c}'}f_{{\bf c}'}F\left({\bf c}',\wtd{w}\right)
\end{equation}
for $i<j$ and $c_i,c_j \in \Z_+$, where the sum is over ${\bf c}'=(c'_k)$ such that $c_i\beta_i+c_j\beta_j=\sum_{i\le k \le j}c'_{k}\beta_k$ with $c'_i<c_i$, $c'_j<c_j$ and $f_{{\bf c}'} \in {\bf k}$ \cite{LeSo91} (cf.~\cite{Ki12}).

\begin{df} \label{df:Uq-w}
{\em 
For $w \in W$, we denote by $U_q^-(w)$ the vector space over ${\bf k}$ generated by $\left\{\,F\left({\bf c},\wtd{w}\right)\,\left|\,{\bf c}\in \Z_+^\ell\right.\right\}$.
}
\end{df}

\noindent
Note that $U_q^-(w)$ does not depend on the choice of $\wtd{w}\in R(w)$, and it is the ${\bf k}$-subalgebra of $U_q^-(\mf{g})$ generated by $\left\{\left.F(\beta_k)\,\right|\,1\le k\le \ell\,\right\}$ by \eqref{eq:LS formula}.

Let $\{ F^{\rm up}({\bf c},\wtd{w}) \mid {\bf c} \in \Z_+^\ell \}$ be the dual basis of $\{ F({\bf c},\wtd{w}) \mid {\bf c} \in \Z_+^\ell \}$ with respect to the Kashiwara bilinear form $(\,\cdot\,\,,\,\,\cdot\,)_{\rm K}$ on $U_q^-(\g)$ \cite[Section 3.4]{Kas91a} (see also \eqref{eq:bilinear form}).
Since $U_q^-(\g)$ is a (twisted) self-dual bialgebra with respect to $(\,\cdot\,\,,\,\,\cdot\,)_{\rm K}$, we regard $F^{\rm up}({\bf c},\wtd{w})$ as an element of $U_q^-(\g)$ by
\begin{equation} \label{eq:dual PBW monomial}
	F^{\rm up}({\bf c},\wtd{w})
	=
	\frac{1}{\left(\,F({\bf c},\wtd{w})\,,\,F({\bf c},\wtd{w})\,\right)_{\rm K}} F({\bf c},\wtd{w}) \in U_q^-(\g).
\end{equation}
Let $\mc{A} = \mathbb{C}\left[q^{\pm 1}\right]$ and let $U_q^-(w)^{\rm up}_{\mc{A}}$ be the $\mc{A}$-lattice generated by $F^{\rm up}({\bf c},\wtd{w})$ for ${\bf c} \in \Z_+^\ell$.
Then the $\bk$-subalgebra $U_q^-(w)$ is called the \emph{unipotent quantum coordinate ring} associated to $w$, since $\C\ot_{\mc{A}}U_q^-(w)^{\rm up}_{\mc{A}}$ is isomorphic to the coordinate ring of the unipotent subgroup $N(w)$ of the Kac--Moody group associated to $w$ (see \cite{GLS13a, Ki12} for more details).

\subsection{Dual canonical basis} \label{subsec:DCB}
Let us recall the (dual) canonical basis of $U_q^-(\mf{g})$ (see \cite{Kas91a, Kas93a, Kas02a} for more details).
Given $0\le i\le n$, there exists a unique ${\bf k}$-linear map $e'_i: U_q^-(\mf{g}) \rightarrow U_q^-(\mf{g})$ such that $e'_i(1)=0$, $e'_i(f_j)=\delta_{ij}$ for $0\le j\le n$, and
\begin{equation} \label{eq:derivation}
e'_i(xy)=e'_i(x)y + q^{({\rm wt}(x),\alpha_i)}xe'_i(y),
\end{equation}
for homogeneous $x,y\in U_q^-(\mf{g})$.
Then there exists a unique non-degenerate symmetric ${\bf k}$-valued bilinear form $(\ ,\ )$ on $U_q^-(\mf{g})$ such that
\begin{equation} \label{eq:bilinear form}
(1,1)=1,\quad (f_ix,y)=(x,e'_i(y)),
\end{equation}
for $0\le i\le n$ and $x,y\in U_q^-(\mf{g})$.

For any homogeneous $x\in U_q^-(\mf{g})$ and $0 \le i \le n$, we have $x=\sum_{k\ge 0}f_i^{(k)}x_k$, where $e'_i(x_k)=0$ for $k\ge 0$. Then we define $\wtd{f}_ix=\sum_{k\ge 0}f_i^{(k+1)}x_k$.
Let $\mc{A}_0$ denote the subring of ${\bf k}$ consisting of rational functions regular at $q=0$.
Let $L(\infty)$ be the $\mc{A}_0$-span of $\wtd{f}_{i_1}\dots \wtd{f}_{i_r}1$ for $r \ge 0$ and $0\le i_1,\dots,i_r\le n$, and let 
\begin{equation*}
\begin{split}
B(\infty)&=\{\,\wtd{f}_{i_1}\dots \wtd{f}_{i_r}1 \!\!\pmod{qL(\infty)} \,|\,r\ge 0,\, 0\le i_1,\dots,i_r\le n\,\}\setminus\{0\}\subset L(\infty)/qL(\infty).
\end{split}
\end{equation*}
The pair $(L(\infty),B(\infty))$ is called the crystal base of $U_q^-(\mf{g})$.

Let $U_q^-(\mf{g})_{\mc{A}}$ be the $\mc{A}$-subalgebra of $U_q^-(\mf{g})$ generated by $f_i^{(k)}$ for $0\le i\le n$ and $k\in \Z_+$. Let $- : U_q^-(\mf{g})\rightarrow U_q^-(\mf{g})$ be the automorphism of $\C$-algebras given by $\ov{q}=q^{-1}$ and $\ov{f_i}=f_i$ for $0\le i\le n$.
Then $(L(\infty),\ov{L(\infty)},U_q^-(\mf{g})_{\mc{A}})$ is balanced, that is, 
the map
\begin{equation}\label{eq:balanced}
\xymatrixcolsep{2pc}\xymatrixrowsep{-0.3pc}\xymatrix{
E:=L(\infty)\cap \ov{L(\infty)}\cap U_q^-(\mf{g})_{\mc{A}} \ \ar@{->}[r]  &\ L(\infty)/qL(\infty) \\
	x \ \ar@{|->}[r] & \, x \,\, \scalebox{0.8}{\text{(mod $qL(\infty)$)}}
}
\end{equation}
is a $\C$-linear isomorphism.
Let $G$ denote the inverse of the map \eqref{eq:balanced}. 
Then 
\begin{equation*}
G(\infty):=\{\,G(b)\,|\,b\in B(\infty)\,\}
\end{equation*} 
is an $\mc{A}$-basis of $U_q^-(\mf{g})_{\mc A}$, which is called the canonical basis or global crystal basis.
Let
\begin{equation*}
G^{\rm up}(\infty)=\{\,G^{\rm up}(b)\,|\,b\in B(\infty)\,\}
\end{equation*}
be the dual basis of $G(\infty)$ with respect to the bilinear form \eqref{eq:bilinear form}, that is, 
$(G^{\rm up}(b),G(b'))=\delta_{bb'}$ for $b,b'\in B(\infty)$.
We call $G^{\rm up}(\infty)$ the dual canonical basis of 
$U_q^-(\mf{g})^{\rm up}_{\mc{A}}:=\{\,x\in U_q^-(\mf{g})\,|\,(x,U_q^-(\mf{g})_{\mc{A}})\in \mc{A}\,\}$.

Given a dominant integral weight $\La \in \sfP_+$, let $V(\La)$ be the irreducible highest weight module over $U_q(\mf{g})$. %$U_q(\mf{g})^e$. 
Let $B(\La)$ and $G(\La)=\{\,G_\La(b)\,|\,b\in B(\La)\,\}$ denote the crystal and canonical basis of $V(\La)$, respectively. 
It is known in \cite{Kas93b} that $L(\infty)^\ast=L(\infty)$ and $B(\infty)^\ast=B(\infty)$.
We may regard $B(\La)\subset B(\infty)$ by
\begin{equation}\label{eq:B(Lambda)}
B(\La)=\{\,b\in B(\infty)\,|\,\varepsilon_i^\ast(b)\leq \langle \La, h_i \rangle \ \text{ for \,} 0\le i\le n\,\},
\end{equation}
where $\varepsilon_i^\ast(b)=\max\{\,k\,|\,b^\ast=\wtd{f}_i^k (b_0^\ast)\ \text{for some $b_0\in B(\infty)$}\,\}$.
We have $G_\La(b)=\pi_\La(G(b))$ for $b\in B(\Lambda)$, where $\pi_\La : U_q^-(\mf{g})\longrightarrow V(\La)$ is the canonical projection.

Let $G^{\rm up}(\La)=\{\,G^{\rm up}(b)\,|\,b\in B(\La)\,\}$ be the dual basis of $G(\La)$ with respect to the bilinear form on $V(\La)$ in \cite[(4.2.4), (4.2.5)]{Kas93a}. 
Let 
\begin{equation*}
\xymatrixcolsep{2pc}\xymatrixrowsep{4pc}\xymatrix{
\iota_\La : V(\La)^\vee \ar@{->}[r]  & U_q(\mf{g})^\vee.}
\end{equation*}
be the dual of $\pi_\La$, where $V(\La)^\vee$ and $U_q^-(\mf{g})^\vee$ are the duals of $V(\La)$ and $U_q^-(\mf{g})$ with respect to \cite[(4.2.4), (4.2.5)]{Kas93a} and \eqref{eq:bilinear form}, respectively.
Then we have 
\begin{equation}\label{eq:embedding and dual canonical basis}
\iota_\La(G^{\rm up}_\La(b))=G^{\rm up}(b),
\end{equation}
for $b\in B(\La)$. Here we identify $V(\La)^\vee$ with $V(\La)$, and $U_q^-(\mf{g})^\vee$ with $U_q^-(\mf{g})$.

\begin{rem} \label{rem: e' and te}
{\em
For $b \in B(\La)$ and $m \in \Z_{>0}$, it is known in \cite[Lemma 5.1.1]{Kas93a} that 
\begin{equation*} %\label{eq:useful fact 1}
\begin{split}
& \te_i^{\,m}b = {\bf 0} \,\, \Longleftrightarrow \,\,
e_i^m G^{\rm up}_\La(b) = 0 \,\, \Longleftrightarrow \,\,
(e_i')^m G^{\rm up}(b) = 0.
\end{split}
\end{equation*}
In particular, for $b\in B(\La)$ with $\varepsilon_i(b)=1$, we have
\begin{equation*} %\label{eq:useful fact 2}
G^{\rm up}_\La(\te_ib) = e_iG^{\rm up}_\La(b), 
\end{equation*}
which is equivalent to $G^{\rm up}(\te_ib) = e_i'G^{\rm up}(b)$ (cf. \cite[Theorem 3.14]{Ki12}). 
Thus one may calculate the action of $e_i'$ on some root vectors by means of crystal structure of $B(\infty)$ (see Example \ref{ex: xz for twisted types}).
}
\end{rem}

In the rest of this subsection, we fix $s \in \I$ such that the fundamental $\mr{\g}$-module with highest weight $\ov{\La_s}$ is minuscule. We call such a $s \in \I$ the \emph{minuscule node}. The complete list of minuscule nodes for each type is given in Table \ref{tab:minuscule nodes}.

\begin{table}[h!]
    \centering
    \begin{tabular}{c||c|c|c|c|c|c}
       type & $A_n^{(1)}$ & $A_{2n-1}^{(2)}$ & $D_{n+1}^{(2)}$ & $D_n^{(1)}$ & $E_6^{(1)}$ & $E_7^{(1)}$ \\ \hline
       $s$  & $1, 2, \dots, n$ & $1$ & $n$ & $1, n-1, n$ & $1, 5$ & $6$
    \end{tabular}
    \smallskip
    \caption{the minuscule nodes for type $X_N^{(r)}$}
    \label{tab:minuscule nodes}
\end{table}

For any minuscule node $s$, it follows from Proposition \ref{prop: properties of fundamental translations} that $\sfDel_+(t_{-\lambda_s}) \subset \mr{\sfDel}_+$
so that one may regard $U_q^-(w_s)$ as a subalgebra of $U_q^-(\cg)$. In particular, $\theta \in \sfDel_+(t_{-\lambda_s})$, so $F^{\rm up}(\theta) \in U_q^-(w_s)$.
Let $\wtd{w_s} = (i_1, i_2, \dots, i_\ell) \in R(w_s)$ be given.
For $1\le k\le \ell$, we denote ${\bf 1}_{\beta_k} \in B(\infty)$ by
$F^{\rm up}(\beta_k) \!\pmod{qL(\infty)} \in B(\infty)$.

\begin{rem}
{\em 
The reduced expression $\wtd{w_s}$ can be found in \cite[Remark 3.10]{JKP25}, which is unique up to $2$-braid relations by Proposition \ref{prop: properties of fundamental translations}(4).
}
\end{rem}

We now review some properties of dual PBW vectors associated to $w_s$ for any minuscule node $s$ as follows.\footnote{The properties were verified in \cite{JKP23} for types $A_n^{(1)}$ and $D_n^{(1)}$, but the arguments there rely only on general properties of dual canonical basis. Hence the same statements hold for any minuscule node, and thus we omit the proofs.}

\begin{lem}\!\!\!{\em \cite[Lemma 3.6]{JKP23}} \label{lem:dual PBW properties 1}
\,Let $1 \le k \le \ell$ be given.
\begin{enumerate}[\em (1)]
	\item We have 
            \begin{equation*}
                F^{\rm up}(\beta_k)\in \iota_{\La_s}(G^{\rm up}(\La_s))
            \end{equation*}
            In particular, ${\bf 1}_{\beta_k}\in B(\La_s)$ by regarding $B(\La_s)$ as a subset of $B(\infty)$ $(${\em cf.}~\eqref{eq:B(Lambda)}, \eqref{eq:embedding and dual canonical basis}$)$.
	
	\item For $i \in \I$, we have 
                \begin{equation*}
                     e'_i ( F^{\rm up}(\beta_k) ) = 
                    \begin{cases}
 	                  F^{\rm up}(\beta_k - \alpha_i)  & \text{ if }  (\alpha_i, \beta_k) = d_i + \delta_{i,s},  \\
 	                  0 & \text{otherwise,}
                    \end{cases}
                \end{equation*}
            where we understand $F^{\rm up}(0) = 1$.
\end{enumerate}
\end{lem}

We now consider the following root vector with weight $\theta$, which plays an important role in the present work.
\begin{equation*} %\label{eq:xz}
    \xz := F^{\rm up}(\beta_\ell) = F^{\rm up}(\theta).
\end{equation*}

\begin{lem}\!\!\!{\em \cite[Lemma 3.8]{JKP23}} \label{lem:dual PBW properties 2}
Let $J_0= \left\{ \, i\in \I \, | \, (\al_i, \theta)=0 \, \right\}$ and $J_1= \I \, \setminus \, J_0$.
\begin{enumerate}[\em (1)]
    \item For $i\in \I$, we have
	\begin{equation*}
		\left\{
			\begin{array}{ll}
				e_i'(\xz) = 0 & \text{if $i \in J_0$}, \\
				(e_i')^2 (\xz) = 0 & \text{if $i \in J_1$}.
			\end{array}
		\right.		
	\end{equation*}
	
    \item For $i\in J_1$, we have
        \begin{equation*}
	   \xz \cdot e_i'(\xz) = q^{-d_0(2+a_{0i})} e_i'(\xz) \cdot \xz\,.
        \end{equation*}
\end{enumerate}
\end{lem}

\begin{ex} \label{ex: xz for twisted types}
{\em
    Assume that $\g$ is of type $A_5^{(2)}$ with $n = 3$. In this case, $\cg$ is of type $C_3$, and
    $s = 1$ is the minuscule node. By Proposition \ref{prop: properties of fundamental translations}, it is straightforward to check that
    \begin{equation*}
        \sfDel_+\left( t_{-\lambda_1} \right) = \left\{\, \alpha_1, \alpha_1 + \alpha_2, \alpha_1 + \alpha_2 + \alpha_3, \alpha_1 + 2\alpha_2 + \alpha_3, 2\alpha_1 + 2\alpha_2 + \alpha_3 \,\right\},
    \end{equation*}
    where $2\alpha_1 + 2\alpha_2 + \alpha_3$ is a unique long root in $\sfDel_+\left( t_{-\lambda_1} \right)$, and the other roots are short. Note that $\theta = \alpha_1 + 2\alpha_2 + \alpha_3$.
    On the other hand, we obtain $\wtd{w_1} = (1,2,3,2,1)$ so that $\beta_1 = \alpha_1$, $\beta_2 = \alpha_1 + \alpha_2$, $\beta_3 = 2\alpha_1 + 2\alpha_2 + \alpha_3$, $\beta_4 = \alpha_1 + \alpha_2 + \alpha_3$, and $\beta_5 = \alpha_1 + 2\alpha_2 + \alpha_3$. By Lemma \ref{lem:dual PBW properties 1}(2), we have 
    \begin{equation*}
    \,\,
        \begin{tabular}{c||r|r|r|r|r}
           $(\,\cdot\,,\,\cdot\,)$  & $\beta_1$  & $\beta_2$ & $\beta_3$ & $\beta_4$ & $\beta_5$ \\ \hline \hline
           $\alpha_1$  & $\bf 2$ & $1$ & $\bf 2$ & $1$ & $0$ \\ \hline
           $\alpha_2$  & $-1$ & $\bf 1$ & $0$ & $-1$ & $\bf 1$ \\ \hline
           $\alpha_3$  & $0$ & $-2$ & $0$ & $\bf 2$ & $0$ \\
        \end{tabular}
        \,\quad\,
        \xymatrix@C=1.7em @R=1.7em{
            1 \ar@{<-}[r]^{1} & F_{\beta_1}^{\rm up} \ar@{<-}[r]^{\,\,2} & F_{\beta_2}^{\rm up} \ar@{<-}[r]^{\,\,3} & F_{\beta_4}^{\rm up} \ar@{<-}[r]^{\,\,2} & F_{\beta_5}^{\rm up} \ar@{<-}[r]^{\,\,1} & F_{\beta_3}^{\rm up}
        }
    \end{equation*}
    where $(\alpha_i, \beta_k) = d_i + \delta_{i,s}$ is written in bold, and $X \overset{i}{\longrightarrow} Y$ denotes $Y = e_i'(X)$.
    Thus we have $e_2'(\xz) = F_{\beta_4}^{\rm up}$, $(e_2')^2 (\xz) = 0$ and $e_i' (\xz) = 0$ for $i \in J_0$, where $J_0 = \{1, 3\}$ and $J_1 = \{ 2 \}$. Since $d_0 = 1$ and $a_{02} = -1$, we have $\xz \cdot e_2'(\xz) = q^{-1} e_2'(\xz) \cdot \xz$.
    Note that it is known (e.g.~see \cite[Section 5.6]{HN}) that the description of the (level $0$) $\g$-crystal $B(\ov{\La_1})$ is given by
    
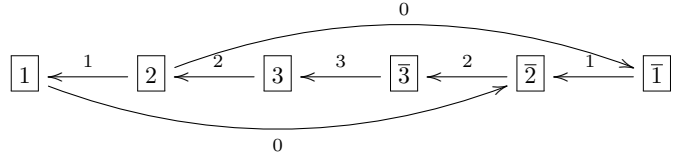
\begin{equation*}
        \xymatrix@C=3em{
            \scalebox{0.9}{\boxed{1\vphantom{\ov{3}}}} \ar@{<-}[r]^{1} & \scalebox{0.9}{\boxed{2\vphantom{\ov{3}}}} \ar@{<-}[r]^{\,\,2} & \scalebox{0.9}{\boxed{3\vphantom{\ov{3}}}} \ar@{<-}[r]^{3} & \scalebox{0.9}{\boxed{\ov{3}}} \ar@{<-}[r]^{2} & \scalebox{0.9}{\boxed{\ov{2}}} \ar@{<-}[r]^{\!\!1} \ar@{<-}@/^1.65pc/[llll]^{0} & \scalebox{0.9}{\boxed{\ov{1}}} \ar@{<-}@/_1.65pc/[llll]_{0}
        }
    \end{equation*}
    where $\boxed{a\vphantom{b}} \overset{i}{\longleftarrow} \boxed{b}$ denotes $\boxed{a\vphantom{b}} = \te_i\, \boxed{b}$.
    \smallskip

    Next, assume that $\g$ is of type $D_3^{(2)}$ with $n = 2$. In this case, $\cg$ is of type $B_2$ and $s = 2$ is the minuscule node. It is straightforward to check that
    \begin{equation*}
        \sfDel_+\left( t_{-\lambda_2} \right) = \left\{\, \alpha_2, \alpha_1+\alpha_2, \alpha_1+2\alpha_2 \,\right\},
    \end{equation*}
    where $\alpha_1+2\alpha_2$ is the long root in $\sfDel_+\left( t_{-\lambda_2} \right)$. Note that $\theta = \alpha_1+\alpha_2$.
    On the other hand, we have $\wtd{w_2} = (2,1,2)$ so that $\beta_1 = \alpha_2$, $\beta_2 = \alpha_1 + 2\alpha_2$, and $\beta_3 = \alpha_1 + \alpha_2$. By Lemma \ref{lem:dual PBW properties 1}(2), we have 
	\begin{equation} \label{eq: computation on root vectors in D32}
		\begin{tabular}{c||c|c|c}
            $(\,\cdot\,,\,\cdot\,)$ & $\beta_1$  & $\beta_2$  & $\beta_3$ \\ \hline \hline
            $\alpha_1$ & $-2$ & $0$ & $\bf 2$ \\ \hline
            $\alpha_2$ & $\bf 2$ & $\bf 2$ & $0$ \\
        \end{tabular}
        \,\quad\,
        \xymatrix@C=3em @R=3em{
            1 \ar@{<-}[r]^{2} & F_{\beta_1}^{\rm up} \ar@{<-}[r]^{1} & F_{\beta_3}^{\rm up} \ar@{<-}[r]^{2} & F_{\beta_2}^{\rm up}
        }
	\end{equation}
    Thus this yields that $e_1'(\xz) = F_{\beta_1}^{\rm up}$, $(e_1')^2 (\xz) = 0$ and $e_i' (\xz) = 0$ for $i \in J_0$, where $J_0 = \{ 2 \}$ and $J_1 = \{ 1 \}$. Since $d_0 = 1$ and $a_{01} = -2$, we have $\xz \cdot e_1'(\xz) = e_1'(\xz) \cdot \xz$.
    Note that it is known (e.g.~see \cite[Section 5.7]{HN}) that the description of the (level $0$) $\g$-crystal $B(\ov{\La_2})$ is given by
    \vskip -7mm
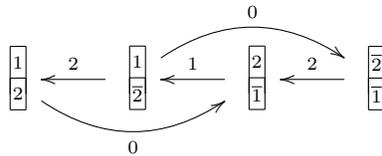
    \begin{equation*}
        \xymatrix{
            \raisebox{0.3em}{\spin{$1$}{$2$}} \ar@{<-}[r]^{2} & \raisebox{0.3em}{\spin{$1$}{$\ov{2}$}} \ar@{<-}[r]^{1} & \raisebox{0.3em}{\spin{$2$}{$\ov{1}$}} \ar@{<-}[r]^{2} \ar@{<-}@/^1.65pc/[ll]^{0} &
            \raisebox{0.3em}{\spin{$\ov{2}$}{$\ov{1}$}}
            \ar@{<-}@/_1.65pc/[ll]_{0}
        }
    \end{equation*}
    %where each tableau is drawn with half-boxes.

    Finally, assume that $\g$ is of type $D_4^{(2)}$ with $n = 3$. In this case, $\cg$ is of type $B_3$ and $s = 3$ is the minuscule node. Then we have $\wtd{w_3} = (3,2,1,3,2,3)$ and
    \begin{equation*}
        \sfDel_+(t_{-\lambda_3}) = \{ \alpha_3, \alpha_2 + 2\alpha_3, \alpha_1+\alpha_2+2\alpha_3, \alpha_2+\alpha_3, \alpha_1+2\alpha_2+2\alpha_3, \alpha_1+\alpha_2+\alpha_3 \},
    \end{equation*}
    where $\alpha_2+2\alpha_3$, $\alpha_1+\alpha_2+2\alpha_3$, $\alpha_1+2\alpha_2+2\alpha_3$ are long and the others are short. Note that $\theta = \alpha_1+\alpha_2+\alpha_3$.
    By Lemma \ref{lem:dual PBW properties 1}(2), we obtain
    \begin{equation*}
        \xymatrix@C=3em @R=0em{
            & & & \ar@{->}[dl]_{1} F_{\beta_6}^{\rm up} \\
            1 \ar@{<-}[r]^{3} & F_{\beta_1}^{\rm up} \ar@{<-}[r]^{2} & F_{\beta_4}^{\rm up} & & \ar@{->}[ul]_{3} \ar@{->}[dl]^{1} F_{\beta_3}^{\rm up} \ar@{<-}[r]^{2} & F_{\beta_5}^{\rm up} & F^{\rm up}_{\beta_1} F^{\rm up}_{\beta_5} \ar@{->}[l]_{3\quad} \\
            & & & \ar@{->}[ul]^{3} F_{\beta_2}^{\rm up}
        }
    \end{equation*}
    where $\beta_1 = \alpha_3$, $\beta_2 = \alpha_2+2\alpha_3$, $\beta_3 = \alpha_1 + \alpha_2 + 2\alpha_3$, $\beta_4 = \alpha_2 + \alpha_3$, $\beta_5 = \alpha_1 + 2\alpha_2 + \alpha_3$ and $\beta_6 = \alpha_1 + \alpha_2 + \alpha_3$. Note that the description of the (level $0$) $\g$-crystal $B(\ov{\Lambda_3})$ is given by 
	
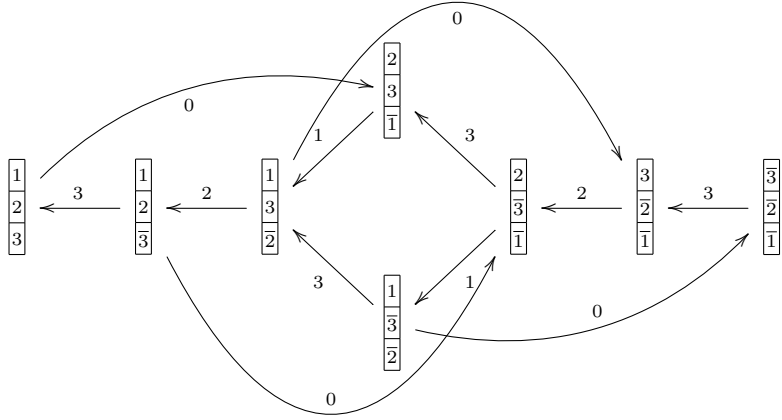
\begin{equation*}
		 \xymatrix@C=3em @R=0em{
            & & &\!\!\! \ar@{->}[dl]_{1} \raisebox{0.3em}{\spinii{$2$}{$3$}{$\ov{1}$}} \\
            \raisebox{0.3em}{\spinii{$1$}{$2$}{$3$}} \ar@{<-}[r]^{3} \ar@{->}@/^2pc/[rrru]_{0} & \!\!\raisebox{0.3em}{\spinii{$1$}{$2$}{$\ov{3}$}} \ar@{<-}[r]^{2} \ar@{->}@/_6.5pc/[rrr]^{0} & \!\!\raisebox{0.3em}{\spinii{$1$}{$3$}{$\ov{2}$}} \ar@{->}@/^6.5pc/[rrr]_{0} & & \ar@{->}[ul]_{3} \ar@{->}[dl]^{1} \!\! \raisebox{0.3em}{\spinii{$2$}{$\ov{3}$}{$\ov{1}$}} \ar@{<-}[r]^{2} & \!\! \raisebox{0.3em}{\spinii{$3$}{$\ov{2}$}{$\ov{1}$}} & \!\! \raisebox{0.3em}{\spinii{$\ov{3}$}{$\ov{2}$}{$\ov{1}$}} \ar@{->}[l]_{3} \\
            & & & \ar@{->}[ul]^{3} \!\!\! \raisebox{0.3em}{\spinii{$1$}{$\ov{3}$}{$\ov{2}$}} \ar@{->}@/_2pc/[rrru]^{0}
        }
	\end{equation*}
}
\end{ex}

\section{Representation theory of Borel algebras} \label{sec:representation theory of borel algebras}

\subsection{Quantum loop algebras and Borel subalgebras} % Definitions on (twisted) quantum loop algebras and Borel subalgebras

Let $\sfC = (c_{ij})_{i,j\in I_\fin}$ be a Cartan matrix of finite type $X_n$ with an index set $I_\fin = \{ 1, 2, \dots, n \}$, and 
let $\g_\fin = \g(\sfC)$ be a finite-dimensional simple Lie algebra corresponding to $\sfC$. We denote by $U_q(\ms{L}\g_\fin)$ the quantum loop algebra associated to $\g_{\fin}$, which is defined as a quotient of $U_q'(\g)$ by the relation $\sum_{i \in I} k_i^{a_i} = 1$, where $\g = \g(\sfA)$ is an affine Kac--Moody algebra associated to a Cartan matrix $\sfA$ of type $X_n^{(1)}$ (as in Section \ref{subsec: affine root system}) so that $\cg = \g_\fin$ and $\I = I_\fin$, and $U_q'(\g)$ is the $\bk$-subalgebra of $U_q(\g)$ without the degree operator $q^d$.
Here we often write $\widehat{\g_\fin}$ instead of $\g(\sfA)$ when we emphasize its type.

The quantum loop algebra $U_q(\ms{L}\g_\fin)$ has the Drinfeld--Jimbo presentation as in Section \ref{subsec: quantum group}.
It also has another presentation, due to Drinfeld \cite{Dri87b} (see also \cite{Bec94a}), with infinitely many generators $x_{i,m}^\pm$, $k_i^{\pm 1}$ $(i \in I_\fin\,\text{ and }\, m \in \Z)$, $\psi_{i,\pm m}^\pm$ $(i \in I_\fin \,\text{ and }\, m \ge 0)$, called the loop generators, and defining relations on them.

Let $\sigma : I_\fin \rightarrow I_\fin$ be a bijection such that $c_{\sigma(i)\,\sigma(j)} = c_{ij}$ for all $i,j \in I_\fin$. Such a map $\sigma$ is called an automorphism of the Dynkin diagram of $\g_\fin$. Let $r$ be the order of $\sigma$. Under the Cartan--Killing classification of Dynkin diagrams of finite types, we have $r \in \{ 2, 3 \}$, only when $\g_\fin$ is of type $A_n$, $D_n$ $(n \ge 3)$, and $E_6$.

Let $I_\fin^\sigma$ be the set of orbits of $\sigma$, where we denote by $\ov{i} \in I_\fin^\sigma$ for $i \in I_\fin$. 
Let $\widehat{\g_\fin}^\sigma$ be a subalgebra of $\widehat{\g_\fin}$ fixed by an automorphism of $\widehat{\g_\fin}$ induced from $\sigma$ (see \cite[Section 8]{Kac} for more details). Then it is known in \cite{Kac} that $\widehat{\g_\fin}^\sigma\!\! = \g(\sfA)$ is an affine Kac--Moody algebra associated to a Cartan matrix $\sfA = (a_{ij})_{i,j \in \hat{I}_\fin^\sigma}$ of type $X_N^{(r)}$, where $\hat{I}_\fin^\sigma = I_\fin^\sigma \sqcup \{ 0 \}$ and $N = |I_\fin|$. Recall that we use the numbering of Dynkin diagrams in Table \ref{tab:dynkins}.

As in untwisted types, the twisted quantum loop algebra $U_q(\ms{L}\g_\fin^\sigma)$ is defined as a quotient of the $\bk$-subalgebra $U_q'(\widehat{\g_\fin}^\sigma)$ of the quantum group $U_q(\widehat{\g_\fin}^\sigma)$ by the relation $\sum_{i \in \hat{I}_\fin^\sigma} k_i^{a_i} = 1$, and it also has a Drinfeld presentation with infinitely many generators $x_{i,m}^\pm$, $k_i^{\pm 1}$ $(i \in I_\fin\,\text{ and }\, m \in \Z)$, $\psi_{i,\pm m}^\pm$ $(i \in I_\fin \,\text{ and }\, m \ge 0)$, and defining relations together with $\sigma$ \cite{CP98}. In particular, for a primitive $r$-th root of unity, say $\omega$, the defining relations involve
\begin{equation} \label{eq: loop generators under sigma}
    x_{\sigma(i),m}^\pm = \omega^m x_{i,m}^\pm, \quad
    k_{\sigma(i)}^{\pm 1} = k_i^{\pm 1}, \quad
    \psi_{\sigma(i),\pm m}^\pm  = \omega^m \psi_{i,\pm m}^\pm.
\end{equation}

\begin{rem} \label{rem: vanishing generators}
{\em 
Put
\begin{equation*}
    {\tt d}_i = \begin{cases}
        1 & \text{if $r = 1$ or $X_N^{(r)} = A_{2n}^{(2)}$,} \\
        \max\left\{ 1, a_i^\vee a_i^{-1} \right\} & \text{otherwise.}
    \end{cases}
\end{equation*}
One can check that $\sigma(i) = i$ if and only if $r | {\tt d}_i$. Note that ${\tt d}_i = 1$ if $\sigma(i) \neq i$.
By \eqref{eq: loop generators under sigma}, we have $x_{i, m}^\pm = 0$ and $\psi_{i,\pm m}^\pm = 0$ when $\sigma(i) = i$ and ${\tt d}_i \nmid m$ \cite[Remark 4.4]{Dam12}.
}
\end{rem}

Since the quantum group $U_q(\g(\sfA))$ has a Hopf algebra structure, the (twisted) quantum loop algebras inherit a Hopf algebra structure from that of $U_q(\g(\sfA))$, where the comultiplication $\Delta$ and the antipode $S$ are given as in \eqref{eq:comultiplication} and \eqref{eq:antipode}, respectively. Also the (twisted) quantum loop algebra has a triangular decomposition as in \eqref{eq:triangular decomp} in terms of the loop generators \cite{BCP99, Dam15}.

\begin{df}
{\em
Let $U_q(\bo)$ (resp.~$U_q(\bo^\sigma)$) be the $\bk$-subalgebra of $U_q(\ms{L}\g_\fin)$ (resp.~$U_q(\ms{L}\g_\fin^\sigma)$) generated by $e_i$ and $k_i^{\pm 1}$ for $i \in I$ (resp.~$i \in \hat{I}_\fin^\sigma$), which is called the \emph{Borel subalgebra} of the (twisted) quantum loop algebra.
}
\end{df}

Note that $U_q(\bo)$ (resp.~$U_q(\bo^\sigma)$) is a Hopf subalgebra of $U_q(\ms{L}\g_\fin)$ (resp.~$U_q(\ms{L}\g_\fin^\sigma)$). 
It is known (e.g.~see \cite[\S 4.21]{Jan96}) that the Borel subalgebra is isomorphic to the $\bk$-algebra with generators $e_i$ and $k_i^{\pm 1}$ for $i \in I$ (resp.~$i \in \hat{I}_\fin^\sigma$) satisfying the relations \eqref{eq: relations 1}--\eqref{eq: qsr1} in the Drinfeld-Jimbo presentation of $U_q(\ms{L}\g_\fin)$ (resp.~$U_q(\ms{L}\g_\fin^\sigma)$) except for the relations involving $f_i$'s. 
We further remark that it follows from \cite{Bec94a, Dam00, Dam12} that $U_q(\bo)$ (resp.~$U_q(\bo^\sigma)$) contains the loop generators $x_{i,l}^+$, $k_i^{\pm 1}$, $x_{i,m}^-$, and $\psi_{i,l}^+$ for $i \in I_\fin$, $l \ge 0$, and $m > 0$, but it is not generated by these generators, since $e_0$ cannot be generated by them. It also has a triangular decomposition induced from that of the corresponding quantum loop algebra in terms of the loop generators (see \cite{HJ, Wan23} for more details).

\begin{rem} \label{rem: convention-1; notation for Borel}
{\em 
For each $\ov{i} \in I_\fin^\sigma$, we always take the smallest representative $i \in I$ for $\ov{i}$, and identify it with its orbit. For example, if $\sfC$ is of type $E_6$, then $\ov{1} = \{ 1, 5 \}$, $\ov{2} = \{ 2, 4 \}$, $\ov{3} = \{ 3 \}$, and $\ov{6} = \{ 6 \}$ so that we may identify $I_\fin^\sigma = \{ 1, 2, 3, 6 \}$ (but we denote usually by $4$ the orbit $\ov{6}$, that is, $I_\fin^\sigma = \{ 1, 2, 3, 4 \}$). Then we identify $\I$ with $I_\fin^\sigma$ in this sense. 
Note that $\I \neq I_\fin$ for twisted types, while $\I = I_\fin$ for untwisted types.
}
\end{rem}

\subsection{Category $\mc{O}$ and prefundamental modules} % Review on the category O introduced by Hernandez-Jimbo (untwisted cases) and Wang (twisted cases).

In this subsection, we briefly introduce the category $\mc{O}$ and prefundamental modules following \cite{HJ, Wan23}. Let us recall Remark \ref{rem: convention-1; notation for Borel}.
\smallskip

\emph{Throughout this subsection, for twisted types, we often write $U_q(\bo)$ (resp.~$I$) instead of $U_q(\bo^\sigma)$ (resp.~$\hat{I}_\fin^\sigma$) for simplicity, if there is no confusion.}
\smallskip

Let $\mf{t}$ be the subalgebra of $U_q(\bo)$ generated by $k_i^{\pm 1}$ for $i\in\I$, and let $\mf{t}^\ast = (\bk^\times)^{\I}$ be the set of maps from $\I$ to $\bk^\times$, which is a group under pointwise multiplication.
Let $V$ be a $U_q(\bo)$-module. For $\omega \in {\mf t}^*$, we define the weight space of $V$ with weight $\omega$ by
\begin{equation} \label{eq:weight space wrt t}
	V_{\omega} = \{\, v \in V \, | \, k_i v = \omega(i)v \  \text{ for } i \in \I \,\}.
\end{equation}
We say that $V$ is ${\mf t}$-diagonalizable if $V = \oplus_{\omega \in {\mf t}^*} V_{\omega}$.

A series ${\bf \Psi} = (\Psi_{i, m})_{i \in I_\fin, m \ge 0}$ of elements in $\bk$ such that $\Psi_{i, 0} \neq 0$ for all $i \in I_\fin$ is called an $\ell$-weight.
We often identify ${\bf \Psi} = (\Psi_{i, m})_{m \ge 0}$ with ${\bf \Psi} = (\Psi_i(z))_{i \in I_\fin}$, a tuple of formal power series, where
\begin{equation*}
	\Psi_i(z) = \sum_{m \ge 0} \Psi_{i,m} z^m.
\end{equation*}

We denote by ${\mf t}_{\ell}^*$ the set of $\ell$-weights. Since $\Psi_i(z)$ is invertible, ${\mf t}_{\ell}^*$ is a group under multiplication. Let $\varpi : {\mf t}_{\ell}^* \longrightarrow {\mf t}^*$ be the surjective morphism defined by $\varpi(\Psi)(i) = \Psi_{i,0}$ for $i\in I_\fin$.

For ${\bf \Psi} \in {\mf t}_{\ell}^*$, we define an $\ell$-weight space of $V$ with $\ell$-weight ${\bf \Psi}$ by
\begin{equation*}
	V_{\bf \Psi} = \left\{ v \in V \left| \text{ there exist } p \in \mathbb{Z}_+ \text{ such that } (\psi_{i,m}^+ - \Psi_{i,m})^p v = 0 \text{ for all } i \in I_\fin \text{ and } m \ge 0 \right. \right\}.
\end{equation*}
For ${\bf \Psi} \in {\mf t}_{\ell}^*$, we say that $V$ is of highest $\ell$-weight $\bf \Psi$ if there exists a non-zero vector
$v \in V$ such that
\vskip 2mm
\begin{center}
	(i) $V = U_q(\bo)v$, \,\, (ii) $e_i v = 0$ for all $i \in I_\fin$, \,\, (iii) $\psi_{i,m}^+ v = \Psi_{i, m}v$ for $i \in I_\fin$ and $m \ge 0$.
\end{center}
\vskip 2mm
A non-zero vector $v \in V$ is called a highest $\ell$-weight vector of $\ell$-weight ${\bf \Psi}$ if it satisfies the conditions (ii) and (iii), where ${\bf \Psi}$ is called the highest $\ell$-weight of $V$.
There exists a unique irreducible $U_q(\bo)$-module of highest $\ell$-weight $\bf \Psi$, which we denote by $L({\bf \Psi})$ \cite{HJ, Wan23}.

\begin{rem} \label{rem: l-weight on orbit}
{\em 
For twisted types, it follows from \eqref{eq: loop generators under sigma} and the relation on $\psi_{i,m}^+$'s that a highest $\ell$-weight ${\bf \Psi} = (\Psi_i(z))_{i \in I_\fin}$ satisfies the relation
\begin{equation*}
    \Psi_{\sigma(i)}(z) = \Psi_i(\omega z)
\end{equation*}
for all $i \in I_\fin$, where $\sigma : I_\fin \rightarrow I_\fin$ is an automorphism of the Dynkin diagram of $\g_\fin$.
}    
\end{rem}

\begin{df}{\em \!(\!\!\cite[Definition 3.7]{HJ}, \cite[Definition 4.8]{Wan23})} \label{df:prefundamental}
	{\em
	For $s \in I_0$ and $a \in \C^{\times}$, let $L_{s, a}^-$ be an irreducible $U_q(\bo)$-module of highest $\ell$-weight ${\bf \Psi} = (\Psi_i(z))_{i \in I_\fin}$, where $\Psi_i(z)$ is defined as follows:
    \begin{itemize}
        \item for untwisted types,
            \begin{equation*} 
		          \Psi_i(z) =
                \begin{cases}
                    (1-az)^{-1} & \text{if $i=s$},      \\  
                    1             & \text{otherwise},
                \end{cases}
	        \end{equation*}

        \item for twisted types with $s \neq \sigma(s)$, 
            \begin{equation*} 
		          \,\,\quad\,\,
                \Psi_i(z) =
                \begin{cases}
                    (1-a\omega^kz)^{-1} & \text{if $i=\sigma^k(s)$},      \\  
                    1             & \text{otherwise},
                \end{cases}
	        \end{equation*}

        \item for twisted types with $s = \sigma(s)$, 
            \begin{equation*} 
                \,\,
		          \Psi_i(z) =
                \begin{cases}
                    (1-a^r z^r)^{-1} & \text{if $i=s$},      \\  
                    1             & \text{otherwise},
                \end{cases}
	        \end{equation*}
            where $r$ is the order of $\sigma$.
    \end{itemize}
    The irreducible $U_q(\bo)$-module $L_{s,a}^+$ is defined analogously, with each (non-trivial) factor above replaced by its inverse.
	The $U_q(\bo)$-modules $L_{s,a}^-$ and $L_{s,a}^+$ are called the negative and positive \emph{prefundamental modules} over $U_q(\bo)$, respectively.
	}
\end{df}

We define a map $q^{(\,\,\cdot\,\,)} : \mr{\sfP} \rightarrow {\mf t}^*$ given by $q^{\ov{\La}_i}(j) = q^{(\ov{\La}_i,\,\alpha_j)} = q_i^{\delta_{ij}}$ for $i, j \in \I$. Note that we have $q^{\alpha_i}(j)= q_i^{a_{ij}}$ for $i, j \in \I$.
We define a partial order $\le$ on ${\mf t}^*$ by
$\omega' \le \omega$  if and only if $\omega'\omega^{-1}$ is a product of $q^{-\alpha_i}$'s.
%$\widetilde{\alpha}_i^{-1}$'s.
For $\lambda \in {\mf t}^*$, put $D(\lambda) = \{\, \omega \in {\mf t}^* \, | \, \omega \le \lambda \,\}$.

\begin{df}{\em \!\!\!\cite[Definition 3.8]{HJ}} \label{df:category O}
	{\em
		Let $\mathcal{O}$ be the category of $U_q(\bo)$-modules $V$ such that
		\begin{enumerate}
			\item[(i)] $V$ is ${\mf t}$-diagonalizable,
			\item[(ii)] $\dim V_{\omega} < \infty$ for all $\omega \in {\mf t}^*$,
			\item[(iii)] there exist $\lambda_1, \dots, \lambda_t \in {\mf t}^*$ such that the weights of $V$ are in $\bigcup_{j=1}^t D(\lambda_j)$.
		\end{enumerate}
	}
\end{df}

\noindent
The category $\mc{O}$ is closed under taking finite direct sums,  quotients, and finite tensor products of objects in $\mc{O}$.

The simple objects in $\mc{O}$ are characterized in terms of tuples ${\bf \Psi} = (\Psi_i(z))_{i \in I_\fin}$ of rational functions regular and non-zero at $z = 0$, called Drinfeld rational fractions \cite{HJ}, which can be regarded as a natural extension of Drinfeld polynomials \cite{CP95a, CP95b, CP98}, as follows:
\begin{thm}{\em \!(\!\!\cite[Theorem 3.11]{HJ}, \cite[Theorem 4.6]{Wan23})} \label{thm:characterization of simples in O}
	For ${\bf \Psi} \in {\mf t}_{\ell}^*$, $L({\bf \Psi})$ is in the category $\mc{O}$ if and only if $\Psi_i(z)$ is rational, regular, and non-zero at $z=0$, for all $i\in I_\fin$.
\end{thm}

For ${\bf \Psi}, {\bf \Psi}'\in {\mf t}_{\ell}^*$, 
it follows from the formulas of $\Delta(\Psi_{i,\pm k}^{\pm})$ and $\Delta(x_{i,k}^+)$ \cite{Dam00} (see also \cite[Theorem 2.6]{HJ} and \cite[Section 3.2]{Wan23}) that $L({\bf \Psi}{\bf \Psi}')$ is a subquotient of $L({\bf \Psi})\ot L({\bf \Psi}')$.
Since we have by \cite[Corollary 4.8, Corollary 5.1]{HJ} (for untwisted types) and \cite[Lemma 4.13, Lemma 4.15]{Wan23} (for twisted types) that $L_{s,a}^{\pm} \in \mc{O}$ for all $s \in \I$ and $a \in \bk^{\times}$, we conclude that any irreducible $U_q(\bo)$-module in $\mc{O}$ is a subquotient of a finite tensor product of prefundamental modules and one-dimensional modules, where a one-dimensional module has a trivial action of $e_0, e_1, \dots, e_n$.

\subsection{$U_q(\mf{b})$-module structure on $U_q^-(w_s)$}
Let $\wtd{w_s} = (i_1, i_2, \dots, i_\ell) \in R(w_s)$ be given. Set
\begin{equation} \label{eq: xz}
    \xz := F^{\rm up}(\beta_\ell).
\end{equation}
One can check that $\beta_\ell = k\delta + \theta$ for some $k \in \Z_+$ \cite[Remark 3.8]{JKP25}. In particular, if $s$ is minuscule, then $k = 0$, that is, $\beta_\ell = \theta$. It follows from \eqref{eq:LS formula} that $U_q^-(w_s)$ is closed under the left (resp.~right) multiplication by $\xz$. On the other hand, it is known (e.g.~\cite[Lemma 4.9]{JKP25} and reference therein) that $U_q^-(w_s)$ is closed under the action of $e_i'$ for all $i \in I$.

\begin{prop}
For any $a \in \bk^\times$, we have a $U_q(\bo)$-module structure on $U_q^-(w_s)$ by
\begin{equation} \label{eq: left action}
    k_i(u) = 
    \begin{cases}
        q^{(\alpha_i, {\rm wt}(u))} u & \text{if $i \in \I$,} \\
        q^{-(\theta, {\rm wt}(u))}u & \text{if $i = 0$,}
    \end{cases}
    \qquad
    e_i(u) = 
    \begin{cases}
        e_i'(u) & \text{if $i \in \I$,} \\
        a \xz u & \text{if $i = 0$,}
    \end{cases}
\end{equation}
for a homogeneous element $u \in U_q^-(w_s)$. Moreover $U_q^-(w_s)$ belongs to $\mc{O}$ as a $U_q(\bo)$-module.
\end{prop}
\begin{proof}
Let $\ms{B}_q(\g)$ be the $q$-boson algebra generated by $e_i'$ and $f_i$ for $i \in I$ with the defining relations \cite[Section 3.3]{Kas91a}.
By \cite[Lemma 2.6]{JKP25}, the canonical projection $\iota : U_q^-(\g) \rightarrow M$ is an isomorphism of $\ms{B}_q(\g)$-modules, where $M$ is given in \cite[Section 2.2]{JKP25}.
By \cite[Proposition 4.12]{JKP25} and Proposition \ref{prop: properties of fundamental translations}, we have
\begin{equation*}
    \iota(e_0(u)) = \iota(a\xz u) = a \ov{\xz u} = -\frac{a}{(q_{\tau(0)}-q_{\tau(0)}^{-1})q_{\tau(0)}^2} {\bf x}_0 \ov{u} = {\bf e}_0 \iota(u),
\end{equation*}
where ${\bf e}_0$ is the $0$-action on $M$ as in \cite[(4.12)]{JKP25} up to the scalar multiplication. This proves our assertion by \cite[Theorem 4.14]{JKP25}, where $M_s = \iota(U_q^-(w_s))$.
\end{proof}

\section{Character formula of prefundamental modules} \label{sec:character formula of Lsa}
\subsection{Character folding map}
Let us review a (conjectural) relationship between prefundamental modules over untwisted and twisted quantum loop algebras in terms of character homomorphisms \cite{Wan23}.
\smallskip

\emph{In the rest of this paper, we often use the superscript $\sigma$ on the notations for $\bo^\sigma$ to distinguish them from those for $\bo$ and avoid confusion.}
\smallskip

For a $U_q(\bo)$-module $V$ in $\mc{O}$, it has a weight space decomposition with respect to $\left\{\, k_i \,|\, i \in \I\,\right\}$
\begin{equation*}
	V = \bigoplus_{\omega \in {\mf t}^*} V_{\omega},
\end{equation*}
where $V_{\omega}$ is given as in \eqref{eq:weight space wrt t}.
For $\omega \in {\mf t}^*$, we write $V_\beta = V_\omega$ for $\beta \in \mr{\sfP}$ such that $q^{\beta} = \omega$.
Let $\Z[\![ e^\beta ]\!]_{\beta \in \mr{\sfP}}$ be the ring of formal power series in variables $e^\beta$ for $\beta \in \mr{\sfP}$ with the multiplication $e^\beta \cdot e^\gamma = e^{\beta+\gamma}$.

Let $K(\mc{O})$ be the Grothendieck ring of the category $\mc{O}$ (cf.~\cite[Section 9]{Kac}).
For $[V] \in K(\mc{O})$, we define
\begin{equation*}
	{\rm ch}\left( [V] \right) = \sum_{\beta \in \mr{\sfP}} \left( \dim V_\beta \right) e^{\beta} \in \Z[\![ e^{\pm \ov{\La}_i} ]\!]_{i \in \I},
\end{equation*}
which is called the (ordinary) character of $V$. We write ${\rm ch}(V) = {\rm ch}([V])$ simply. 
Put $y_i = e^{\ov{\La}_i}$ (resp.~$z_i = e^{\ov{\La}_i}$) for $\bo$ (resp.~$\bo^\sigma$) for $i \in \I$. Consider the surjective map
\begin{equation*}
    \pi : 
    \xymatrix@R=0em @C=3em{
        \Z[\![ y_i^{\pm 1} ]\!]_{i \in I_\fin} \ar@{->}[r] & \Z[\![ z_i^{\pm 1} ]\!]_{i \in I_\fin^\sigma} %\\
        %y_i \ar@{|->}[r] & z_{\ov{i}}
    }
\end{equation*}
sending $y_i$ to $z_{\ov{i}}$. Note that $\pi\left( e^{\pm \alpha_i} \right) = e^{\pm  \alpha_{\ov{i}}}$ for $i \in I_\fin$.

Let $\mc{O}_1$ be the full subcategory of $\mc{O}$ consisting of $V$'s such that the $\ell$-weights ${\bf \Psi} = (\Psi_i(z))_{i \in \I}$ of $V$ satisfy the roots and poles of $\Psi_i(z)$ are in $q^\Z \omega^\Z$ for all $i \in \I$. Similarly, we define the full subcategory $\mc{O}_1^\sigma$ for the category $\mc{O}^\sigma$ of $U_q(\bo^\sigma)$.
Then it is known in \cite{Wan23} that there exists a ring homomorphism $\ov{\pi} : K(\mc{O}_1) \longrightarrow K(\mc{O}_1^\sigma)$ such that the following diagram commutes (cf.~\cite[Corollary 4.16]{H10a}):
\begin{equation*}
    \xymatrix@R=2.5em @C=5em{
        K(\mc{O}_1) \ar@{->}[r]^{{\rm ch}\quad\,} \ar@{->}[d]_{\ov{\pi}} & \Z[\![ y_i^{\pm 1} ]\!]_{i \in I_\fin} \ar@{->}[d]^{\pi} \\
        K(\mc{O}_1^\sigma) \ar@{->}[r]_{{\rm ch}^\sigma \quad\,} & \Z[\![ z_i^{\pm 1} ]\!]_{i \in I_\fin^\sigma}
    }
\end{equation*}

\begin{thm} \label{thm: character folding}
For $s \in \I$ and $a \in \bk^\times$, we have 
\begin{equation*}
{\rm ch}^\sigma \big( L_{\ov{s},a}^\pm \big) = \pi \left( {\rm ch}\left( L_{s,a}^\pm \right) \right),
\end{equation*}
where $\left[ L_{s,a}^\pm \right] \in K(\mc{O}_1)$ so that $[ L_{\ov{s},a}^\pm ] \in K(\mc{O}_1^\sigma)$.
\end{thm}
\begin{proof}
    It follows from \cite[Theorem 4.24]{Wan23}.
\end{proof}

\begin{rem}
{\em 
The above theorem is established in terms of the (twisted) $q$-character for $\mc{O}$, which can be viewed as an analogue of \cite[Theorem 4.15]{H10a} to the category $\mc{O}$ (see \cite{HJ, Wan23} for the notion of the (twisted) $q$-character for $\mc{O}$).
More generally, it was conjectured in \cite{H10a} that the folding map $\pi$ relates certain finite-dimensional simple $U_q(\ms{L}\g)$-modules to certain finite-dimensional simple  $U_q(\ms{L}\g^\sigma)$-modules in terms of $q$-characters, and this conjecture is further formulated in \cite{Wan23} for the categories $\mc{O}$ and $\mc{O}^\sigma$. Recently, the conjecture in \cite{H10a} is proved independently in \cite{FuQi26} for type $A$ and $D$, and in \cite{NeWa26} for the general case.
}
\end{rem}

\begin{cor}
For $s \in \I$ and $a \in \bk^\times$, we have ${\rm ch}^\sigma ( L_{\ov{s},a}^- ) = {\rm ch}^\sigma ( L_{\ov{s},a}^+ )$.
\end{cor}
\begin{proof}
It is known in \cite{HJ} that ${\rm ch}\left( L_{s,a}^- \right) = {\rm ch}\left( L_{s,a}^+ \right)$ for any $a \in \bk^\times$, and $L_{s,a}^\pm$ can be obtained from $L_{s,1}^\pm$ by the algebra automorphism $\tau_a : U_q(\bo) \rightarrow U_q(\bo)$ such that $\tau_a(x_{i,m}^\pm) = a^{\pm m} x_{i,m}^\pm$ and $\tau_a(\psi_i(z)) = \psi_i(az)$. This completes the proof by Theorem \ref{thm: character folding}.
\end{proof}

\subsection{Character formula of prefundamental modules}
Let us recall that $r$ is the integer, which is the superscript of $X_N^{(r)}$.
For $\beta \in \ov{\sfDel_+(t_{-\lambda_s})}$, set $\gamma_\beta \coloneq \gamma_{\beta'}$ if $\beta = \ov{\beta'}$ for $\beta' \in {\sfDel}_+(t_{-\lambda_s})$, and 
\begin{equation} \label{eq: def of xis}
	\xi_s(\beta) = 
    \begin{cases}
        \displaystyle \frac{d_s}{\gamma_\beta} & \text{if $\g$ is not of type $A_{2n}^{(2)}$,} \\
        \vphantom{\displaystyle \frac{1}{2}} \hskip 0.5mm 1 & \text{if $\g$ is of type $A_{2n}^{(2)}$ and $\beta' \in {\sfDel}_+(t_{-\lambda_s})_1$,} \\
        \displaystyle \frac{1}{2} & \text{if $\g$ is of type $A_{2n}^{(2)}$ and $\beta' \in {\sfDel}_+(t_{-\lambda_s})_2$,}
    \end{cases}
\end{equation}
where ${\sfDel}_+(t_{-\lambda_s})_1$ (resp.~${\sfDel}_+(t_{-\lambda_s})_2$) is the subset of ${\sfDel}_+(t_{-\lambda_s})$ containing the positive roots of the form $\alpha+k\delta$ (resp.~$2\alpha + (2k+1)\delta$) for $\alpha \in \mr{\sfDel}_+$ (resp.~$\alpha \in (\mr{\sfDel}_+)_s$) and $k \in \Z_+$.

\begin{thm} \label{thm: characters}
For $s \in \I$ and $a \in \bk^\times$, we have 
\begin{equation} \label{eq: product formula}
    {\rm ch}\left( U_q^-(w_s) \right) = {\rm ch}\left( L_{s,a}^\pm  \right) = 
    \begin{cases}
        \displaystyle 
        \prod_{\beta \in \ov{\sfDel_+(t_{-\lambda_s})}} \left( \frac{1}{1-e^{-{\beta}}} \right)^{[\beta]_s} & \text{if $r = 1$,} \\
        \displaystyle 
        \prod_{\beta \in \ov{\sfDel_+(t_{-\lambda_s})}} \left( \frac{1}{1-e^{-\beta}} \right)^{\xi_s(\beta)\,[\beta]_s} & \text{if $r > 1$,}
    \end{cases}
\end{equation}
where $[\beta]_s$ is the coefficient of the simple root $\alpha_s$ in $\beta$.
\end{thm}
\begin{proof}
By definition of $\lambda_s$, we have
\begin{equation*}
    (\lambda_s, \alpha) = 
    \begin{cases}
        [\alpha]_s & \text{if $\g$ is of untwisted affine type or } A_{2n}^{(2)}, \\
        d_s [\alpha]_s & \text{otherwise.}
    \end{cases}
\end{equation*}
By Proposition \ref{prop: properties of fundamental translations}, we have \eqref{eq: product formula} for ${\rm ch}\left( U_q^-(w_s) \right)$.
For $r = 1$, \eqref{eq: product formula} for ${\rm ch}\left( L_{s,a}^\pm \right)$ follows from \cite{HJ, Nao13, LeeCH, Neg26} (cf.~Remark \ref{rem:character formula}).

Assume $r > 1$. Let us consider
\begin{equation*}
    \Pi : 
    \xymatrix@R=0em @C=5em{
    \sfQ \ar@{->}[r] & \sfQ^\sigma \\
    \alpha_i \ar@{|->}[r]  & \alpha_{\ov{i}}
    }
\end{equation*}
where $\sfQ$ (resp.~$\sfQ^\sigma$) is the root lattice for $\g_{\fin}$ (resp.~$\g_{\fin}^\sigma$).
By Proposition \ref{prop: properties of fundamental translations}, the map $\Pi$ induces the surjective map from $\ov{\sfDel_+(t_{-\lambda_s})}$ onto $\ov{\sfDel_+^\sigma(t_{-\lambda_{\ov{s}}})}$.
Furthermore, we have
\begin{align} 
& \sum_{\beta' \in \Pi^{-1}(\beta)} [\beta']_s = \xi_s(\beta) [\beta]_s. \label{eq: our claim in char mul}
\end{align}
By Theorem \ref{thm: character folding}, we deduce the formula \eqref{eq: product formula} of ${\rm ch}^\sigma ( L_{s,a}^\pm )$ from \eqref{eq: our claim in char mul}, since 
\begin{equation*}
\begin{split}
	{\rm ch}^\sigma( L_{s,a}^\pm ) &\overset{\,\,(\ast)\,\,}{=} \pi \left( {\rm ch} (L_{s,a}^\pm) \right) \\
	&\overset{\text{\eqref{eq: product formula}}}{=} \prod_{\beta \in \ov{\sfDel_+(t_{-\lambda_s})}} \left( \frac{1}{1-e^{-\Pi(\beta)}} \right)^{[\beta]_s} \\
	&=  \prod_{\beta \in \ov{\sfDel_+^\sigma(t_{-\lambda_s})}} \left( \frac{1}{1-e^{-\beta}} \right)^{\sum_{\beta' \in \Pi^{-1}(\beta)} [\beta']_s} \\
	&\overset{\text{\eqref{eq: our claim in char mul}}}{=} \prod_{\beta \in \ov{\sfDel_+^\sigma(t_{-\lambda_s})}} \left( \frac{1}{1-e^{-\beta}} \right)^{\xi_s(\beta)[\beta]_s}
\end{split}
\end{equation*}
where $(\ast)$ holds by Theorem \ref{thm: character folding}, and the second equality follows from the formula \eqref{eq: product formula} for $r = 1$. Now we verify \eqref{eq: our claim in char mul} as follows.
\smallskip 

{\it Case 1}. Suppose $X_N^{(r)} = A_{2n}^{(2)}$. Note that $\g_\fin$ is of type $A_{2n}$ and $\cg = \g_\fin^\sigma$ is of type $B_n$. Each positive root for $\g_\fin^\sigma$ is one of the following forms
    \begin{align}
        & \sum_{j \le k \le n} \alpha_k \quad (1 \le j \le n), \label{eq: positive short roots of Bn} \\
        &\sum_{j \le k < l} \alpha_k \quad\, (1 \le j < l \le n), \label{eq: positive long roots of Bn 1} \\
        & \sum_{j \le k < l} \alpha_k + 2\sum_{l \le k  \le n} \alpha_k \quad (1 \le j < l \le n), \label{eq: positive long roots of Bn 2}
    \end{align}
    where the positive roots in \eqref{eq: positive short roots of Bn} are short and the others in \eqref{eq: positive long roots of Bn 1} and \eqref{eq: positive long roots of Bn 2} are long.
    Note that each root in $\ov{\sfDel_+(t_{-\lambda_s})}$ is of the form $\sum_{j \le k \le l} \alpha_k$ for $1 \le j \le l \le 2n$. Then we have
    \begin{align*}
        \left| \Pi^{-1} \left( \beta \right) \right| = 
        \begin{cases}
            2 & \text{if } \beta \text{ is \eqref{eq: positive long roots of Bn 2}} \text{ for } l \le s, \\
            1 & \text{otherwise,}
        \end{cases}
        & \qquad 
        \xi_s(\beta) = 
        \begin{cases}
            \displaystyle \frac{1}{2} & \text{if } \beta = 2\alpha \text{ for some } \alpha \in (\mr{\sfDel}_+)_s, \vphantom{\displaystyle \sum_{j \le k < l}} \\
            \vphantom{\displaystyle \frac{1}{2}} 1 & \text{otherwise}. \\
        \end{cases}
    \end{align*}
    Hence we have \eqref{eq: our claim in char mul} in this case.
    \smallskip

    {\it Case 2}. Suppose $X_N^{(r)} = A_{2n-1}^{(2)}$. Note that $\g_\fin$ is of type $A_{2n-1}$ and $\cg = \g_\fin^\sigma$ is of type $C_n$. 
    Each positive root for $\g_\fin^\sigma$ is one of the following forms
    \begin{align}
        & \sum_{j \le k < l} \alpha_k \quad (1 \le j < l \le n), \label{eq: positive short roots of Cn 1} \\
        & \sum_{j \le k < l} \alpha_k + 2\sum_{l \le k  < n} \alpha_k + \alpha_n \quad (1 \le j < l \le n), \label{eq: positive short roots of Cn 2} \\
        & \sum_{j \le k < l} 2\alpha_k + \alpha_n \quad\, (1 \le j \le n), \label{eq: positive long roots of Cn}
    \end{align}
    where the positive roots in \eqref{eq: positive long roots of Cn} are long and the others in \eqref{eq: positive short roots of Cn 1} and \eqref{eq: positive short roots of Cn 2} are short. 
    Since each root in $\ov{\sfDel_+(t_{-\lambda_s})}$ is of the form $\sum_{j \le k \le l} \alpha_k$ for $1 \le j \le l \le 2n-1$, we have
    \begin{equation*}
        \xi_s(\beta) = 
        \begin{cases}
            1 & \text{if $s \neq n$ and $\beta$ is short,} \\
            1/2 & \text{if $s \neq n$ and $\beta$ is long,} \\
            2 & \text{if $s = n$ and $\beta$ is short,} \\
            1 & \text{if $s = n$ and $\beta$ is long,}
        \end{cases}
        \qquad 
        \left| \Pi^{-1}(\beta) \right|
        =
        \begin{cases}
             [\beta]_s & \text{if $s \neq n$ and $\beta$ is short,} \\
             1 & \text{if $s \neq n$ and $\beta$ is long,} \\
             2 & \text{if $s = n$ and $\beta$ is short,} \\
             1 & \text{if $s = n$ and $\beta$ is long,}
        \end{cases}
    \end{equation*}
    so that $[\beta']_s = [\beta'']_s$ for $\beta', \beta'' \in \Pi^{-1}(\beta)$, and $\left| \Pi^{-1} \left( \beta \right) \right| = \xi_s(\beta)\,[\beta']_s$.
    \smallskip

    {\it Case 3}. Suppose $X_N^{(2)} = D_{n+1}^{(2)}$. Note that $\g_\fin$ is of type $D_{n+1}$ and $\cg = \g_\fin^\sigma$ is of type $B_n$.  Recall that the positive roots for $\g_\fin^\sigma$ are given in \eqref{eq: positive short roots of Bn}--\eqref{eq: positive long roots of Bn 2}. The positive roots for $\g_\fin$ are given as follows:
    \begin{equation} \label{eq: positive roots of Dn}
    \begin{split}
        & \sum_{j \le k < l} \alpha_k \quad (1 \le j < l \le n+1), \\
        & \!\!\! \sum_{j \le k \le n-1} \alpha_k + \alpha_{n+1} \quad (1 \le j < n+1), \\
        & \sum_{j \le k < l} \alpha_k + 2 \sum_{l \le k < n} \alpha_k + \alpha_n + \alpha_{n+1} \quad (1 \le j < l < n+1).
    \end{split}
    \end{equation}
    Then we have
    \begin{equation*}
        \xi_s(\beta) = 
        \begin{cases}
            2 & \text{if $s \neq n$ and $\beta$ is short,} \\
            1 & \text{if $s \neq n$ and $\beta$ is long,} \\
            1 & \text{if $s = n$ and $\beta$ is short,} \\
            1/2 & \text{if $s = n$ and $\beta$ is long,}
        \end{cases}
        \qquad 
        \left| \Pi^{-1}(\beta) \right|
        =
        \begin{cases}
             2 & \text{if $s \neq n$ and $\beta$ is short,} \\
             1 & \text{if $s \neq n$ and $\beta$ is long,} \\
             1 & \text{if $s = n$ and $\beta$ is short,} \\
             1 & \text{if $s = n$ and $\beta$ is long.}
        \end{cases}
    \end{equation*}
    If $\beta \in \ov{\sfDel_+^\sigma(t_{-\lambda_s})}$ is long and $[\beta]_s = 2$, then we have two cases by \eqref{eq: positive short roots of Bn}--\eqref{eq: positive long roots of Bn 2} and \eqref{eq: positive roots of Dn}:
    \begin{enumerate}
        \item $\xi_s(\beta) = 1$ and $\Pi^{-1}(\beta) = \{ \beta' \}$, where $[\beta']_s = 2$, or 
        \item $\xi_s(\beta) = \frac{1}{2}$ and $\Pi^{-1}(\beta) = \{ \beta' \}$, where $[\beta']_s = 1$.
    \end{enumerate}
    In any case, we have \eqref{eq: our claim in char mul}. If $s \neq n$ and $\beta$ is short, one can check that $\xi_s(\beta) = 2$ and $[\beta']_s = [\beta'']_s = 1$ for $\beta', \beta'' \in \Pi^{-1}(\beta)$, so \eqref{eq: our claim in char mul} also holds in this case. The other cases follow by a similar argument.
    \smallskip

    {\it Case 4}. Suppose $X_N^{(r)} = E_6^{(2)}$. Note that $\g_\fin$ is of type $E_6$ and $\cg = \g_\fin^\sigma$ is of type $F_4^\dagger$, where there are $36$ (resp.~$24$) positive roots for $\g_\fin$ (resp.~$\g_\fin^\sigma$).
    For $\g_\fin$, we denote a positive root $a\alpha_1 + b\alpha_2 + c\alpha_3 + d \alpha_4 + e \alpha_5 + f \alpha_6$ by $\substack{\tiny f \\ abcde}$. We also use the similar notation for the positive roots of $\g_\fin^\sigma$. Recall Remark \ref{rem: convention-1; notation for Borel}.
    Then $\ov{\sfDel_+(t_{-\lambda_s})}$ and $\ov{\sfDel_+^\sigma(t_{-\lambda_s})}$ are given as follows, which enable us to verify \eqref{eq: our claim in char mul} in this case.

    \begin{enumerate}[(1)]
        \item If $s = 1$, then we have the following, where we place the roots of $\ov{\sfDel_+(t_{-\lambda_1})}$ on the top, and then those of $\ov{\sfDel_+^\sigma(t_{-\lambda_1})}$ at the bottom (we use this convention for other cases).
        \begin{align*}
        \xymatrix@C=-0.5em @R=1em{
            \rootesix{1}{0}{0}{0}{0}{0} \ar@{|->}[d] & \rootesix{1}{1}{0}{0}{0}{0} \ar@{|->}[d] & \rootesix{1}{1}{1}{0}{0}{0} \ar@{|->}[d] & \rootesix{1}{1}{1}{1}{0}{0} \ar@{|->}[d] & \rootesix{1}{1}{1}{1}{1}{0} \ar@{|->}[d] & \rootesix{1}{1}{1}{0}{0}{1} \ar@{|->}[d] & \rootesix{1}{1}{1}{1}{0}{1} \ar@{|->}[d] & \rootesix{1}{1}{2}{1}{0}{1} \ar@{|->}[d] & \rootesix{1}{2}{2}{1}{0}{1} \ar@{|->}[d] & \rootesix{1}{1}{1}{1}{1}{1} \ar@{|->}[d] & \rootesix{1}{1}{2}{1}{1}{1} \ar@{|->}[d] & \rootesix{1}{2}{2}{1}{1}{1} \ar@{|->}[d] & \rootesix{1}{1}{2}{2}{1}{1} \ar@{|->}[d] & \rootesix{1}{2}{2}{2}{1}{1} \ar@{|->}[d] & \rootesix{1}{2}{3}{2}{1}{1} \ar@{|->}[d] &  \rootesix{1}{2}{3}{2}{1}{2} \ar@{|->}[d] \\
            \rootf{1}{0}{0}{0} & \rootf{1}{1}{0}{0} & \rootf{1}{1}{1}{0} & \rootf{1}{2}{1}{0} & \blue{\rootf{2}{2}{1}{0}} & \rootf{1}{1}{1}{1} & \rootf{1}{2}{1}{1} & \rootf{1}{2}{2}{1} & \rootf{1}{3}{2}{1} & \blue{\rootf{2}{2}{1}{1}} & \blue{\rootf{2}{2}{2}{1}} & \rootf{2}{3}{2}{1} & \rootf{2}{3}{2}{1} & \blue{\rootf{2}{4}{2}{1}} & \blue{\rootf{2}{4}{3}{1}} & \blue{\rootf{2}{4}{3}{2}} \\
        }
        \end{align*}
        where the roots in blue are long and $|\Pi^{-1}(\beta)| = 1$ for roots $\beta$ except for $2321$.
        \smallskip

        \item If $s = 2$, then we have the following.
        \begin{align*}
            &
            \xymatrix@C=-0.3em @R=1em{
             \rootesix{0}{1}{0}{0}{0}{0} \ar@{|->}[d] & \rootesix{1}{1}{0}{0}{0}{0} \ar@{|->}[d] & \rootesix{0}{1}{1}{0}{0}{0} \ar@{|->}[d] & \rootesix{1}{1}{1}{0}{0}{0} \ar@{|->}[d] & \rootesix{0}{1}{1}{1}{0}{0} \ar@{|->}[d] & \rootesix{1}{1}{1}{1}{0}{0} \ar@{|->}[d] & \rootesix{0}{1}{1}{1}{1}{0} \ar@{|->}[d] & \rootesix{1}{1}{1}{1}{1}{0} \ar@{|->}[d] & \rootesix{1}{2}{2}{1}{0}{1} \ar@{|->}[d] & \rootesix{1}{2}{2}{1}{1}{1} \ar@{|->}[d] & \rootesix{1}{2}{2}{2}{1}{1} \ar@{|->}[d] & \rootesix{1}{2}{3}{2}{1}{1} \ar@{|->}[d] & \rootesix{0}{1}{1}{0}{0}{1} \ar@{|->}[d] & \rootesix{0}{1}{1}{1}{0}{1} \ar@{|->}[d] & \rootesix{0}{1}{2}{1}{0}{1} \ar@{|->}[d] \\
             \rootf{0}{1}{0}{0} & \rootf{1}{1}{0}{0} & \rootf{0}{1}{1}{0} & \rootf{1}{1}{1}{0} & \blue{\rootf{0}{2}{1}{0}} & \rootf{1}{2}{1}{0} & \rootf{1}{2}{1}{0} & \blue{\rootf{2}{2}{1}{0}} & \rootf{1}{3}{2}{1} & \rootf{2}{3}{2}{1} & \blue{\rootf{2}{4}{2}{1}} & \blue{\rootf{2}{4}{3}{1}} & \rootf{0}{1}{1}{1} & \blue{\rootf{0}{2}{1}{1}} & \blue{\rootf{0}{2}{2}{1}} \\
            }
            \\
            &
            \xymatrix@C=-0.3em @R=1em{
             \rootesix{1}{1}{1}{0}{0}{1} \ar@{|->}[d] & \rootesix{1}{1}{1}{1}{0}{1} \ar@{|->}[d] & \rootesix{1}{1}{2}{1}{0}{1} \ar@{|->}[d] & \rootesix{1}{2}{3}{2}{1}{2} \ar@{|->}[d] & \rootesix{0}{1}{1}{1}{1}{1} \ar@{|->}[d] & \rootesix{1}{1}{1}{1}{1}{1} \ar@{|->}[d] & \rootesix{0}{1}{2}{1}{1}{1} \ar@{|->}[d] & \rootesix{1}{1}{2}{1}{1}{1} \ar@{|->}[d] & \rootesix{0}{1}{2}{2}{1}{1} \ar@{|->}[d] & \rootesix{1}{1}{2}{2}{1}{1} \ar@{|->}[d] \\
             \rootf{1}{1}{1}{1} & \rootf{1}{2}{1}{1} & \rootf{1}{2}{2}{1} & \blue{\rootf{2}{4}{3}{2}} & \rootf{1}{2}{1}{1} & \blue{\rootf{2}{2}{1}{1}} & \rootf{1}{2}{2}{1} & \blue{\rootf{2}{2}{2}{1}} & \rootf{1}{3}{2}{1} & \rootf{2}{3}{2}{1}
            }
        \end{align*}
        where $|\Pi^{-1}(\beta)| = 1$ for roots $\beta$ except for $1210, 1321, 2321, 1211, 1221$.
        In particular, we have 
        \begin{equation*}
            \xi_2(2321) = 1, \quad \Pi^{-1}(2321) = \{ \rootesix{1}{2}{2}{1}{1}{1}, \rootesix{1}{1}{2}{2}{1}{1} \}
        \end{equation*}
        so that $\sum_{\beta' \in \Pi^{-1}(2321)}[\beta']_2 = 3$ and $\xi_2(2321) [2321]_2 = 3$.
        \smallskip

        \item If $s = 3$, then we have the following.
        \begin{align*}
            &
            \xymatrix@C=-0.3em @R=1em{
            \rootesix{0}{0}{1}{0}{0}{0} \ar@{|->}[d] & \rootesix{0}{1}{1}{0}{0}{0} \ar@{|->}[d] & \rootesix{1}{1}{1}{0}{0}{0} \ar@{|->}[d] & \rootesix{0}{0}{1}{1}{0}{0} \ar@{|->}[d] & \rootesix{0}{1}{1}{1}{0}{0} \ar@{|->}[d] & \rootesix{1}{1}{1}{1}{0}{0} \ar@{|->}[d] & \rootesix{0}{0}{1}{1}{1}{0} \ar@{|->}[d] & \rootesix{0}{1}{1}{1}{1}{0} \ar@{|->}[d] & \rootesix{1}{1}{1}{1}{1}{0} \ar@{|->}[d] & \rootesix{1}{2}{3}{2}{1}{1} \ar@{|->}[d] & \rootesix{0}{1}{2}{1}{0}{1} \ar@{|->}[d] & \rootesix{0}{1}{2}{1}{1}{1} \ar@{|->}[d] & \rootesix{0}{1}{2}{2}{1}{1} \ar@{|->}[d] & \rootesix{1}{1}{2}{1}{0}{1} \ar@{|->}[d] & \rootesix{1}{1}{2}{1}{1}{1} \ar@{|->}[d] \\
            \blue{\rootf{0}{0}{1}{0}} & \rootf{0}{1}{1}{0} & \rootf{1}{1}{1}{0} & \rootf{0}{1}{1}{0} & \blue{\rootf{0}{2}{1}{0}} & \rootf{1}{2}{1}{0} & \rootf{1}{1}{1}{0} & \rootf{1}{2}{1}{0} & \blue{\rootf{2}{2}{1}{0}} & \blue{\rootf{2}{4}{3}{1}} & \blue{\rootf{0}{2}{2}{1}} & \rootf{1}{2}{2}{1} & \rootf{1}{3}{2}{1} & \rootf{1}{2}{2}{1} & \blue{\rootf{2}{2}{2}{1}}
            } \\
            &
            \xymatrix@C=-0.3em @R=1em{
            \rootesix{1}{1}{2}{2}{1}{1} \ar@{|->}[d] & \rootesix{1}{2}{2}{1}{0}{1} \ar@{|->}[d] & \rootesix{1}{2}{2}{1}{1}{1} \ar@{|->}[d] & \rootesix{1}{2}{2}{2}{1}{1} \ar@{|->}[d] & \rootesix{1}{2}{3}{2}{1}{2} \ar@{|->}[d] & \rootesix{0}{0}{1}{0}{0}{1} \ar@{|->}[d] & \rootesix{0}{1}{1}{0}{0}{1} \ar@{|->}[d] & \rootesix{1}{1}{1}{0}{0}{1} \ar@{|->}[d] & \rootesix{0}{0}{1}{1}{0}{1} \ar@{|->}[d] & \rootesix{0}{1}{1}{1}{0}{1} \ar@{|->}[d] & \rootesix{1}{1}{1}{1}{0}{1} \ar@{|->}[d] & \rootesix{0}{0}{1}{1}{1}{1} \ar@{|->}[d] & \rootesix{0}{1}{1}{1}{1}{1} \ar@{|->}[d] & \rootesix{1}{1}{1}{1}{1}{1} \ar@{|->}[d] \\
            \rootf{2}{3}{2}{1} & \rootf{1}{3}{2}{1} & \rootf{2}{3}{2}{1} & \blue{\rootf{2}{4}{2}{1}} & \blue{\rootf{2}{4}{3}{2}} & \blue{\rootf{0}{0}{1}{1}} & \rootf{0}{1}{1}{1} & \rootf{1}{1}{1}{1} & \rootf{0}{1}{1}{1} & \blue{\rootf{0}{2}{1}{1}} & \rootf{1}{2}{1}{1} & \rootf{1}{1}{1}{1} & \rootf{1}{2}{1}{1} & \blue{\rootf{2}{2}{1}{1}}
            }
        \end{align*}
        Note that $|\Pi^{-1}(\beta)| = 2$ for all short roots, while $|\Pi^{-1}(\beta)| = 1$ for all long roots. If $\beta$ is a short root with $[\beta]_3 = m$, then we have $\xi_3(\beta) = 2$ and $[\beta']_3 = m$ for $\beta' \in \Pi^{-1}(\beta)$. If $\beta$ is a long root with $[\beta]_3 = m$, then $\xi_3(\beta) = 1$ and $[\beta']_3 = m$, where $\Pi^{-1}(\beta) = \{ \beta' \}$.
        \smallskip

        \item If $s = 6$\footnote{Recall that we denote $\ov{6}$ by $4$ in $I_\fin^\sigma$}, then we have the following. 
        \begin{align*}
            &
            \xymatrix@C=-0.3em @R=1em{
            \rootesix{0}{0}{0}{0}{0}{1} \ar@{|->}[d] & \rootesix{0}{0}{1}{0}{0}{1} \ar@{|->}[d] & \rootesix{0}{1}{1}{0}{0}{1} \ar@{|->}[d] & \rootesix{1}{1}{1}{0}{0}{1} \ar@{|->}[d] & \rootesix{0}{0}{1}{1}{0}{1} \ar@{|->}[d] & \rootesix{0}{1}{1}{1}{0}{1} \ar@{|->}[d] & \rootesix{1}{1}{1}{1}{0}{1} \ar@{|->}[d] & \rootesix{0}{0}{1}{1}{1}{1} \ar@{|->}[d] & \rootesix{0}{1}{1}{1}{1}{1} \ar@{|->}[d] & \rootesix{1}{1}{1}{1}{1}{1} \ar@{|->}[d] & \rootesix{1}{2}{3}{2}{1}{2} \ar@{|->}[d] & \rootesix{0}{1}{2}{1}{0}{1} \ar@{|->}[d] & \rootesix{0}{1}{2}{1}{1}{1} \ar@{|->}[d] & \rootesix{0}{1}{2}{2}{1}{1} \ar@{|->}[d] & \rootesix{1}{1}{2}{1}{0}{1} \ar@{|->}[d] \\
            \blue{\rootf{0}{0}{0}{1}} & \blue{\rootf{0}{0}{1}{1}} & \rootf{0}{1}{1}{1} & \rootf{1}{1}{1}{1} & \rootf{0}{1}{1}{1} & \blue{\rootf{0}{2}{1}{1}} & \rootf{1}{2}{1}{1} & \rootf{1}{1}{1}{1} & \rootf{1}{2}{1}{1} & \blue{\rootf{2}{2}{1}{1}} & \blue{\rootf{2}{4}{3}{2}} & \blue{\rootf{0}{2}{2}{1}} & \rootf{1}{2}{2}{1} & \rootf{1}{3}{2}{1} & \rootf{1}{2}{2}{1} 
            } \\
            &
            \xymatrix@C=-0.3em @R=1em{
            \rootesix{1}{1}{2}{1}{1}{1} \ar@{|->}[d] & \rootesix{1}{1}{2}{2}{1}{1} \ar@{|->}[d] & \rootesix{1}{2}{2}{1}{0}{1} \ar@{|->}[d] & \rootesix{1}{2}{2}{1}{1}{1} \ar@{|->}[d] & \rootesix{1}{2}{2}{2}{1}{1} \ar@{|->}[d] & \rootesix{1}{2}{3}{2}{1}{1} \ar@{|->}[d] \\
            \blue{\rootf{2}{2}{2}{1}} & \rootf{2}{3}{2}{1} & \rootf{1}{3}{2}{1} & \rootf{2}{3}{2}{1} & \blue{\rootf{2}{4}{2}{1}} & \blue{\rootf{2}{4}{3}{1}}
            }
        \end{align*}
        Note that $|\Pi^{-1}(\beta)| = 2$ for all short roots, while $|\Pi^{-1}(\beta)| = 1$ for all long roots. As in (3), we have $\xi_4(\beta) = 2$ for a short root $\beta$, and $\xi_4(\beta) = 1$ otherwise. Also, $[\beta']_4 = [\beta]_4$ for $\beta' \in \Pi^{-1}(\beta)$.
    \end{enumerate}
    %By (1)--(4), we have \eqref{eq: our claim in char mul} in this case.
    \smallskip

    {\it Case 5}. Suppose $X_N^{(r)} = D_4^{(3)}$. Note that $\g_\fin$ is of type $D_4$ and $\cg = \g_\fin^\sigma$ is of type $G_2^\dagger$. Recall that the positive roots for $\g_\fin$ are given in \eqref{eq: positive roots of Dn} where $n = 3$. The set of positive roots for $\g_\fin^\sigma$ is $\{\, \alpha_1, \alpha_2, \alpha_1 + \alpha_2, 2\alpha_1 + \alpha_2, 3\alpha_1 + \alpha_2, 3\alpha_1 + 2\alpha_2\,\}$. Then $\ov{\sfDel_+(t_{-\lambda_s})}$ and $\ov{\sfDel_+^\sigma(t_{-\lambda_s})}$ are given as follows:
    \begin{equation} \label{eq: roots in D43}
    \begin{split}
        \ov{\sfDel_+(t_{-\lambda_s})} &= 
        \begin{cases}
            \{ 1000, 1100, 1110, 1101, 1111, 1211 \} & \text{if } s = 1, \\
            \{ 0100, 1100, 0110, 0101, 1110, 1101, 0111, 1111, 1211 \} & \text{if } s = 2,
        \end{cases}
        \\
        \ov{\sfDel_+^\sigma(t_{-\lambda_s})} &= 
        \begin{cases}
            \{ 10, 11, 21, \blue{31}, \blue{32} \} & \text{if } s = 1, \\
            \{ \blue{01}, 11, 21, \blue{31}, \blue{32} \} & \text{if } s = 2,
        \end{cases}
    \end{split}
    \end{equation}
    where the roots are written by a similar convention as in {\it Case 4} and the roots in blue are long. Then we have
    \begin{equation*}
        \xi_s(\beta) = 
        \begin{cases}
            1 & \text{if $s = 1$ and $\beta$ is short,} \\
            1/3 & \text{if $s = 1$ and $\beta$ is long,} \\
            3 & \text{if $s = 2$ and $\beta$ is short,} \\
            1 & \text{if $s = 2$ and $\beta$ is long,}
        \end{cases}
        \qquad 
        \left| \Pi^{-1}(\beta) \right|
        =
        \begin{cases}
            [\beta]_1 & \text{if $s = 1$ and $\beta$ is short,} \\
            1 & \text{if $s = 1$ and $\beta$ is long,} \\
            3 & \text{if $s = 2$ and $\beta$ is short,} \\
            1 & \text{if $s = 2$ and $\beta$ is long,}
        \end{cases}
    \end{equation*}
    In particular, if $\beta \in \ov{\sfDel_+^\sigma(t_{-\lambda_2})}$ is long and $[\beta]_2 = 2$, then one can check that $\Pi^{-1}(\beta) = \{ \beta' \}$ and $[\beta']_2 = 2$ by \eqref{eq: roots in D43}.
    This shows \eqref{eq: our claim in char mul} in this case.
\end{proof}

\begin{ex}
{\em 
Suppose $X_N^{(r)} = A_2^{(2)}$.
    Since
    %\begin{equation*}
    $
        \displaystyle {\rm ch}\left( L_{1,a}^{\pm} \right) = \frac{1}{1-e^{-\alpha_1}} \frac{1}{1-e^{-(\alpha_1+\alpha_2)}},
        %\displaystyle {\rm ch}_{A_2^{(1)}}\left( L_{1,a}^{\pm} \right) = \frac{1}{1-e^{-\alpha_1}} \frac{1}{1-e^{-(\alpha_1+\alpha_2)}},
    $
    %\end{equation*}
    it follows from Theorem \ref{thm: character folding} that
    \begin{equation} \label{eq: product formula in A22}
        {\rm ch}^\sigma ( L_{1,a}^{\pm} ) = \pi\left( {\rm ch} ( L_{1,a}^{\pm} ) \right) = 
        \frac{1}{1-e^{-\alpha_1}} \frac{1}{1-e^{-2\alpha_1}}.
    \end{equation}
    By Proposition \ref{prop: properties of fundamental translations}, we have $\sfDel_+(t_{-\lambda_1}) = \left\{ \alpha_1,\, 2\alpha_1 + \delta \right\}$ with $\xi_1(\alpha_1) = 1$ and $\xi_1(2\alpha_1) = \frac{1}{2}$.
     We remark that an explicit construction of $L_{1,a}^{\pm}$ for type $A_2^{(2)}$ is known in \cite{Gal15} (cf.~\cite[Appendix B]{Wan23}). The construction gives the following character formula
    \begin{equation} \label{eq: expansion in A22}
        {\rm ch}^\sigma \left( L_{1,a}^{\pm} \right) = \sum_{k = 0}^\infty \left\lfloor \frac{k+2}{2} \right\rfloor e^{-k\alpha_1}.
    \end{equation}
    By expanding \eqref{eq: product formula in A22}, we observe that the coefficient of $e^{-k\alpha_1}$ is given by $l+1$ when $k = 2l$ or $k = 2l+1$ for $l \in \Z_+$. This matches the coefficient of $e^{-k\alpha_1}$ in \eqref{eq: expansion in A22} for each $k \ge 0$.
}
\end{ex}

\begin{rem} \label{rem:character formula}
{\em
The product formula \eqref{eq: product formula} of ${\rm ch}( L_{s,a}^- )$ was conjectured in \cite{MY14}, where some cases were proved. 
Recall that ${\rm ch}( L_{s,a}^- )$ is the limit of normalized characters of Kirillov--Reshetikhin (KR) modules \raisebox{2pt}{\scalebox{0.85}{$W_{k,aq_s^{-2k+1}}^{(s)}$}} as $k \rightarrow \infty$ \cite{Na03a, Her06}, and the KR modules are known as the minimal affinizations of irreducible $U_q(\mathring{\g})$-modules $V(k\varpi_s)$ of highest weight $k\varpi_s$ (see~\cite{ChHe10} and references therein), where $\varpi_s$ is the $s$-th fundamental weight of $\mathring{\g}$. %Cha95, CP95b, CP96b, CP96c
Then \eqref{eq: product formula} was proved for type $A_n^{(1)}$, $B_n^{(1)}$ and $C_n^{(1)}$ in \cite{Nao13}, and for type $G_2^{(1)}$ in \cite{LiNa16} by analyzing graded limits of minimal affinizations (cf.~\cite{Nao14} for types $D_n^{(1)}$). 
It was further proved in \cite{LeeCH} for all untwisted types, except for certain cases of $E_8^{(1)}$, by using an algebraic relation among the limits, which may be 
obtained from the $Q$-system \cite{Na03a, Her06} (cf.~\cite{FH18}) by extracting the dominant exponential terms.
Recently, it has also been proved in \cite{Neg26} for all untwisted types by realizing $L_{s,a}^-$ in terms of quantum shuffle algebras.
}
\end{rem}

\section{Minuscule prefundamental modules}
\label{sec:minuscule prefundamental modules}

\subsection{Computation of $\ell$-weights}
Throughout this subsection, we assume that $\g$ is not of type $A_{2n}^{(2)}$ for brevity.
Let $2\rho$ be the sum of positive roots of $\mathring{\mf g}$. Take a reduced expression $t_{2\rho}=s_{i_1} \cdots s_{i_N}\in W$. 
Consider a doubly infinite sequence $\dots, i_{-2}, \,i_{-1}, \,i_0, \,i_1, \, i_2 \dots$ by $i_k := i_{k\, ({\rm mod}\, N)}$ for $k \in \mathbb{Z}$. Then we define
\begin{equation*}
	\beta_k = 
	\begin{cases} 
	s_{i_0}s_{i_{-1}}\cdots s_{i_{k+1}} (\alpha_{i_k}) & \textrm{if $k \leqq 0$}, \\
	s_{i_1}s_{i_2}\cdots s_{i_{k-1}}(\alpha_{i_k}) & \textrm{if $k > 0$},
	\end{cases}
    \qquad
    \E_{\beta_k} = 
\begin{cases}
    T_{i_0}^{-1} T_{i_{-1}}^{-1} \cdots T_{i_{k+1}}^{-1}(e_{i_k}) & \textrm{if $k \leqq 0$}, \\
    T_{i_1} T_{i_2} \cdots T_{i_{k-1}}(e_{i_k}) & \textrm{if $k > 0$}.    
\end{cases}
\end{equation*}
Similarly, we define $\F_{\beta_k}$ for $-\beta_k$ in the same way with $e_{i_k}$ replaced by $f_{i_k}$.
Note that $\E_{\beta_k} \in U_q^+ (\g)$ and $\F_{\beta_k} \in U_q^- (\g)$ \cite{Lu10}, which are called the (real) root vectors for $\beta_k$. %\cite[Proposition 40.1.3]{Lu10}.
In particular, if $\beta_k = \alpha_i$ for some $i \in \I$, then $\E_{\beta_k} = e_i$ and $\F_{\beta_k} = f_i$. %(cf.~\cite[Corollary 4.3]{MuTi18}).

\begin{lem} %{\em (\!\!\cite[Lemma 4.2]{JKP23})} 
\label{lem:independence of E}
	For $i \in I$ and $k > 0$, the root vectors $\E_{k\delta\pm\alpha_i}$ and $\F_{k\delta\pm\alpha_i}$ are independent of the choice of a reduced expression of $t_{2\rho}$.
	\qed
\end{lem}
\begin{proof}
    It follows from \cite[Corollary 4.6]{MuTi18} (cf.~\cite[Lemma 4.2]{JKP23}).
\end{proof}

Let $o : \I \,\rightarrow\, \{ \pm 1 \}$ be a map such that $o(i) = -o(j)$ whenever $a_{ij} < 0$.
It follows from \cite{Bec94a, Dam00} (cf.~\cite{Dam12,Dam15}) that
\begin{equation} \label{eq: Drinfeld generator 0}
	\psi_{i, {{\tt d}_i} k}^+ = o(i)^k (q_i-q_i^{-1})k_i
	\left( \E_{k{{\tt d}_i}\delta-\alpha_i}\E_{\alpha_i} - q_i^{-2}\E_{\alpha_i} \E_{k{{\tt d}_i}\delta-\alpha_i} \right)
\end{equation}
for $i \in \I$ and $k > 0$, where ${{\tt d}_i} \in \Z_+$ is given in Remark \ref{rem: vanishing generators}.
Note that if $s \in \I$ is a minuscule node, then ${{\tt d}}_s = 1$.

\begin{lem} \label{lem:inductive formula}
For $i \in \I$ and $k \in \mathbb{Z}_{>0}$, we have
\begin{equation*}
\begin{split}
\E_{(k+1){{\tt d}_i}\delta-\alpha_i} &= -\frac{1}{q_i+q_i^{-1}}
\Big(\,
\E_{{{\tt d}_i}\delta-\alpha_i} \E_{\alpha_i}\E_{k{{\tt d}_i}\delta-\alpha_i} - q_i^{-2} \E_{\alpha_i}\E_{{{\tt d}_i}\delta-\alpha_i}\E_{k{{\tt d}_i}\delta-\alpha_i} \\
& \qquad \qquad \quad\,\,\,- \E_{k{{\tt d}_i}\delta -\alpha_i} \E_{{{\tt d}_i}\delta-\alpha_i} \E_{\alpha_i} + q_i^{-2} \E_{k{{\tt d}_i}\delta - \alpha_i} \E_{\alpha_i} \E_{{{\tt d}_i}\delta - \alpha_i}\,
\Big)\,.
\end{split}
\end{equation*}
\end{lem}
\begin{proof}
    The proof for untwisted types can be found in \cite[Lemma 4.3]{JKP23} (cf.~\cite[Lemma 5.10]{JKP25}), while the proof for twisted types is almost identical by using the same functional equations (e.g.~see \cite[Remark 3.6]{BN04}) involving the factor ${{\tt d}_i}$.
\end{proof}

\begin{lem} \label{lem: computation of l-weights}
{\em 
Let $V$ be a $U_q(\bo)$-module in the category $\mc{O}$, and let $V_\omega$ be a weight space of dimension $1$.
Suppose that a non-zero vector $v \in V_\omega$ satisfies 
\begin{enumerate}
    \item $e_i v = 0$ for all $i \in I_\fin$, and
    \item for each $i \in I_\fin$, there exists $b_i,\, c_i \in \bk$ such that $\E_{\alpha_i} \E_{{{\tt d}_i} \delta - \alpha_i}\,v = b_i v$ and $\E_{2{{\tt d}_i}\delta-\alpha_i}\, v = c_i \E_{{{\tt d}_i} \delta - \alpha_i}\, v$.
\end{enumerate}
Then $v$ is an $\ell$-weight vector with the $\ell$-weight $(\Psi_i(z))_{i \in I_{\rm fin}}$ given by 
\begin{equation*}
    \Psi_i(z) = \omega(i) \frac{1 - (c_i + b_i(q_i^{-1} - q_i^{-3})) o(i) z^{{\tt d}_i} }{1 - o(i) c_i z^{{\tt d}_i}}.
\end{equation*}
}
\end{lem}
\begin{proof}
By Lemma \ref{lem:independence of E} and \cite{Bec94a, Dam00}, there exists $h \in U_q(\bo)^0$ such that $[h,\E_{k {\tt d}_i \delta - \alpha_i}] = \E_{(k+1){\tt d}_i \delta - \alpha_i}$ for $k > 0$, where $U_q(\bo)^0=\langle\,\psi_{i,k}^+, k_i^{\pm 1} \,\rangle_{i\in\I,\, k>0}$.
Since $\dim V_\omega = 1$, we have $h v = d v$ for some $d \in \bk$. By (2), we also have
\begin{equation*}
    h \E_{{\tt d}_i\delta - \alpha_i} v = \E_{2{\tt d}_i \delta - \alpha_i} v + \E_{{\tt d}_i \delta - \alpha_i} h v = (c_i + d) \E_{{\tt d}_i \delta - \alpha_i} v.
\end{equation*}
Now we claim that 
\begin{equation} \label{eq: our claim 1}
\E_{k{\tt d}_i \delta - \alpha_i} v = c_i^{k-1} \E_{{\tt d}_i \delta - \alpha_i} v
\end{equation}
for $ k > 0$. We prove \eqref{eq: our claim 1} by induction on $k$. 
Clearly, our claim holds for $k = 1$. Next, we compute
\begin{equation*}
\begin{split}
    \E_{(k+1){\tt d}_i\delta-\alpha_i} v 
    &= h \E_{k{\tt d}_i \delta -\alpha_i} v - \E_{k{\tt d}_i\delta - \alpha_i} h v \\
    &= c_i^{k-1} h \E_{{\tt d}_i\delta-\alpha_i} v - d c_i^{k-1} \E_{{\tt d}_i\delta-\alpha_i} v \\
    &= c_i^{k-1} (c_i+d)\E_{{\tt d}_i\delta-\alpha_i} v - d c_i^{k-1} \E_{{\tt d}_i\delta-\alpha_i} v \\
    &= c_i^k \E_{{\tt d}_i\delta - \alpha_i} v
\end{split}
\end{equation*}
as desired. By \eqref{eq: Drinfeld generator 0}, \eqref{eq: our claim 1} and (1) (cf.~Remark \ref{rem: vanishing generators}), we have 
\begin{equation*}
    \Psi_i(z) = \omega(i) - \omega(i) b_i (q_i^{-1} - q_i^{-3}) \sum_{k=1}^\infty o(i)^k c_i^{k-1} z^{{{\tt d}_i} k} = \omega(i)\frac{1-(c_i+b_i(q_i^{-1}-q_i^{-3}))o(i) z^{{\tt d}_i}}{1-o(i)c_i z^{{\tt d}_i}}.
\end{equation*}
This completes the proof.
\end{proof}

\begin{cor} \label{cor: l-weights}
{\em
    Under the same hypothesis, we have
    \begin{equation*}
        \Psi_i(z) = \begin{cases}
            \omega(i) \vphantom{\displaystyle \omega(i)\frac{1}{1-o(i)c_i z^{{\tt d}_i}}} & \text{if $b_i = 0$,} \\
            \omega(i)(1-b_i(q_i^{-1}-q_i^{-3})o(i) z^{{\tt d}_i} ) \vphantom{\displaystyle \omega(i)\frac{1}{1-o(i)c_i z^{{\tt d}_i}}} & \text{if $c_i = 0$,} \\
            \displaystyle \omega(i)\frac{1}{1-o(i)c_i z^{{\tt d}_i}} & \text{if $c_i + b_i(q_i^{-1}-q_i^{-3}) = 0$.}
        \end{cases}
    \end{equation*}
}
\end{cor}

\subsection{Realization of minuscule prefundamental modules for twisted types} 
Assume that $\sfA$ is of type $A_{2n-1}^{(2)}$ or $D_{n+1}^{(2)}$, where the Dynkin diagram corresponding to $\sfA$ is given in Appendix \ref{sec:list of dynkins}.
For types $A_{2n-1}^{(2)}$ and $D_{n+1}^{(2)}$, the minuscule nodes $s$ are $1$ and $n$, respectively. Note that ${\tt d}_s = 1$, since $\sigma(s) \neq s$ (see Remark \ref{rem: vanishing generators}).

Now we are in a position to state our main result in this paper.

\begin{thm} \label{thm: main1}
For a minuscule node $s \in \I$, we have
\begin{equation*}
    U_q^-(w_s) \cong L_{s,a\eta_s}^-
\end{equation*}
for some $\eta_s \in \bk^\times$, where $U_q^-(w_s)$ is the $U_q(\bo)$-module with respect to \eqref{eq: left action}.
\end{thm}
\begin{proof}
    By Theorem \ref{thm: characters}, it is enough to show that the $\ell$-weight $(\Psi_i(z))_{i \in I_\fin}$ of $1 \in U_q^-(w_s)$ is of the form
    \begin{equation*}
        \Psi_i(z) = 
        \begin{cases}
            \displaystyle \frac{1}{1- a\eta_s \omega^k z} & \text{if } i = \sigma^k(s), \\
            \vphantom{\displaystyle \frac{1}{1- a\eta_s \omega^k z}} \,1 & \text{otherwise,}
        \end{cases}
    \end{equation*}
    (cf.~the proof of \cite[Theorem 5.4]{JKP25}).
    If $i \neq \sigma^k(s)$ for any $k$, then it follows from Proposition \ref{prop: properties of fundamental translations} and \cite[Lemma 4.9]{JKP25} (see also \cite[Remark 4.10]{JKP25}) that $-\alpha_{\ov{i}}$ is not a weight of $U_q^-(w_s)$ so that $\E_{\delta-\alpha_{\ov{i}}} 1 = 0$. This implies that $\Psi_{\sigma^k(i)}(z) = 1$ for any $k$, by Remark \ref{rem: l-weight on orbit} and Corollary \ref{cor: l-weights}.

    Suppose $i = \sigma^k(s)$. By Remark \ref{rem: l-weight on orbit}, it is enough to compute $\Psi_s(z)$.
    To do this, we compute $\E_{\alpha_s} \E_{\delta - \alpha_s}\,1$ and $\E_{2\delta-\alpha_s}\, 1$ as follows.
    \smallskip
    
    {\it Case 1}. Suppose that $\g$ is of type $A_{2n-1}^{(2)}$. Note that $\wtd{w_1} = (1,2, \dots, n-1, n, n-1, \dots, 2, 1)$.
    By the same argument in \cite[Lemma 5.11]{JKP25}, one can check that 
    \begin{equation} \label{eq: expansion of E-del-al1}
        \E_{\delta-\alpha_1} = -q^{-2n+2} e_2 e_3 \dots e_{n-1}e_n e_{n-1} \dots e_2 e_0 + X,
    \end{equation}
    where $X$ is the sum of the remaining monomials of the form $e_{i_1} \dots e_{i_k}$ with $e_{i_k} \neq e_0$.
    By \eqref{eq: expansion of E-del-al1}, Remark \ref{rem: e' and te} and Lemma \ref{lem:dual PBW properties 1}, we have
    \begin{equation*}
        \E_{\delta-\alpha_1} 1 = -aq^{-2n+2} f_1, \quad
        \E_{\alpha_1} \E_{\delta-\alpha_1} 1 = -aq^{-2n+2} 1,
    \end{equation*}
    (cf.~Example \ref{ex: xz for twisted types}). 
    By \eqref{eq:derivation}, we also have $\E_{\alpha_1} \E_{\delta-\alpha_1}^2 1 = a^2 q^{-4n+4}(1+q^{-2}) f_1$. Then it follows from Lemma \ref{lem:inductive formula} that 
    \begin{equation*}
        \E_{2\delta-\alpha_1}1 = a(q^{-2n+1}-q^{-2n-1}) \E_{\delta-\alpha_1} 1.
    \end{equation*}
    Until now, we have seen that $1 \in U_q^-(w_1)$ satisfies the conditions (1) and (2) in Lemma \ref{lem: computation of l-weights}, where $b_1 = -aq^{-2n+2}$ and $c_1 = a(q^{-2n+1}-q^{-2n-1})$. By Corollary \ref{cor: l-weights}, we have
    \begin{equation*}
        \Psi_1(z) = \frac{1}{1-o(1)c_1z},
    \end{equation*}
    since $c_1 + b_1(q^{-1}-q^{-3}) = 0$.
    This completes the proof in this case, where we set $\eta_1 = o(1)(q^{-2n+1}-q^{-2n-1})$.
    \smallskip

    {\it Case 2}. Suppose that $\g$ is of type $D_{n+1}^{(2)}$. Note that $\wtd{w_n} = w_1 w_2 \dots w_n$, where $w_k = (n, n-1, \dots, k)$. By definition, one can check
    \begin{equation} \label{eq: expansion of E-del-aln}
        \E_{\delta-\alpha_n} = (-1)^{n-1} q^{-2n+2} e_{n-1} e_{n-2} \dots e_1 e_0 + X,
    \end{equation}
    where $X$ is the sum of the remaining monomials of the form $e_{i_1} \dots e_{i_k}$ with $e_{i_k} \neq e_0$. 
    As in the previous case, we calculate
    \begin{align*}
        \E_{\alpha_n}\E_{\delta-\alpha_n} 1 &= (-1)^{n-1} a q^{-2n+2} 1, \\
        \E_{2\delta - \alpha_n} 1 &= (-1)^{n-1} a (q^{-2n-1}-q^{-2n+1}) \E_{\delta-\alpha_n} 1,
    \end{align*}
    (cf.~Example \ref{ex: xz for twisted types}).
    By Lemma \ref{lem: computation of l-weights} and Corollary \ref{cor: l-weights}, we have
    \begin{equation*}
        \Psi_n(z) = \frac{1}{1-o(n) c_n z},
    \end{equation*}
    where $b_n = (-1)^{n-1} aq^{-2n+2}$ and $c_n = (-1)^{n-1} a(q^{-2n-1}-q^{-2n+1})$.
    This completes the proof in this case, where we set $\eta_n = (-1)^{n-1} o(n) (q^{-2n-1}-q^{-2n+1})$.
\end{proof}

\begin{ex}
{\em 
Let us continue Example \ref{ex: xz for twisted types}. 
Assume that $\g$ is of type $D_3^{(2)}$.
For simplicity, we take $o(2) = 1$.
By Theorem \ref{thm: characters}, we have 
\begin{equation*}
    {\rm ch}( U_q^-(w_2) ) = {\rm ch}( L_{2,a}^\pm ) = \frac{1}{1-e^{-\alpha_2}} \frac{1}{1-e^{-(\alpha_1+2\alpha_2)}} \frac{1}{1-e^{-(\alpha_1+\alpha_2)}}.
\end{equation*}
%Note that $\gamma_{\alpha_2} = 1$, $\gamma_{\alpha_1 + 2\alpha_2} = 2$ and $\gamma_{\alpha_1 + \alpha_2} = 1$. 
%
On the other hand, we have $\E_{\delta-\alpha_2} = e_0 e_1 - q^{-2} e_1 e_0$.
By Lemma \ref{lem:dual PBW properties 1}, we have 
\begin{equation*}
    \E_{2\delta-\alpha_2} 1 = aq^{-3}(1-q^{-2}) \E_{\delta-\alpha_2} 1, \quad
    \E_{\alpha_2} \E_{\delta-\alpha_2} 1 = -aq^{-2} 1,    
\end{equation*}
where $b_2 = -aq^{-2}$ and $c_2 = aq^{-3}(1-q^{-2})$. Note that $b_2(q^{-3}-q^{-1}) = c_2$. 
By Lemma \ref{lem: computation of l-weights} and Corollary \ref{cor: l-weights}, we have
\begin{equation*}
    \Psi_i(z) = \begin{cases}
        \displaystyle \frac{1}{1-aq^{-3}(1-q^{-2})z} & \text{if $i = 2$,} \\
        \vphantom{\displaystyle \frac{1}{1-aq^{-3}(1-q^{-2})z}} \,1 & \text{if $i = 1$.}
    \end{cases}
\end{equation*}
Hence $U_q^-(w_2) \cong L_{2,a\eta_2}^-$ with $\displaystyle \eta_2 = q^{-3}(1-q^{-2})$.
}
\end{ex}

Following \cite{JKP23}, we endow $U_q^-(w_s)$ with another $U_q(\bo)$-module structure as follows:

\begin{equation} \label{eq: right action}
    k_i(u) = 
    \begin{cases}
        q^{(\alpha_i, {\rm wt}(u))} u & \text{if $i \in \I$,} \\
        q^{-(\theta, {\rm wt}(u))}u & \text{if $i = 0$,}
    \end{cases}
    \qquad
    e_i(u) = 
    \begin{cases}
        e_i'(u) & \text{if $i \in \I$,} \\
        aq^{-(\theta, {\rm wt}(u))} u \xz & \text{if $i = 0$.}
    \end{cases}
\end{equation}

\begin{thm} \label{thm: main2}
    Let $s \in \I$ be a minuscule node. Then $U_q^-(w_s)$ is a $U_q(\bo)$-module, which is an object of $\mc{O}$ with respect to \eqref{eq: right action}. Moreover, $U_q^-(w_s)$ is isomorphic to $L_{s,-a\eta_s}^+$ as a $U_q(\bo)$-module.
\end{thm}
\begin{proof}
    The proof of this theorem is almost identical to \cite[Theorem 4.17]{JKP23} (cf.~\cite[Theorem 5.15]{JKP25}) by using Lemmas \ref{lem:dual PBW properties 1} and \ref{lem:dual PBW properties 2}. 
    Here we provide the proof in the case of type $D_{n+1}^{(2)}$, where $s = n$. 
    The proof for type $A_{2n-1}^{(2)}$ is similar.
    \smallskip

    By definition, for $i, j \in I_0$ with $i \neq j$, the operators $e_i$'s in \eqref{eq: right action} satisfy the quantum Serre relations \eqref{eq: qsr1} by \cite[Lemma 3.4.2]{Kas91a}.
    Let us check the remaining cases for $(i,j) = (0,1)$ or $(i,j) = (1,0)$.
    Let $u \in U_q^-(w_n)$ be a homogeneous element. 
    Put $c_0 = aq^{-({\rm wt}(u), \theta)}$ and $c_1 = q^{({\rm wt}(u), \alpha_1)}$ for simplicity.
    \smallskip
    
    \begin{enumerate}[(1)]
    	\item Suppose that $i = 1$ and $j = 0$. Since $a_{ij} = -1$ and $q_i = q^2$, the quantum Serre relation \eqref{eq: qsr1} is given by
	\begin{equation*}
		e_1^2 e_0 - (q^2 + q^{-2}) e_1 e_0 e_1 + e_0 e_1^2 = 0.
	\end{equation*}
	Then one can check
	\begin{equation*}
	\begin{split}
			e_1^2 e_0 (u) &= c_0\left( (e_1')^2(u) \xz + c_1 (1+q^4) e_1'(u) e_1'(\xz) \right), \\
			e_1 e_0 e_1 (u) &= c_0 \left( q^{-2} (e_1')^2(u) \xz + q^2 c_1 e_1'(u) e_1'(\xz) \right), \\
			e_0 e_1^2 (u) &= c_0 q^{-4} (e_1')^2(u) \xz.
	\end{split}
	\end{equation*}
	Note that $(e_1')^2(\xz) = 0$ by Lemma \ref{lem:dual PBW properties 2}(1), so we obtain the above formula of $e_1^2 e_0 (u)$. 
    In $\left( e_1^2 e_0 - (q^2 + q^{-2}) e_1 e_0 e_1 + e_0 e_1^2 \right) (u)$, we obtain
    \begin{align*}
        \text{$\big($the coefficient of $(e_1')^2(u) \xz$ \!$\big)$} &= c_0 ( 1 - (q^2 + q^{-2}) q^{-2} + q^{-4} ) = 0, \\
        \text{$\big($the coefficient of $e_1'(u)e_1'(\xz)$ \!$\big)$} &= c_0 c_1 ( (1+q^4) - (q^2 + q^{-2}) q^2 ) = 0.
    \end{align*}
    Thus the operators $e_i$'s in \eqref{eq: right action} satisfy the quantum Serre relation in this case.
	\smallskip
	
	\item Suppose that $i = 0$ and $j = 1$. Since $a_{ij} = -2$ and $q_i = q$, the quantum Serre relation is given by 
	\begin{equation*}
		e_0^3 e_1 - (q^2 + 1 + q^{-2}) e_0^2 e_1 e_0 + (q^2 + 1 + q^{-2}) e_0 e_1 e_0^2 - e_1 e_0^3 = 0.
	\end{equation*}
	Then one can check 
	\begin{equation*}
	\begin{split}
			e_0^3 e_1(u) &= c_0^3 e_1'(u) \xz^3, \\
			e_0^2 e_1 e_0 (u) &= c_0^3 \left( q^2 e_1'(u) \xz^3 + c_1 u e_1'(\xz) \xz^2 \right), \\
			e_0 e_1 e_0^2 (u) &= c_0^3 \left( q^4 e_1'(u) \xz^3 + c_1 (1+q^{-2}) u e_1'(\xz) \xz^2 \right), \\
			e_1 e_0^3 (u) &= c_0^3 \left( q^6 e_1'(u) \xz^3 + c_1 (1+ q^{-2} + q^{-4}) u e_1'(\xz) \xz^2 \right).
	\end{split}
	\end{equation*}
	Note that $\xz e_1'(\xz) = e_1'(\xz) \xz$ by Lemma \ref{lem:dual PBW properties 2}(2), so we obtain the formulas of $e_0 e_1 e_0^2 (u)$ and $e_1 e_0^3 (u)$. 
    In $\left( e_0^3 e_1 - (q^2 + 1 + q^{-2}) e_0^2 e_1 e_0 + (q^2 + 1 + q^{-2}) e_0 e_1 e_0^2 - e_1 e_0^3 \right)(u)$, we obtain
    \begin{align*}
        \text{$\big($the coefficient of $e_1'(u) \xz^3$ \!$\big)$} &= c_0^3 \left( 1 - [3]_q q^2 + [3]_q q^4 - q^6 \right) = 0, \\ 
        \text{$\big($the coefficient of $u e_1'(\xz) \xz^2$ \!$\big)$} &= c_0^3 c_1 ( -[3]_q + (1+q^{-2}) [3]_q - (1+q^{-2} + q^{-4})) = 0,
    \end{align*}
    where $[3]_q = q^2 + 1 + q^{-2}$.
    Thus the operators $e_i$'s in \eqref{eq: right action} satisfy the quantum Serre relation in this case.
    \end{enumerate}
    The other relations for $e_i$ and $k_j$ can be checked by similar computations with \cite[Theorem 3.10]{JKP23}, so we leave the details to the reader.
    \smallskip
    
    Next, the $U_q(\bo)$-module $U_q^-(w_n)$ is $\mathfrak{t}$-diagonalizable by definition, where each weight space of $U_q^-(w_n)$ is finite-dimensional. Also, all weights of $U_q^-(w_n)$ are in $D(\overline{1})$, where $\overline{1}$ is the map from $I$ to $\bk^\times$ by $i \mapsto 1$ for all $i \in I$. Hence $U_q^-(w_n)$ is an object of $\mc{O}$ (see Definition \ref{df:category O}).
	\smallskip
	
    Finally, we check the second statement as follows.
	By Theorem \ref{thm: characters}, it is enough to compute the highest $\ell$-weight of $U_q^-(w_s)$ with respect to \eqref{eq: right action}.
    Recall \eqref{eq: expansion of E-del-aln}.
    Since $({\rm wt}(f_n), \theta) = 0$ and $({\rm wt}(f_n), \alpha_{n-1}) = 2$, we have 
    \begin{equation*}
        \E_{\alpha_n} \E_{\delta-\alpha_n}^2 1 = a^2 q^{-4n+6} (1+q^{-2}) f_n,
    \end{equation*}
    which implies that $E_{2\delta-\alpha_n} 1 = 0$.
    By Lemma \ref{lem: computation of l-weights} and Corollary \ref{cor: l-weights}, we conclude
    \begin{equation*}
        \Psi_i(z) = \begin{cases}
            1 + a \eta_n z & \text{if $i = n$,} \\
            1 & \text{otherwise.}
        \end{cases}
    \end{equation*}
    This completes the proof.
\end{proof}

\subsection{Further discussion on non-(co)minuscule cases} \label{subsec:non-(co)minuscule cases}
It is natural to ask whether the current realization of $L_{s,a}^\pm$ via the unipotent quantum coordinate ring $U_q^-(w_s)$ can be extended to the general case, that is, to non-(co)minuscule cases. An attempt in this direction was made in \cite{JKP25}. In that work, the authors generalized the $0$-operator on $U_q^-(w_s)$ \eqref{eq: left action} using braid group actions for all untwisted affine types with arbitrary $s \in \I$, while the remaining actions are still given by the $q$-derivations $e_i'$ for all $i \in \I$. 
However, for non-(co)minuscule cases, the $U_q(\bo)$-module $U_q^-(w_s)$ cannot be isomorphic to $L_{s,a}^-$, since the $U_q(\bo)$-module $U_q^-(w_s)$ has a different highest $\ell$-weight from that of $L_{s,a}^-$ and it is not irreducible \cite[Proposition 4.16]{JKP25}.

Let us recall that the negative prefundamental module $L_{s,a}^-$ is constructed from a certain inductive system associated with the family of KR modules $W_{k,aq_s^{-2k+1}}^{(s)}$ $(k \in \Z_+)$, over the asymptotic algebra formulated in terms of Drinfeld generators \cite{HJ}.
Note that the description of the $0$-action $e_0$ induced from the construction over the asymptotic algebra is more complicated than that of the other operators $e_i$ (see~\cite[Proposition 2.4]{HJ}). 
Moreover, to the best of our knowledge, no explicit description of the $0$-action $e_0$ of KR modules is known, except for type $A_n^{(1)}$ using Jimbo's evaluation homomorphism \cite{Jim85}, and consequently also that of prefundamental modules, although one can in principle extract the $0$-action of KR modules from the fusion construction using normalized $R$-matrices \cite{KKMMNN92a, Kas02a}.
This explains why the realization of $L_{s,a}^\pm$ using $U_q^-(w_s)$ in Theorems \ref{thm: main1}, \ref{thm: main2}, and \cite{JKP23, JKP25} may be interesting. In particular, the simple description of $0$-action $e_0$ is rather surprising in this context.

Let us now recall the crystal side of the picture.
It was conjectured by Hatayama et al. \!\!\cite{HKOTY99} that for $s \in \I$ and $k \in \mathbb{Z}_+$, there exists $a_{s, k} \in \bk^\times$ such that $W^{(s)}_{k, a_{s, k}}$ has the crystal pseudo-base introduced in \cite{KKMMNN92b} (cf.~\cite{Kas91a}).
The conjecture has been proved for $s$ in the orbit of or adjacent to $0$ in all affine types \cite{KKMMNN92a,KKMMNN92b}, all non-exceptional types \cite{OS08}, types $G_2^{(1)}$ and $D_4^{(3)}$ \cite{Nao18}, and types $E_{6,7,8}^{(1)}$, $F_4^{(1)}$, $E_6^{(2)}$ with near adjoint nodes \cite{NaSc21}.
Let $B^{s, k}$ be the crystal of the KR module $W^{(s)}_{k, a_{s, k}}$, which is called the KR crystal for short.
Note that it is well-known (e.g.~\cite{CP95}) that $W^{(s)}_{k, a}$ is isomorphic to $W_{k,1}^{(s)}$, so $W^{(s)}_{k, a_{s, k}} \cong W_{k,aq_s^{-2k+1}}^{(s)}$.

For (co)minuscule $s \in \I$, the crystal $B^{s,k}$ is classically irreducible, that is, it is connected as a classical crystal \cite{ScTi12}, and the $0$-th crystal operator $\wtd{e}_0$ can be described explicitly (see~\cite{FOS09} and references therein).
In particular, it is known in \cite{Kw13, JK19, Jang22, JaSc26}  that the KR crystal $B^{s,k}$ can be realized inside the PBW crystal associated to $U_q^-(w_s)$, where the $0$-crystal operator is given simply by adding $1$ to the Lusztig datum corresponding to the root $\theta$.
This is one of the motivations for \eqref{eq: right action} due to the relationship of the crystal operator $\te_i$ with the $q$-derivation $e_i'$ for $i \in \I$ (see Remark \ref{rem: e' and te}).

In contrast, for non-(co)minuscule cases, the restriction of ${\scalebox{0.9}{$W_{k,aq_s^{-2k+1}}^{(s)}$}}$ as a $U_q(\cg)$-module decomposes into a direct sum of several irreducible $U_q(\cg)$-modules \cite{Cha01}.
This suggests that the present approach may not be extended directly to the general case, since it is difficult to expect that the multiplication by the maximal root vector can capture the $0$-actions among (distinct) irreducible factors in the classical decomposition.
Thus, to obtain an isomorphism with $L_{s,a}^\pm$ for non-(co)minuscule $s \in \I$, 
we expect that a different description of the actions $e_i$ $(i \in I)$ on $U_q^-(w_s)$ will be required.
\vskip 2mm

\noindent
{\em \bf Acknowledgement.}\!
This work is partly motivated by a joint work with Travis Scrimshaw related to Kirillov--Reshetikhin crystals \cite{JaSc26}. The author is very grateful to Euiyong Park for helpful discussions and encouraging him in the preparation of this paper.
He also would like to sincerely thank the anonymous referees for helpful and detailed comments, which significantly improved the manuscript.
This work was supported by the National Research Foundation of Korea (NRF) grant funded by the Korea government (MSIT) (No.~RS-2023-00277388).
\vskip 2mm

\noindent
\textbf{Data Availability.}\; %This paper has no associated data.
Data sharing is not applicable to this article as no datasets were generated or analyzed during the current study.

\section*{Declarations}

\noindent
\textbf{Conflict of Interest.}\;
The author states that there is no conflict of interest.

\appendix

\section{Dynkin diagrams of affine types with Kac labels} \label{sec:list of dynkins}

\begin{table}[H]
	\resizebox{\columnwidth}{!}{%
		\begin{tabular}{|c||c||c|c|}
			\hline
			type & Dynkin diagram & type & Dynkin diagram \\
			%& cominuscule ($\circ$) & minuscule ($\star$) \\
			\hline \hline
			$A_1^{(1)}$ & 
			\begin{minipage}[c][0.06\textheight]{0.5\linewidth}
			\setlength{\unitlength}{0.17in}
			\begin{picture}(16,3.7)
					\put(7,2.1){\makebox(0,0)[c]{$\circ$}}
					\put(9,2.1){\makebox(0,0)[c]{$\circ$}}
					\put(7.3,2.0){\line(1,0){1.3}}
					\put(7.3,2.2){\line(1,0){1.3}}
					\put(7.0,1.89){\Large $\prec$}
					\put(8.2,1.89){\Large $\succ$}
					\put(7,1.3){\makebox(0,0)[c]{\tiny ${\alpha}_0$}}
					\put(9,1.3){\makebox(0,0)[c]{\tiny ${\alpha}_1$}}
                    %
                    % Kac label
                    \put(7,2.6){\makebox(0,0)[c]{\tiny $1$}}
					\put(9,2.6){\makebox(0,0)[c]{\tiny $1$}}
				\end{picture}
			\end{minipage} &
                $A_2^{(2)}$ & 
                \begin{minipage}[c][0.06\textheight]{0.5\linewidth}
			\setlength{\unitlength}{0.17in}
			\begin{picture}(16,3.7)
					\put(7,2.1){\makebox(0,0)[c]{$\circ$}}
					\put(9,2.1){\makebox(0,0)[c]{$\circ$}}
                        \put(7.3,1.8){\line(1,0){1.3}}
					\put(7.3,2.0){\line(1,0){1.3}}
					\put(7.3,2.2){\line(1,0){1.3}}
                        \put(7.3,2.4){\line(1,0){1.3}}
					%
					%\put(7.0,1.69){\scalebox{2}{$\prec$}}
					\put(7.7,1.69){\scalebox{2}{$\succ$}}
					\put(7,1.3){\makebox(0,0)[c]{\tiny ${\alpha}_0$}}
					\put(9,1.3){\makebox(0,0)[c]{\tiny ${\alpha}_1$}}
				    %
                        % Kac labels
                        %\put(7,2.7){\makebox(0,0)[c]{\tiny $2$}}
                        %\put(9,2.7){\makebox(0,0)[c]{\tiny $1$}}
                        \put(7,2.7){\makebox(0,0)[c]{\tiny $1$}}
                        \put(9,2.7){\makebox(0,0)[c]{\tiny $2$}}
                    \end{picture}
			\end{minipage}
                \\ %$1$ & $1$  \\
			\hline
			$\underset{(n \ge 2)}{A_n^{(1)}}$ & 
			\begin{minipage}[c][0.09\textheight]{0.5\linewidth}
			\setlength{\unitlength}{0.17in}
				\begin{picture}(16,5.5)
					\put(3,2.1){\makebox(0,0)[c]{$\circ$}}
					\put(5.6,2.1){\makebox(0,0)[c]{$\circ$}}
					\put(10.5,2.1){\makebox(0,0)[c]{$\circ$}}
					\put(13,2.1){\makebox(0,0)[c]{$\circ$}}
					\put(8.1,4.1){\makebox(0,0)[c]{$\circ$}}
					\put(3.7,2.1){\line(1,0){1.3}}
					\put(6,2.1){\line(1,0){1.3}}
					\put(8.7,2.1){\line(1,0){1.3}}
					\put(11.2,2.1){\line(1,0){1.3}}
					\put(7.7,3.9){\line(-3,-1){4.3}}
					\put(8.5,3.9){\line(3,-1){4}}

					\put(8,2.05){\makebox(0,0)[c]{$\cdots$}}
					\put(8,3.2){\makebox(0,0)[c]{\tiny $\alpha_0$}}
					\put(3.2,1.1){\makebox(0,0)[c]{\tiny ${\alpha}_1$}}
					\put(5.6,1.1){\makebox(0,0)[c]{\tiny ${\alpha}_2$}}
					\put(10.4,1.1){\makebox(0,0)[c]{\tiny ${\alpha}_{n-1}$}}
					\put(13,1.1){\makebox(0,0)[c]{\tiny ${\alpha}_{n}$}}
					%
                    % Kac label
                    \put(3,2.6){\makebox(0,0)[c]{\tiny $1$}}
					\put(5.6,2.6){\makebox(0,0)[c]{\tiny $1$}}
					\put(10.5,2.6){\makebox(0,0)[c]{\tiny $1$}}
					\put(13,2.6){\makebox(0,0)[c]{\tiny $1$}}
					\put(8.1,4.6){\makebox(0,0)[c]{\tiny $1$}}
				\end{picture}
			\end{minipage} & $\underset{(n \ge 2)}{A_{2n}^{(2)}}$
            & \begin{minipage}[c][0.06\textheight]{0.5\linewidth}
			\setlength{\unitlength}{0.17in}
			\begin{picture}(16,3.7)
					\put(3,2.1){\makebox(0,0)[c]{$\circ$}}
					\put(5.6,2.1){\makebox(0,0)[c]{$\circ$}}
					\put(10.5,2.1){\makebox(0,0)[c]{$\circ$}}
					\put(13,2.1){\makebox(0,0)[c]{$\circ$}}
					\put(3.7,2){\line(1,0){1.3}}
					\put(3.7,2.2){\line(1,0){1.3}}
					\put(6,2.1){\line(1,0){1.3}}
					\put(8.7,2.1){\line(1,0){1.3}}
					\put(11.2,2){\line(1,0){1.3}}
					\put(11.2,2.2){\line(1,0){1.3}}
					%\put(11.5,1.88){\Large $\prec$}
                        \put(11.5,1.88){\Large $\succ$}
					%\put(4,1.88){\Large $\prec$}
                        \put(4,1.88){\Large $\succ$}

					\put(8,2.05){\makebox(0,0)[c]{$\cdots$}}
					\put(3.2,1.1){\makebox(0,0)[c]{\tiny ${\alpha}_0$}}
					\put(5.6,1.1){\makebox(0,0)[c]{\tiny ${\alpha}_1$}}
					\put(10.4,1.1){\makebox(0,0)[c]{\tiny ${\alpha}_{n-1}$}}
					\put(13,1.1){\makebox(0,0)[c]{\tiny ${\alpha}_{n}$}}
					%
					% Kac label (ge 2)
					\put(3,2.8){\makebox(0,0)[c]{\tiny $1$}}
                    \put(5.6,2.8){\makebox(0,0)[c]{\tiny $2$}}
					\put(10.4,2.8){\makebox(0,0)[c]{\tiny $2$}}
                    \put(13,2.8){\makebox(0,0)[c]{\tiny $2$}}
				\end{picture}
			\end{minipage}
            \\ %$1, 2, \dots, n$  &  $1, 2, \dots, n$ \\
			\hline
			$\underset{(n \ge 3)}{B_n^{(1)}}$ & 
			\begin{minipage}[c][0.10\textheight]{0.5\linewidth}
			\setlength{\unitlength}{0.17in}
				\begin{picture}(16,3.7)
					\put(2.8,0.7){\makebox(0,0)[c]{$\circ$}}
					\put(2.8,3.3){\makebox(0,0)[c]{$\circ$}}
					\put(5.6,2.1){\makebox(0,0)[c]{$\circ$}}
					\put(10.5,2.1){\makebox(0,0)[c]{$\circ$}}
					\put(13,2.1){\makebox(0,0)[c]{$\circ$}}

					\put(5.3,1.8){\line(-2,-1){2.1}}
					\put(5.3,2.2){\line(-2,1){2.1}}
					\put(6,2.1){\line(1,0){1.3}}
					\put(8.7,2.1){\line(1,0){1.3}}
					\put(11.2,2){\line(1,0){1.3}}
					\put(11.2,2.2){\line(1,0){1.3}}
					\put(11.5,1.88){\Large $\succ$}

					\put(8,2.05){\makebox(0,0)[c]{$\cdots$}}
					\put(2.8,2.5){\makebox(0,0)[c]{\tiny ${\alpha}_0$}}
					\put(2.8,0){\makebox(0,0)[c]{\tiny ${\alpha}_1$}}
					\put(5.6,1.1){\makebox(0,0)[c]{\tiny ${\alpha}_2$}}
					\put(10.4,1.1){\makebox(0,0)[c]{\tiny ${\alpha}_{n-1}$}}
					\put(13,1.1){\makebox(0,0)[c]{\tiny ${\alpha}_{n}$}}
					%
					% Kac label
                    \put(2.8,1.4){\makebox(0,0)[c]{\tiny $1$}}
					\put(2.8,4){\makebox(0,0)[c]{\tiny $1$}}
					\put(5.6,2.8){\makebox(0,0)[c]{\tiny $2$}}
					\put(10.4,2.8){\makebox(0,0)[c]{\tiny $2$}}
					\put(13,2.8){\makebox(0,0)[c]{\tiny $2$}}
				\end{picture}
			\end{minipage} & 
			$\underset{(n \ge 3)}{A_{2n-1}^{(2)}}$ & 
			\begin{minipage}[c][0.10\textheight]{0.5\linewidth}
			\setlength{\unitlength}{0.17in}
				\begin{picture}(16,3.7)
					\put(2.8,0.7){\makebox(0,0)[c]{$\circ$}}
					\put(2.8,3.3){\makebox(0,0)[c]{$\circ$}}
					\put(5.6,2.1){\makebox(0,0)[c]{$\circ$}}
					\put(10.5,2.1){\makebox(0,0)[c]{$\circ$}}
					\put(13,2.1){\makebox(0,0)[c]{$\circ$}}

					\put(5.3,1.8){\line(-2,-1){2.1}}
					\put(5.3,2.2){\line(-2,1){2.1}}
					\put(6,2.1){\line(1,0){1.3}}
					\put(8.7,2.1){\line(1,0){1.3}}
					\put(11.2,2){\line(1,0){1.3}}
					\put(11.2,2.2){\line(1,0){1.3}}
					\put(11.5,1.88){\Large $\prec$}

					\put(8,2.05){\makebox(0,0)[c]{$\cdots$}}
					\put(2.8,2.5){\makebox(0,0)[c]{\tiny ${\alpha}_0$}}
					\put(2.8,0){\makebox(0,0)[c]{\tiny ${\alpha}_1$}}
					\put(5.6,1.1){\makebox(0,0)[c]{\tiny ${\alpha}_2$}}
					\put(10.4,1.1){\makebox(0,0)[c]{\tiny ${\alpha}_{n-1}$}}
					\put(13,1.1){\makebox(0,0)[c]{\tiny ${\alpha}_{n}$}}
					%
					% Kac label
                    \put(2.8,1.4){\makebox(0,0)[c]{\tiny $1$}}
					\put(2.8,4){\makebox(0,0)[c]{\tiny $1$}}
					\put(5.6,2.8){\makebox(0,0)[c]{\tiny $2$}}
					\put(10.4,2.8){\makebox(0,0)[c]{\tiny $2$}}
					\put(13,2.8){\makebox(0,0)[c]{\tiny $1$}}
				\end{picture}
			\end{minipage}
			\\ %$1$ & $n$ \\
			\hline
			$\underset{(n \ge 2)}{C_n^{(1)}}$ & 
			\begin{minipage}[c][0.06\textheight]{0.5\linewidth}
			\setlength{\unitlength}{0.17in}
			\begin{picture}(16,3.7)
					\put(3,2.1){\makebox(0,0)[c]{$\circ$}}
					\put(5.6,2.1){\makebox(0,0)[c]{$\circ$}}
					\put(10.5,2.1){\makebox(0,0)[c]{$\circ$}}
					\put(13,2.1){\makebox(0,0)[c]{$\circ$}}
					\put(3.7,2){\line(1,0){1.3}}
					\put(3.7,2.2){\line(1,0){1.3}}
					\put(6,2.1){\line(1,0){1.3}}
					\put(8.7,2.1){\line(1,0){1.3}}
					\put(11.2,2){\line(1,0){1.3}}
					\put(11.2,2.2){\line(1,0){1.3}}
					\put(11.5,1.88){\Large $\prec$}
					\put(4,1.88){\Large $\succ$}

					\put(8,2.05){\makebox(0,0)[c]{$\cdots$}}
					\put(3.2,1.1){\makebox(0,0)[c]{\tiny ${\alpha}_0$}}
					\put(5.6,1.1){\makebox(0,0)[c]{\tiny ${\alpha}_1$}}
					\put(10.4,1.1){\makebox(0,0)[c]{\tiny ${\alpha}_{n-1}$}}
					\put(13,1.1){\makebox(0,0)[c]{\tiny ${\alpha}_{n}$}}
					%
					% Kac label
                    \put(3,2.8){\makebox(0,0)[c]{\tiny $1$}}
					\put(5.6,2.8){\makebox(0,0)[c]{\tiny $2$}}
					\put(10.4,2.8){\makebox(0,0)[c]{\tiny $2$}}
                    \put(13,2.8){\makebox(0,0)[c]{\tiny $1$}}
				\end{picture}
			\end{minipage} & $\underset{(n \ge 2)}{D_{n+1}^{(2)}}$ &
			\begin{minipage}[c][0.06\textheight]{0.5\linewidth}
			\setlength{\unitlength}{0.17in}
			\begin{picture}(16,3.7)
					\put(3,2.1){\makebox(0,0)[c]{$\circ$}}
					\put(5.6,2.1){\makebox(0,0)[c]{$\circ$}}
					\put(10.5,2.1){\makebox(0,0)[c]{$\circ$}}
					\put(13,2.1){\makebox(0,0)[c]{$\circ$}}
					\put(3.7,2){\line(1,0){1.3}}
					\put(3.7,2.2){\line(1,0){1.3}}
					\put(6,2.1){\line(1,0){1.3}}
					\put(8.7,2.1){\line(1,0){1.3}}
					\put(11.2,2){\line(1,0){1.3}}
					\put(11.2,2.2){\line(1,0){1.3}}
					\put(11.5,1.88){\Large $\succ$}
					\put(4,1.88){\Large $\prec$}

					\put(8,2.05){\makebox(0,0)[c]{$\cdots$}}
					\put(3.2,1.1){\makebox(0,0)[c]{\tiny ${\alpha}_0$}}
					\put(5.6,1.1){\makebox(0,0)[c]{\tiny ${\alpha}_1$}}
					\put(10.4,1.1){\makebox(0,0)[c]{\tiny ${\alpha}_{n-1}$}}
					\put(13,1.1){\makebox(0,0)[c]{\tiny ${\alpha}_{n}$}}
					%
					% Kac label
					\put(3,2.8){\makebox(0,0)[c]{\tiny $1$}}
					\put(5.6,2.8){\makebox(0,0)[c]{\tiny $1$}}
					\put(10.5,2.8){\makebox(0,0)[c]{\tiny $1$}}
					\put(13,2.8){\makebox(0,0)[c]{\tiny $1$}}
				\end{picture}
			\end{minipage} \\ %$n$ & $1$  \\
			\hline
			$\underset{(n \ge 4)}{D_n^{(1)}}$ &
			\begin{minipage}[c][0.11\textheight]{0.5\linewidth}
				\setlength{\unitlength}{0.17in}
				\begin{picture}(16,3.7)
					\put(2.8,0.7){\makebox(0,0)[c]{$\circ$}}
					\put(2.8,3.3){\makebox(0,0)[c]{$\circ$}}
					\put(5.6,2){\makebox(0,0)[c]{$\circ$}}
					\put(10.4,2){\makebox(0,0)[c]{$\circ$}}
					\put(13.1,3.3){\makebox(0,0)[c]{$\circ$}}
					\put(13.1,0.7){\makebox(0,0)[c]{$\circ$}}
					\put(5.3,1.8){\line(-2,-1){2.1}}
					\put(5.3,2.2){\line(-2,1){2.1}}
					\put(6,2){\line(1,0){1.3}}
					\put(8.7,2){\line(1,0){1.3}}
					\put(10.7,2.2){\line(2,1){2}}
					\put(10.7,1.8){\line(2,-1){2}}

					\put(8,1.95){\makebox(0,0)[c]{$\cdots$}}
					\put(2.8,2.5){\makebox(0,0)[c]{\tiny ${\alpha}_0$}}
					\put(2.8,0.0){\makebox(0,0)[c]{\tiny ${\alpha}_1$}}
					\put(5.6,1){\makebox(0,0)[c]{\tiny ${\alpha}_2$}}
					\put(10.4,1){\makebox(0,0)[c]{\tiny ${\alpha}_{n-2}$}}
					\put(13.1,2.5){\makebox(0,0)[c]{\tiny ${\alpha}_{n-1}$}}
					\put(13.1,0.0){\makebox(0,0)[c]{\tiny ${\alpha}_{n}$}}
					%
					% Kac label (ge 2)
                    \put(2.8,1.4){\makebox(0,0)[c]{\tiny $1$}}
					\put(2.8,4){\makebox(0,0)[c]{\tiny $1$}}
					\put(5.6,2.7){\makebox(0,0)[c]{\tiny $2$}}
					\put(10.4,2.7){\makebox(0,0)[c]{\tiny $2$}}
                    \put(13.1,4){\makebox(0,0)[c]{\tiny $1$}}
					\put(13.1,1.4){\makebox(0,0)[c]{\tiny $1$}}
				\end{picture}
			\end{minipage} & $E_6^{(1)}$ & 
			\begin{minipage}[c][0.15\textheight]{0.5\linewidth}
			\setlength{\unitlength}{0.17in}
				\begin{picture}(16,8.5)
					\put(3,2.1){\makebox(0,0)[c]{$\circ$}}
					\put(5.6,2.1){\makebox(0,0)[c]{$\circ$}}
					\put(8,2.1){\makebox(0,0)[c]{$\circ$}}
					\put(10.5,2.1){\makebox(0,0)[c]{$\circ$}}
					\put(13, 2.1){\makebox(0,0)[c]{$\circ$}}
					\put(8,4.5){\makebox(0,0)[c]{$\circ$}}
					\put(8,7){\makebox(0,0)[c]{$\circ$}}

					\put(8, 4){\line(0, -3){1.3}}
					\put(8, 6.5){\line(0, -3){1.3}}
					\put(3.7,2.1){\line(1,0){1.3}}
					\put(6,2.1){\line(1,0){1.3}}
					\put(8.7,2.1){\line(1,0){1.3}}
					\put(11.2,2.1){\line(1,0){1.3}}
					\put(8.7,7){\makebox(0,0)[c]{\tiny $\alpha_0$}}
					\put(3.2,1.1){\makebox(0,0)[c]{\tiny ${\alpha}_1$}}
					\put(8.7,4.5){\makebox(0,0)[c]{\tiny $\alpha_6$}}
					\put(5.6,1.1){\makebox(0,0)[c]{\tiny ${\alpha}_2$}}
					\put(8,1.1){\makebox(0,0)[c]{\tiny $\alpha_3$}}
					\put(10.4,1.1){\makebox(0,0)[c]{\tiny ${\alpha}_4$}}
					\put(13,1.1){\makebox(0,0)[c]{\tiny ${\alpha}_5$}}
					%
					%Kac labels
                    \put(3,2.7){\makebox(0,0)[c]{\tiny $1$}}
					\put(7.4,4.5){\makebox(0,0)[c]{\tiny $2$}}
					\put(5.6,2.7){\makebox(0,0)[c]{\tiny $2$}}
					\put(8.5,2.7){\makebox(0,0)[c]{\tiny $3$}}
					\put(10.4,2.7){\makebox(0,0)[c]{\tiny $2$}}
                    \put(7.4,7){\makebox(0,0)[c]{\tiny $1$}}
                    \put(13, 2.7){\makebox(0,0)[c]{\tiny $1$}}
				\end{picture}
			\end{minipage} \\ %$1,\,6$ & $1,\,6$ \\  \\ %$1,\,n-1,\,n$ & $1,\,n-1,\,n$ \\
			\hline
			$E_7^{(1)}$ & 
			\begin{minipage}[c][0.10\textheight]{0.5\linewidth}
			\setlength{\unitlength}{0.17in}
				\begin{picture}(16,5.5)
					\put(0.7,2.1){\makebox(0,0)[c]{$\circ$}}
					\put(3,2.1){\makebox(0,0)[c]{$\circ$}}
					\put(5.6,2.1){\makebox(0,0)[c]{$\circ$}}
					\put(8,2.1){\makebox(0,0)[c]{$\circ$}}
					\put(10.5,2.1){\makebox(0,0)[c]{$\circ$}}
					\put(13, 2.1){\makebox(0,0)[c]{$\circ$}}
					\put(15.5,2.1){\makebox(0,0)[c]{$\circ$}}
					\put(8,4.5){\makebox(0,0)[c]{$\circ$}}
					\put(8, 4){\line(0, -3){1.3}}
					\put(1.2,2.1){\line(1,0){1.3}}
					\put(3.7,2.1){\line(1,0){1.3}}
					\put(6,2.1){\line(1,0){1.3}}
					\put(8.7,2.1){\line(1,0){1.3}}
					\put(11.2,2.1){\line(1,0){1.3}}
					\put(13.7,2.1){\line(1,0){1.3}}
					\put(0.8,1.1){\makebox(0,0)[c]{\tiny ${\alpha}_0$}}
					\put(3.2,1.1){\makebox(0,0)[c]{\tiny ${\alpha}_1$}}
					\put(8.7,4.5){\makebox(0,0)[c]{\tiny $\alpha_7$}}
					\put(5.6,1.1){\makebox(0,0)[c]{\tiny ${\alpha}_2$}}
					\put(8,1.1){\makebox(0,0)[c]{\tiny $\alpha_3$}}
					\put(10.4,1.1){\makebox(0,0)[c]{\tiny ${\alpha}_4$}}
					\put(13,1.1){\makebox(0,0)[c]{\tiny ${\alpha}_5$}}
					\put(15.5, 1.1){\makebox(0,0)[c]{\tiny ${\alpha}_6$}}
					%
					%Kac labels
                    \put(0.7,2.7){\makebox(0,0)[c]{\tiny $1$}}
					\put(3.2,2.7){\makebox(0,0)[c]{\tiny $2$}}
					\put(7.5,4.5){\makebox(0,0)[c]{\tiny $2$}}
					\put(5.6,2.7){\makebox(0,0)[c]{\tiny $3$}}
					\put(8.5,2.7){\makebox(0,0)[c]{\tiny $4$}}
					\put(10.4,2.7){\makebox(0,0)[c]{\tiny $3$}}
					\put(13,2.7){\makebox(0,0)[c]{\tiny $2$}}
                    \put(15.5,2.7){\makebox(0,0)[c]{\tiny $1$}}
				\end{picture}
			\end{minipage} & $E_8^{(1)}$ & 
			\begin{minipage}[c][0.10\textheight]{0.5\linewidth}
			\setlength{\unitlength}{0.15in}
				\begin{picture}(16,5.5)
					\put(0.7,2.1){\makebox(0,0)[c]{$\circ$}}
					\put(3,2.1){\makebox(0,0)[c]{$\circ$}}
					\put(5.6,2.1){\makebox(0,0)[c]{$\circ$}}
					\put(8,2.1){\makebox(0,0)[c]{$\circ$}}
					\put(10.5,2.1){\makebox(0,0)[c]{$\circ$}}
					\put(13, 2.1){\makebox(0,0)[c]{$\circ$}}
					\put(15.5,2.1){\makebox(0,0)[c]{$\circ$}}
					\put(18,2.1){\makebox(0,0)[c]{$\circ$}}
					\put(13,4.5){\makebox(0,0)[c]{$\circ$}}
					\put(13,4){\line(0, -3){1.3}}
					\put(1.2,2.1){\line(1,0){1.3}}
					\put(3.7,2.1){\line(1,0){1.3}}
					\put(6,2.1){\line(1,0){1.3}}
					\put(8.7,2.1){\line(1,0){1.3}}
					\put(11.2,2.1){\line(1,0){1.3}}
					\put(13.7,2.1){\line(1,0){1.3}}
					\put(16.2,2.1){\line(1,0){1.3}}
					\put(0.8,1.1){\makebox(0,0)[c]{\tiny ${\alpha}_0$}}
					\put(3.2,1.1){\makebox(0,0)[c]{\tiny ${\alpha}_1$}}
					\put(13.8,4.5){\makebox(0,0)[c]{\tiny $\alpha_8$}}
					\put(5.6,1.1){\makebox(0,0)[c]{\tiny ${\alpha}_2$}}
					\put(8,1.1){\makebox(0,0)[c]{\tiny $\alpha_3$}}
					\put(10.4,1.1){\makebox(0,0)[c]{\tiny ${\alpha}_4$}}
					\put(13,1.1){\makebox(0,0)[c]{\tiny ${\alpha}_5$}}
					\put(15.5, 1.1){\makebox(0,0)[c]{\tiny ${\alpha}_6$}}
					\put(18, 1.1){\makebox(0,0)[c]{\tiny ${\alpha}_7$}}
					%
					% Kac labels
					\put(0.8,2.7){\makebox(0,0)[c]{\tiny $1$}}
					\put(3.2,2.7){\makebox(0,0)[c]{\tiny $2$}}
                    \put(5.6,2.7){\makebox(0,0)[c]{\tiny $3$}}
                    \put(8,2.7){\makebox(0,0)[c]{\tiny $4$}}
                    \put(10.4,2.7){\makebox(0,0)[c]{\tiny $5$}}
                    \put(13.4,2.7){\makebox(0,0)[c]{\tiny $6$}}
					\put(12.4,4.5){\makebox(0,0)[c]{\tiny $3$}}
					\put(15.5,2.7){\makebox(0,0)[c]{\tiny $4$}}
                    \put(18,2.7){\makebox(0,0)[c]{\tiny $2$}}
				\end{picture}
			\end{minipage}  \\ %$7$ & $7$ \\
			\hline
			$F_4^{(1)}$ & 
			\begin{minipage}[c][0.06\textheight]{0.5\linewidth}
			\setlength{\unitlength}{0.17in}
			\begin{picture}(16,3.9)
					\put(4,2.1){\makebox(0,0)[c]{$\circ$}}
					\put(6,2.1){\makebox(0,0)[c]{$\circ$}}
					\put(8,2.1){\makebox(0,0)[c]{$\circ$}}
					\put(10,2.1){\makebox(0,0)[c]{$\circ$}}
					\put(12,2.1){\makebox(0,0)[c]{$\circ$}}
					\put(4.3,2.1){\line(1,0){1.3}}
					\put(6.3,2.1){\line(1,0){1.3}}
					\put(8.3,2.0){\line(1,0){1.3}}
					\put(8.3,2.2){\line(1,0){1.3}}
					\put(10.3,2.1){\line(1,0){1.3}}
					\put(8.6,1.88){\Large $\succ$}
					\put(4,1.3){\makebox(0,0)[c]{\tiny ${\alpha}_0$}}
					\put(6,1.3){\makebox(0,0)[c]{\tiny ${\alpha}_1$}}
					\put(8,1.3){\makebox(0,0)[c]{\tiny ${\alpha}_2$}}
					\put(10,1.3){\makebox(0,0)[c]{\tiny ${\alpha}_3$}}
					\put(12,1.3){\makebox(0,0)[c]{\tiny ${\alpha}_4$}}
					%
					% Kac labels
					\put(4,2.7){\makebox(0,0)[c]{\tiny $1$}}
                    \put(6,2.7){\makebox(0,0)[c]{\tiny $2$}}
					\put(8,2.7){\makebox(0,0)[c]{\tiny $3$}}
					\put(10,2.7){\makebox(0,0)[c]{\tiny $4$}}
					\put(12,2.7){\makebox(0,0)[c]{\tiny $2$}}
				\end{picture}
			\end{minipage} & $E_6^{(2)}$ &
			\begin{minipage}[c][0.06\textheight]{0.5\linewidth}
			\setlength{\unitlength}{0.17in}
			\begin{picture}(16,3.9)
					\put(4,2.1){\makebox(0,0)[c]{$\circ$}}
					\put(6,2.1){\makebox(0,0)[c]{$\circ$}}
					\put(8,2.1){\makebox(0,0)[c]{$\circ$}}
					\put(10,2.1){\makebox(0,0)[c]{$\circ$}}
					\put(12,2.1){\makebox(0,0)[c]{$\circ$}}
					\put(4.3,2.1){\line(1,0){1.3}}
					\put(6.3,2.1){\line(1,0){1.3}}
					\put(8.3,2.0){\line(1,0){1.3}}
					\put(8.3,2.2){\line(1,0){1.3}}
					\put(10.3,2.1){\line(1,0){1.3}}
					\put(8.6,1.88){\Large $\prec$}
					\put(4,1.3){\makebox(0,0)[c]{\tiny ${\alpha}_0$}}
					\put(6,1.3){\makebox(0,0)[c]{\tiny ${\alpha}_1$}}
					\put(8,1.3){\makebox(0,0)[c]{\tiny ${\alpha}_2$}}
					\put(10,1.3){\makebox(0,0)[c]{\tiny ${\alpha}_3$}}
					\put(12,1.3){\makebox(0,0)[c]{\tiny ${\alpha}_4$}}
					%
					% Kac labels
                    \put(4,2.7){\makebox(0,0)[c]{\tiny $1$}}
					\put(6,2.7){\makebox(0,0)[c]{\tiny $2$}}
					\put(8,2.7){\makebox(0,0)[c]{\tiny $3$}}
					\put(10,2.7){\makebox(0,0)[c]{\tiny $2$}}
					\put(12,2.7){\makebox(0,0)[c]{\tiny $1$}}
				\end{picture}
			\end{minipage}
			 \\ % {\rm none} & {\rm none}  \\
			\hline
			$G_2^{(1)}$ & 
			\begin{minipage}[c][0.06\textheight]{0.5\linewidth}
			\setlength{\unitlength}{0.17in}
			\begin{picture}(16,3.7)
					\put(6,2.1){\makebox(0,0)[c]{$\circ$}}
					\put(8,2.1){\makebox(0,0)[c]{$\circ$}}
					\put(10,2.1){\makebox(0,0)[c]{$\circ$}}
					\put(6.3,2.1){\line(1,0){1.3}}
					\put(8.3,1.9){\line(1,0){1.3}}
					\put(8.3,2.1){\line(1,0){1.3}}
					\put(8.3,2.3){\line(1,0){1.3}}
					\put(8.5,1.84){\LARGE $\succ$}
					\put(6,1.3){\makebox(0,0)[c]{\tiny ${\alpha}_0$}}
					\put(8,1.3){\makebox(0,0)[c]{\tiny ${\alpha}_1$}}
					\put(10,1.3){\makebox(0,0)[c]{\tiny ${\alpha}_2$}}
					%
					% Kac labels
                    \put(6,2.7){\makebox(0,0)[c]{\tiny $1$}}
					\put(8,2.7){\makebox(0,0)[c]{\tiny $2$}}
					\put(10,2.7){\makebox(0,0)[c]{\tiny $3$}}
				\end{picture}
			\end{minipage} & $D_4^{(3)}$ & 
			\begin{minipage}[c][0.06\textheight]{0.5\linewidth}
			\setlength{\unitlength}{0.17in}
			\begin{picture}(16,3.7)
					\put(6,2.1){\makebox(0,0)[c]{$\circ$}}
					\put(8,2.1){\makebox(0,0)[c]{$\circ$}}
					\put(10,2.1){\makebox(0,0)[c]{$\circ$}}
					\put(6.3,2.1){\line(1,0){1.3}}
					\put(8.3,1.9){\line(1,0){1.3}}
					\put(8.3,2.1){\line(1,0){1.3}}
					\put(8.3,2.3){\line(1,0){1.3}}
					\put(8.5,1.84){\LARGE $\prec$}
					\put(6,1.3){\makebox(0,0)[c]{\tiny ${\alpha}_0$}}
					\put(8,1.3){\makebox(0,0)[c]{\tiny ${\alpha}_1$}}
					\put(10,1.3){\makebox(0,0)[c]{\tiny ${\alpha}_2$}}
					%
					% Kac labels
                    \put(6,2.7){\makebox(0,0)[c]{\tiny $1$}}
					\put(8,2.7){\makebox(0,0)[c]{\tiny $2$}}
					\put(10,2.7){\makebox(0,0)[c]{\tiny $1$}}
				\end{picture}
			\end{minipage}
			\\ % {\rm none} & {\rm none}  \\
			\hline
		\end{tabular}
	}
	%\caption{Dynkin diagrams of affine types $X_N^{(r)}$ with Kac labels}
	\label{tab:dynkins}
\end{table}

\providecommand{\bysame}{\leavevmode\hbox to3em{\hrulefill}\thinspace}
\providecommand{\MR}{\relax\ifhmode\unskip\space\fi MR }
\providecommand{\MRhref}[2]{%
  \href{http://www.ams.org/mathscinet-getitem?mr=#1}{#2}
}
\providecommand{\href}[2]{#2}

\end{document}